\documentclass{amsart}
\usepackage{tkz-graph}

\usepackage{amsmath,amsthm,amsfonts,amssymb,times,latexsym,enumerate,comment,marginnote,slashed,dsfont}
\usepackage{pdfpages, relsize, setspace}
\usepackage{xypic}
\usepackage[OT2,T1]{fontenc} 
\usepackage{listings}
\usepackage{tabu}
\usepackage{stmaryrd}
\usepackage{tikz-cd}
\usepackage{bbm}

\newcommand{\mbz}{\mathbb{Z}}

\newcommand{\mbp}{\mathbb{P}}

\renewcommand{\:}{\colon}

\newcommand{\ra}{\rightarrow}
\newcommand{\iso}{\cong}
\newcommand{\bs}{\backslash}

\newcommand{\der}{\partial}

\newcommand{\wtilde}{\widetilde}

\DeclareMathOperator{\Div}{Div}

\DeclareMathOperator{\Aut}{Aut}

\DeclareMathOperator{\supp}{supp}
\DeclareMathOperator{\Gal}{Gal}

\DeclareMathOperator{\im}{Im}

\DeclareMathOperator{\Gr}{Gr}
\DeclareMathOperator{\adj}{adj}
\DeclareMathOperator{\sm}{sm}
\DeclareMathOperator{\Jac}{Jac}
\DeclareMathOperator{\Prym}{Prym}

\DeclareMathOperator{\rd}{red}

\renewcommand{\Im}{\im}

\DeclareMathOperator{\rank}{rk}
\DeclareMathOperator{\corank}{corank}

\DeclareMathOperator{\Spec}{Spec}

\DeclareMathOperator{\id}{id}

\newcommand*{\longhookrightarrow}{\ensuremath{\lhook\joinrel\relbar\joinrel\rightarrow}}
\newcommand*{\inj}{\longhookrightarrow}

\newcommand{\xym}[1]{\xymatrix{#1}}

\newcommand{\mapdef}[5]{
	\begin{tabu}{cccc}
		#1\: & #2 & \ra & #3 \\
		& #4 & \mapsto & #5
	\end{tabu}
}

\newcommand{\gen}[1]{\langle #1 \rangle}

\newcommand{\res}[2]{\left . #1 \right |_{#2}}

\newcommand{\Sp}{\operatorname{Sp}}

\newcommand{\al}{\mathrm{al}}

\mathchardef\mhyphen="2D

\DeclareMathOperator{\Sym}{Sym}

\newcommand{\node}{\mathrm{node}}

\newcommand{\arhook}{\ar@{^{(}->}}

\xyoption{rotate}

\DeclareMathOperator{\GL}{GL}
\DeclareMathOperator{\SL}{SL}
\DeclareMathOperator{\PGL}{PGL}

\DeclareMathOperator{\trace}{Tr}

\newtheorem{theorem}{Theorem}[section]
\newtheorem{lemma}[theorem]{Lemma}
\newtheorem{proposition}[theorem]{Proposition}
\newtheorem*{theorem*}{Theorem}
\newtheorem{corollary}[theorem]{Corollary}

\makeatletter
\newtheorem*{rep@theorem}{\rep@title}
\newcommand{\newreptheorem}[2]{%
	\newenvironment{rep#1}[1]{%
		\def\rep@title{#2 \ref{##1}}%
		\begin{rep@theorem}}%
		{\end{rep@theorem}}}
\makeatother

\newreptheorem{theorem}{Theorem}
\newreptheorem{conjecture}{Conjecture}

\theoremstyle{definition}
\newtheorem{definition}[theorem]{Definition}
\newtheorem{definition-theorem}[theorem]{Definition-Theorem}
\newtheorem{definition-corollary}[theorem]{Definition-Corollary}

\newtheorem{remark}[theorem]{Remark}

\newtheorem{example}[theorem]{Example}

\usepackage{thmtools} 
\usepackage{thm-restate}
\newtheorem{construction}[theorem]{Construction}

\usepackage[alphabetic]{amsrefs}
\usepackage{mathtools}
\numberwithin{theorem}{subsection}

\usepackage{hyperref}
\hypersetup{
    colorlinks=true,
    linkcolor=blue,
    citecolor=blue,
    filecolor=magenta,      
    urlcolor=cyan,
    pdftitle={On intersections of symmetric determinantal varieties and theta characteristics of canonical curves},
    pdfpagemode=FullScreen,
    }
    
\usepackage{url}            % simple URL typesetting
\usepackage{booktabs}       % professional-quality tables

\usepackage[top=3cm,bottom=2cm,left=3cm,right=3cm,marginparwidth=1.75cm]{geometry}

\usepackage[colorinlistoftodos]{todonotes}
\newcommand{\comm}{\todo[inline, color=blue!40]}

\usepackage{tikz}
\usetikzlibrary{matrix,arrows}

\renewcommand{\res}{\operatorname{res}}

\renewcommand{\H}{\operatorname{H}}
\newcommand{\R}{\mathrm{R}}
\newcommand{\Adj}{\mathcal{A}{dj}}
\newcommand{\sheafhom}{\mathcal{H}om}
\newcommand{\sheafext}{\mathcal{E}xt}

\newcommand{\dual}{\widehat}
\newcommand{\lindex}{l}

\newcommand{\odd}{\mathrm{odd}}
\newcommand{\nonessential}{\mathrm{n.e.}}

\DeclareMathOperator{\Proj}{Proj}

\DeclareMathOperator{\sing}{sing}
\DeclareMathOperator{\DAut}{DAut}

\title{On intersections of symmetric determinantal varieties and theta characteristics of canonical curves}

\author{Avinash Kulkarni}
\address{Department of Mathematics, Dartmouth College}
\email{avinash.a.kulkarni@dartmouth.edu}

\author{Sameera Vemulapalli}
\address{Department of Mathematics, Princeton University}
\email{sameerav@math.princeton.edu}


\begin{document}

\maketitle

\begin{abstract}
	From a block-diagonal $(n+1) \times (m+1) \times (m+1)$ tensor symmetric in the last two entries one obtains two varieties: an intersection of symmetric determinantal hypersurfaces $X$ in $n$-dimensional projective space, and an intersection of quadrics $\mathfrak{C}$ in $m$-dimensional projective space. Under mild technical assumptions, we characterize the accidental singularities of $X$ in terms of $\mathfrak{C}$. We apply our methods to algebraic curves and show how to construct theta characteristics of certain canonical curves of genera 3, 4, and 5, generalizing a classical construction of Cayley. 
\end{abstract}

\section{Introduction} \label{sec: introduction}

Let $k$ be a field of characteristic not $2$ or $3$. A tensor $\mathcal{A} \in k^{n+1} \otimes \Sym_2 k^{m+1}$ defines two important varieties. First, write $\mathcal{A}$ as an $(n+1)$-tuple of symmetric $(m+1)\times (m+1)$ matrices $(A_0, \dots, A_{n})$ with $A_i \in \Sym_2 k^{m+1}$ and define the determinantal variety
\[
	X \coloneqq Z(\det(x_0A_0 + \dots + x_n A_n)) \subseteq \mbp^n := \Proj k[x_0, \ldots, x_n].
\]
We call a hypersurface with such a symmetric determinantal representation a \emph{symmetroid}. The second variety we associate to $\mathcal{A}$ is the intersection of $n+1$ quadrics
\[
	\mathfrak{C} := Z(\mathbf{y}^T A_0 \mathbf{y}, \ldots, \mathbf{y}^T A_n \mathbf{y}) \subseteq \mbp^m := \Proj k[y_0, \ldots, y_m],
\] 
where $\mathbf{y} := (y_0, \ldots, y_m)^T$. 

More generally, a tensor $\mathcal{A} \in k^{n+1} \otimes (\bigoplus_{\lindex=1}^r \Sym_2 k^{d_\lindex})$ gives rise to varieties $X$ and  $\mathfrak{C}$ as follows. First, $\mathcal{A}$ can be thought of as a tuple of $r$ tensors $(\mathcal{A}^{(1)}, \ldots, \mathcal{A}^{(r)})$, with each $\mathcal{A}^{(\lindex)} \in k^{n+1} \otimes \Sym_2 k^{d_\lindex}$. Each $\mathcal{A}^{(\lindex)}$ defines a symmetroid $X_{\lindex}$ as above, and we define $X := X_1 \cap \ldots \cap X_r$. Second, we may consider $\mathcal{A}$ as an element of $k^{n+1} \otimes \Sym_2 k^{d_1 + d_2 + \ldots + d_r}$ and construct $\mathfrak{C}$ as an intersection of $n+1$ quadrics in $\mbp^{d_1 + d_2 + \ldots + d_r - 1}$ as before. We call $\mathfrak{C}$ the \emph{Cayley variety of $\mathcal{A}$}. This article studies the connection between $\mathcal{A}$, $X$, and $\mathfrak{C}$. 

This work has two major inspirations. First is the classical relationship between plane quartic curves and intersections of $3$ quadrics in $\mbp^3$ (i.e. $\mathcal{A} \in k^3 \otimes \Sym_2 k^4$), originally studied by Hesse~\cites{HesseI, HesseII} and Dixon~\cite{Dixon}; this case serves as a motivating example throughout the paper. The eight points in $\mbp^3$ are classically known as a \emph{Cayley octad}, and the $28$ secants between these points are in one-to-one correspondence with the bitangents of a plane quartic curve. This relation has been further studied by \cites{ElsenhansJahnel2019, Plaumann2011}. Secondly, orbit parametrizations of geometric objects often have applications to topics of interest in arithmetic geometry; see \cites{BhargavaHCL1, BhargavaHCL2, BhargavaHCL3, BhargavaHCL4} for low rank rings and certain ideal classes, \cite{Wood2014} for binary n-ic rings and certain ideal classes, \cite{BhargavaHoKumar2016} and \cite{BhargavaHoKumar2016} for K3 surfaces, \cite{BhargavaHo2016} for genus 1 curves, and \cite{BhargavaGrossWang2017} for hyperelliptic curves. 

The orbits of $k^{n+1} \otimes (\bigoplus_{\lindex=1}^r \Sym_2 k^{d_{\lindex}})$ under the group $\GL_{n+1} \times (\bigoplus_{\lindex=1}^r \GL_{d_{\lindex}})$ simultaneously classifies two families of geometric objects: intersections of symmetroids in $\mbp^n$ with prescribed line bundles and complete intersections of quadrics in $\mbp^m$. It is natural to ask what the general relationship between these objects is.

\subsection{Statement of results}

A \emph{block-diagonal tensor with $r$ blocks} is a tensor $\mathcal{A} \in k^{n+1} \otimes (\bigoplus_{\lindex = 1}^r \Sym_2 k^{d_\lindex})$. In this article, we assume that the determinantal variety $X_{\lindex}$ associated to each block of $\mathcal{A}(\mathbf{x}, \cdot, \cdot)$ is reduced and irreducible, we have $d_\lindex \geq 2$, and no block is uniformly zero. See Section~\ref{sec: sub: block-diagonal} for more details.

Our first result describes the relationship between accidental singularities of $X$ and singularities of $\mathfrak{C}$. For a hypersurface $X$, i.e. $r = 1$, an \emph{essential singularity} of $X = \det \mathcal{A}(\mathbf{x}, \cdot, \cdot)$ is a point $p$ such that $\corank \mathcal{A}(p, \cdot, \cdot) \geq 2$. Generically, the only singularities of a determinantal hypersurface are its essential singularities \cite{Kerner2012}. It is possible for a determinantal hypersurface to have singularities that are not essential; these are called \emph{accidental singularities}. 

Similarly, an \emph{essential singularity} of an intersection of symmetroids $X$ defined by a block-diagonal tensor $\mathcal{A}$ is a point $p \in X(\bar k)$ such that $\corank \mathcal{A}^{(\lindex)}(p, \cdot, \cdot) \geq 2$ for some block of $\mathcal{A}$, and an \emph{accidental singularity} is a singular point of $X$ that is not an essential singularity. A \emph{coincident singularity} of $X$ is an accidental singularity $x \in X$ which is a singular point of two intersections of symmetroids $X', X'' \supset X$ defined by two distinct strict subsets of the blocks of $\mathcal{A}$. For further details, see Section~\ref{sec: general results}. We denote by $X^\nonessential$ the complement of the essential singularities in $X$. 

\begin{restatable}{theorem}{accidentalsingularitytheorem} 
\label{thm: singularities of symmetroid intersections with multiplicity} 
	Let $\mathcal{A} \in k^{n+1} \otimes (\bigoplus_{\lindex =1}^r \Sym_2 k^{d_\lindex})$ be a block-diagonal tensor with $r$ blocks, let $X$ be the associated intersection of symmetroids, and let $\mathfrak{C}$ be the Cayley variety of $\mathcal{A}$. If $X$ is a complete intersection, $\sing X^\nonessential$ has dimension $0$, and $X$ does not have any coincident singularities, then the scheme
	\[
		\{(x,y) \in X^{\nonessential} \times \mbp^m : \mathcal{A}(x, y, \cdot) = 0,  \ y \in \mathfrak{C}\}
	\]
	is a finite cover of $\sing X^{\mathrm{n.e.}}$ of degree $2^{r-1}$ via projection on the first factor.
\end{restatable}

We apply Theorem~\ref{thm: singularities of symmetroid intersections with multiplicity} to obtain a generalization of the relation between the bitangents of a plane quartic curve and the secant lines of the Cayley octad. The tensor $\mathcal{A}$ defines a map $\psi\: \mbp^m \dashrightarrow \dual \mbp^n$ by
\[
	\psi(\mathbf{y}) := (\mathbf{y}^T A_0 \mathbf{y}, \ldots, \mathbf{y}^T A_n \mathbf{y}),
\]
with base locus $\mathfrak{C}$. The map $\psi$ contracts any secant line of $\mathfrak{C}$ to a point $[H] \in \dual \mbp^n$. When $\mathfrak{C}$ is zero dimensional, $H$ intersects $X$ in a special way. Let the \emph{block spaces} of $\mathcal{A}$ be the linear subspaces of $\mbp^{d_1 + \dots + d_r - 1}$ defined by the natural inclusions $k^{d_\lindex} \inj k^{d_1 + \dots + d_r}$ corresponding to the decomposition $\bigoplus_{\lindex = 1}^r \Sym_2 k^{d_\lindex}$.

\begin{restatable}{proposition}{secantsandhyperplanestheorem}  
\label{prop: deg and dim of tangents of intersections of symmetric determinantal hypersurfaces}
Let $\mathcal{A} \in k^{n+1} \otimes (\bigoplus_{\lindex=1}^r \Sym_2 k^{d_\lindex})$ be a block-diagonal tensor whose associated Cayley variety $\mathfrak{C}$ is $0$-dimensional and let $X$ be the associated complete intersection of symmetroids. Let $\ell$ be a secant line of $\mathfrak{C}$ which does not intersect the block spaces and let $[H] = \psi(\ell)$. If $X\cap H$ contains no essential singularities or coincident singularities of $X$, then $\sing (X \cap H)$ contains a subscheme of dimension $0$ and degree $n$.
\end{restatable}

We apply Proposition~\ref{prop: deg and dim of tangents of intersections of symmetric determinantal hypersurfaces} to four main cases: smooth plane quartics, generic quintic symmetroid surfaces, smooth canonical genus $4$ curves whose canonical embedding is contained in a quadric cone, and smooth canonical genus $5$ curves whose canonical embedding can be written as the intersection of three quadrics of rank $3$. The case of smooth plane quartics is well-known, and the case of generic quintic symmetroid surfaces was studied by Roth in $1930$ \cite{Roth1930}. The latter two cases are novel to the knowledge of the authors. In the case of curves of genus $4$ and $5$, we describe the moduli functor of genus $4$ curves with a vanishing even theta characteristic (respectively, genus $5$ curves with three vanishing even theta characteristics) using the symmetric determinantal representation. We also describe the relationship between the determinantal representations to certain odd theta characteristics. If $\mathcal{A} \in k^4 \otimes (\Sym_2 k^2 \oplus \Sym_2 k^3)$ or $\mathcal{A} \in k^5 \otimes (\Sym_2 k^2)^{\oplus 3}$, we say that $\mathcal{A}$ is \emph{nondegenerate} if the associated intersection of symmetroids is a smooth curve (of genus $4$ or $5$, respectively). As nondegeneracy is preserved by the natural group action, we may speak of nondegenerate orbits.

\begin{restatable}{theorem}{genusfoursummary}
	\label{thm: genusfoursummary}
	\begin{enumerate}[(a)]
	\item
	Let $B$ be a scheme over $\Spec \mbz[\frac{1}{6}]$. There is a canonical bijection between:
	\[
		\left \{  \quad \parbox{7cm}{\centering
		isomorphism classes of tuples $(X, \epsilon, \theta_0)$,\\
		where $X \rightarrow B$ is a smooth genus $4$ curve over $B$ with vanishing even theta
		characteristic $\theta_0$ with a rational divisor class defined over $B$,
		and $\epsilon$ is an even $2$-torsion class}
		 \quad  \right\} 	 
		 \longleftrightarrow 
		 \left \{ \begin{array}{c}
		 \text{nondegenerate orbit classes of } \\
		 \mathcal{O}_B^4 \otimes \left(\Sym_2 \mathcal{O}_B^2 \oplus \Sym_2 \mathcal{O}_B^3 \right) \\
		 \text{under the action of } \\
		 \GL_4(\mathcal{O}_B) \times \GL_2(\mathcal{O}_B) \times \GL_3(\mathcal{O}_B) 
		 \end{array} \right\}
	\]
	\end{enumerate}
	\noindent
	Let $\mathcal{A} \in k^4 \otimes (\Sym_2 k^2 \oplus \Sym_2 k^3)$ be a nondegenerate tensor and let $\theta_0$ and $\epsilon$ be the associated line bundles on $X$. Then:
	\begin{enumerate}[(b)]
	\item
	The images of the $120$ secants of $\mathfrak{C}$ under $\psi$ define $56 + 8$ tritangent planes of $X$. Viewing $X \cap H$ as a divisor of $X$, eight of these tritangents satisfy $X \cap H = 2D$ where $D \in |\theta_0|$. The other $56$ tritangents satisfy $X \cap H = 2D$, where $D'$ is the effective representative of an odd theta characteristic of $X$.
	
	\item
	Let $e_2$ denote the Weil pairing on $\Jac(X)[2]$. Then the $56$ distinct odd theta characteristics constructed from the secants of $\mathfrak{C}$ are precisely the odd theta characteristics $\theta$ such that $e_2(\theta \otimes \theta_0^\vee, \epsilon) = 0$.
	\end{enumerate}

\end{restatable}

\begin{restatable}{theorem}{genusfivesummary}
	\label{thm: genusfivesummary}
	\begin{enumerate}[(a)]
		\item 
			Let $B$ be a scheme over $\Spec \mbz[\frac{1}{2}]$. There is a canonical bijection between:
			\[
				\left \{ \begin{array}{c}
				\text{isomorphism classes of ordered quadruples} \\
				(X, \theta_1, \theta_2, \theta_3) 
				\text{ where } X \rightarrow B \text{ is a smooth}\\
				\text{genus $5$ curve over $B$, and } \theta_1, \theta_2, \theta_3 \text{ are} \\
				\text{distinct vanishing even theta characteristics} 
				
			\end{array}	 \right \}
			\longleftrightarrow
			\left \{ \begin{array}{c}
				\text{nondegenerate orbit classes of} \\
				\mathcal{O}_B^5 \otimes 
				(\Sym_2 \mathcal{O}_B^2 \oplus \Sym_2 \mathcal{O}_B^2 \oplus \Sym_2 \mathcal{O}_B^2) \\
				\text{under the action of } 
				\GL_5(\mathcal{O}_B) \times \GL_2(\mathcal{O}_B)^{\oplus 3}
				\end{array}
			\right \}
			\]	
	\end{enumerate}
	\noindent
	Let $\mathcal{A} \in k^5 \otimes (\Sym_2 k^2)^{\oplus 3}$ be a nondegenerate tensor and let $\theta_1, \theta_2, \theta_3$ denote the three associated vanishing even theta characteristics. Then:
	\begin{enumerate}[(b)]
		\item
			The $496$ secants of the Cayley variety define $112 + 3 \times 8$ tetratangent planes of $X$. Viewing $X \cap H$ as a divisor of $X$, eight of these tetratangents satisfy $X \cap H = 2D$ where $D \in |\theta_\lindex|$ for each $\lindex = 1,2,3$. The other $112$ tetratangents satisfy $X \cap H = 2D'$, where $D'$ is the effective representative of an odd theta characteristic of $X$.
		\item
			Let $e_2$ denote the Weil pairing on $\Jac(X)[2]$. Then the $112$ distinct odd theta characteristics constructed from the secants of $\mathfrak{C}$ are precisely the odd theta characteristics $\theta$ such that $e_2(\theta \otimes \theta_\lindex^{\vee}, T) = 0$ for all $T \in \{\theta_i \otimes \theta_j^\vee : 1 \leq i,j \leq 3\}$ and for some (equivalently, all) $1 \leq \lindex \leq 3$.
	\end{enumerate}
\end{restatable}

Over a field $k$, part (a) of each theorem is classical; for instance, the description of determinantal cubic surfaces was well known to Coble and Wirtinger \cites{Coble1919, DolgachevOrtland1988}. In order to prove the results for moduli stacks, we use a technique of \cite{Ho2009, BhargavaHoKumar2016}. Namely, in these cases we construct a resolution of an arithmetically Cohen-Macaulay sheaf on $\mbp^n$ whose specialization to $X$ defines a resolution of the starting line bundle.

\subsection{Additional results for canonical curves of genus $4$}

A construction of Milne, later studied by Bruin and Sert\"oz, relates the bitangents of a genus $3$ curve to the tritangents of a genus $4$ curve \cites{Milne1923, BruinSertoz2018}. We provide a summary of their construction in Section~\ref{sec: sub: BruinSertoz} in the case of canonical curves of genus $4$ with a vanishing even theta characteristic (see Theorem~\ref{thm: bruin sertoz}). In this case, our tensor descriptions simplify this construction significantly.

%\begin{theorem}[Bruin-Sert\"oz]
%\label{thm: bruin sertoz}
	%Let $X$ be a genus $4$ curve with a vanishing theta characteristic $\theta_0$ and let $\epsilon \in \Jac(X)$ be a $2$-torsion point not of the form $[\theta-\theta_0]$ for some odd theta characteristic $\theta$ (in their terminology, $\epsilon$ is an \emph{even} $2$-torsion point). Let $X_3$ be the cubic symmetroid associated to $\epsilon$ and let $\varphi\: X_3 \dashrightarrow \mbp^2$ be the kernel map. 
	
	%Then the genus $3$ curve associated to $(C, \epsilon)$ is the double cover of a conic in $\mbp^2$ ramified at $8$ marked points. Furthermore, each of the $28$ lines through these $8$ points in $\mbp^2$ corresponds to a pair of odd theta characteristics of $X_{2,3}$. If $D_1$ and $D_2$ are the two effective representatives of the odd theta characteristics associated to a line $L$, then $\varphi(D_1), \varphi(D_2) \subset L$.
%\end{theorem}

Choose a general $\mathcal{A} \in k^4 \otimes (\Sym_2 k^2 \oplus \Sym_2 k^3)$ and let $X$ be the associated genus $4$ curve.  Let $\pi \colon \mbp^4 \dashrightarrow \mbp^2$ be the projection onto the last $3$ coordinates. Then the image of the $112$ secant lines of $\mathfrak{C}$ is a set of $28$ lines in $\mbp^2$. These $28$ lines are precisely the secants through the eight points of $\pi(\mathfrak{C})$ (see Section~\ref{sec: sub: BruinSertoz}).
The tensor $\mathcal{A}$ defines a natural quadruple cover
		\[
			\left \{ \begin{array}{c}
				\text{The $112$ secant lines } \ell \text{ of } \mathfrak{C} \text{ defining} \\
				\text{ odd theta characteristics of } X 			\end{array}
			\right \} \rightarrow \left \{ \begin{array}{c}
			\text{The $28$ lines given by } \pi(\ell) \\
			\text{for } \ell \text{ a secant line of } \mathfrak{C}
			\end{array}\right \}
		\]
which factors through a double cover
		\[
			\left \{ \begin{array}{c}
				\text{The $112$ secant lines } \ell \text{ of } \mathfrak{C} \text{ defining} \\
				\text{ odd theta characteristics of } X 			\end{array}
			\right \} \rightarrow \left \{ \begin{array}{c}
				\text{The $56$ odd theta characteristics represented} \\
				\text{by the image of the secant lines of } \mathfrak{C} \text{ under } \psi
			\end{array} \right\}.
		\]
In this way, a line in $\mbp^2$ corresponds to a pair of odd theta characteristics of $X$.

\begin{restatable}{theorem}{tritangentbitangenttheorem}
\label{thm: tritangent bitangent theorem}
Choose a general $\mathcal{A} \in k^4 \otimes (\Sym_2 k^2 \oplus \Sym_2 k^3)$ and let $X$ be the associated genus $4$ curve.  
The correspondence given above is exactly the correspondence of \cite{BruinSertoz2018} in the ``vanishing theta-null, $\epsilon$ even'' case (from \cite[Table~1]{BruinSertoz2018}).	
\end{restatable}

%We discuss how a construction of Recillas can be realized via simple operations in Section~\ref{sec: sub: Recillas}

% Theorem: The Resolvent curve is constructed easily.

\subsection{Historical notes}
Hesse studied the classical relationship between plane quartic curves and intersections of $3$ quadrics in $\mbp^3$. Hesse's result has since been generalized in two different directions. The first is the work of Roth \cite{Roth1930}, who studied tritangent hyperplanes to a quintic symmetroid in terms of an intersection of four quadrics in $\mbp^5$. Secondly, the relationship between a net of quadrics in $\mbp^n$ and its degeneracy locus in $\mbp^2$ has been studied by Reid and Tyurin \cite{Reid1972quadrics, Tyurin1975}. Reid proved that when $n$ is even, the intermediate Jacobian of the intersection of three quadrics is the Prym variety of a double cover of the determinantal curve in $\mbp^2$ defined by the determinantal representation. In both cases, the determinantal variety is necessarily a hypersurface. Determinantal varieties themselves have been well studied since the 19th century \cite{Dol2012}.

\subsection{Outline}
In Section~\ref{sec: setup}, we discuss notation and background. Section~\ref{sec: plane quartics} is a brief exposition on the classical results of plane quartic curves, meant as an illustrative example of the arguments we apply in general case. In Section~\ref{sec: general results}, we prove Theorem~\ref{thm: singularities of symmetroid intersections with multiplicity}. In Section~\ref{sec: fibres}, we focus on the case where the Cayley variety is zero dimensional and we study the hyperplane sections defined by its secants. In particular, we prove Proposition~\ref{prop: deg and dim of tangents of intersections of symmetric determinantal hypersurfaces}. In Section~\ref{sec: genus 4 curves}, we study canonical curves of genus $4$ defined by an intersection of symmetroids; in particular we prove Theorem~\ref{thm: genusfoursummary} and Theorem~\ref{thm: tritangent bitangent theorem}. In Section~\ref{sec: genus 5 curves}, we study the canonical curves of genus $5$ defined by an intersection of symmetroids; we prove Theorem~\ref{thm: genusfivesummary}. Finally, in Section~\ref{sec: quintic symmetroid} we discuss the tritangent planes to a quintic symmetroid surface in $\mbp^3$.

\subsection*{Acknowledgments}

	We would like to thank Mario Kummer for helpful discussions regarding Section~2.2. This project originated in the Nonlinear Algebra Group at MPI Leipzig; we would like to thank Bernd Sturmfels for introducing the two authors. We would also like to thank Emre Sert\"oz for comments on the manuscript. Avinash Kulkarni has been supported by the Simons Collaboration on Arithmetic Geometry, Number Theory, and Computation (Simons Foundation grant 550033). Sameera Vemulapalli has been supported by an NSF Graduate Fellowship.

\begin{figure}
\begingroup
\renewcommand{\arraystretch}{1.5}
\begin{tabular}{lp{0.7\textwidth}}
	Symbol & Description \\ \hline
	$\mathcal{A}$ & denotes an $(n+1) \times (m+1) \times (m+1)$ symmetric tensor. \\ %Equivalently, $\mathcal{A}$ can be considered as a symmetric $(m+1)\times (m+1)$ matrix of homogeneous linear forms in $n+1$ variables. \\ 
	
	$\mathcal{A}^{(\lindex)}$ & denotes the $\lindex$--th block of a tensor $\mathcal{A} \in k^{n+1} \otimes (\bigoplus_{\lindex=1}^r \Sym_2 k^{d_\lindex})$ \\
	
	$A_0,\dots,A_n$ & denotes the $n+1$ slices of the tensor. \\ 

	$\mathfrak{C}$ & denotes the Cayley variety in $\mbp^m$, i.e, the variety defined by $\{y \in \mbp^m : \mathcal{A}(\cdot, y, y) = 0\}$. \\
	
	%$a_{i,j,k}$ & denotes the element in the $j$-th row and the $k$-th column of $A_i$. \\
	
	%$\mathcal{M}_{j,k}$ & denotes the $j,k$ minor of $\mathcal{A}$ considered as an $(m+1)\times (m+1)$ matrix of homogeneous linear forms in $n+1$ variables. \\
		
	%$X_{(d_1, \ldots, d_r)}$ & denotes an intersection of symmetroids in $\mbp^n$ of degrees $(d_1, \ldots, d_r)$ for $d_1,\dots,d_r \geq 2$. Of course, particularly special intersections of symmetroids may be given their own symbol. \\
	
	$X$ & denotes an intersection of symmetroids. \\
	
	$X_{d_i}$ with $1\leq i \leq r$ & denotes a symmetroid of degree $d_i$ such that $X = \cap_i X_{d_i}$. \\
	
	%$\wtilde X_{?}$ & a double covering of $X_{?}$. \\

	$Q_x$ for $x \in \mbp^n$ & denotes the quadric in $\mbp^m$ defined by $Z(\mathcal{A}(x, \mathbf{y}, \mathbf{y}))$. \\

	%$q_x$ for $x \in \mathbb{A}^{n+1}$ & the quadratic form corresponding to the symmetric matrix $\mathcal{A}(x, \cdot, \cdot)$. For $x = (1, 0, \dots, 0),\dots,(0 , \dots, 0, 1) \in \mathbb{A}^{n+1}$, we denote $q_x$ by $q_0,\dots,q_n$ respectively. \\
	
	%$Y$ & generally denotes some variety in $\mbp^m$, usually defined as an intersection of quadrics. \\

	$H$ & hyperplane in $\mbp^n$ \\

	$[H]$ & the point in $\dual \mbp^n$ corresponding to the hyperplane $H$ \\
	
	$\ell(p,q)$ & the line between points $p$ and $q$ (in some projective space). \\ \hline
\end{tabular}
\bigskip
\caption{Table of notation} \label{fig: notation table}
\endgroup
\end{figure}

\begin{figure} 
	\begingroup
	\renewcommand{\arraystretch}{1.5}
	\begin{tabular}{lcl}
		Symbol & & Description \\ \hline
		
		$\psi$ & $\colon$ & $\mbp^m \dashrightarrow \dual \mbp^n$. \\
		$\psi_{X_{d_i}} $ & $\colon$ & $ \mbp^{d_i-1} \dashrightarrow \dual \mbp^n$. \\
		%$\psi_{d_1,\dots,d_r} $ & $\colon$ & $ \prod_{i = 1}^k \mbp^{d_i-1} \rightarrow \Gr(n+1-r, n+1)$. \\
		%$\varphi $ & $\colon$ & $ \cup_{i=1}^kX_{d_i} \dashrightarrow \mbp^m$. \\
		$\varphi_{X_{d_i}} $ & $\colon$ & $ X_{d_i} \dashrightarrow \mbp^{d_i-1}$ \\
		$\iota_{d_i} $ & $\colon$ & $ \mbp^{d_i-1} \dashrightarrow \mbp^{m}$ \\
		$\varphi_{d_1,\dots,d_r} $ & $\colon$ & $ X \rightarrow \prod_{i=1}^k \mbp^{d_i-1}$ \\
		$\theta_{X_{d_i}} $ & $\colon$ & $ X_{d_i} \dashrightarrow \dual \mbp^n$ the Gauss map. \\
		%$\theta_X$ & $\colon $ & $X_{d_1}\cup \dots \cup X_{d_r} \dashrightarrow \dual \mbp^n$ \\
		%$\theta_{X,\Gr} $ & $\colon$ & $ X \dashrightarrow \Gr(n+1-r, n+1)$
	\end{tabular}
	\bigskip
	\caption{Table of notation for maps} \label{fig: notation table II}
	\endgroup
\end{figure}

\section{General notation and setup} \label{sec: setup}

We denote by $k$ an arbitrary field of characteristic not $2$ or $3$. Let $\bar k$ denote an algebraic closure of $k$. A summary of the main notation for this article is given in Figures~\ref{fig: notation table} and~\ref{fig: notation table II}. 

\bigskip
Let $\mathcal{A} \in k^{n+1} \otimes \Sym_2 k^{m+1}$ be a tensor, symmetric in the last two entries. We view $\mathcal{A}$ as an $(n+1) \times (m+1) \times (m+1)$ array of elements of $k$. We denote the contraction of $\mathcal{A}$ along a vector $v \in k^{n+1}$ by $\mathcal{A}(v, \cdot, \cdot)$. More generally, we will contract along an element of $R^n$, where $R$ is a $k$-algebra. Similarly, we will denote the contraction by an element $y \in R^{m+1}$ by $\mathcal{A}(\cdot, y, \cdot)$ or $\mathcal{A}(\cdot, \cdot, y)$, depending along which axis we contract. We may always view a tensor $\mathcal{A} \in k^{n+1} \otimes \Sym_2 k^{m+1}$ as a multilinear map $\mathcal{A}\: (k^{n+1})^\vee \otimes (k^{m+1})^\vee \otimes (k^{m+1})^\vee \rightarrow k$ symmetric in the last two entries, and vice-versa. The contraction by an element is simply evaluating this multilinear map in the appropriate entry. We refer to a \emph{slice} of the tensor $\mathcal{A}$ as the contraction along a standard basis vector; let $A_i$ denote the $i$--th slice of $\mathcal{A}$.
%\bigskip
Denote $\mbp^{n} := \Proj k[x_0,\ldots,x_n]$ and $\mbp^m := \Proj k[y_0, \ldots, y_m]$. We will also denote $\mathbf{x} := (x_0, \ldots, x_n)$ and $\mathbf{y} := (y_0, \ldots, y_m)$. The contractions $\mathcal{A}(\mathbf{x}, \cdot, \cdot), \mathcal{A}(\cdot, \mathbf{y}, \cdot)$ are matrices of linear forms. If $x \in \mbp^n(k)$ is a point, then the contraction
	$
		\mathcal{A}(x, \cdot, \cdot)
	$
is a symmetric $(m+1) \times (m+1)$ matrix with entries in $k$, well-defined up to scaling. In other words, one obtains the defining equation for a quadric in $\mbp^m$. We express this as a rational map
	\[
		\begin{tabu}{cccc}
			\dual \psi\: & \mbp^n &\dashrightarrow & \mbp\!\H^0(\mbp^m, \mathcal{O}_{\mbp^m}(2)) \\
			& x & \mapsto & \mathcal{A}(x, \cdot, \cdot)
		\end{tabu}.
	\] 
The map $\dual \psi$ is the dual of the map $\psi$ which we now define. The tensor $\mathcal{A}$ consists of $n+1$ slices, which are symmetric $(m+1) \times (m+1)$ matrices. Write $(q_0, \ldots, q_n)$ for the quadrics defined by these $n+1$ slices. The $q_i$ define the rational map
	\[
		\begin{tabu}{cccc}
			\psi\: & \mbp^m &\dashrightarrow & \dual \mbp^n \\
			& y & \mapsto & \mathcal{A}(\cdot, y, y) = (q_0(y), \ldots, q_n(y))
		\end{tabu}.
	\]

\begin{definition}
	The \emph{Cayley variety} $\mathfrak{C}$ associated to $\mathcal{A}$ is the scheme in $\mbp^m$ given by 
	\[
	\{y \in \mbp^m : \mathcal{A}(x,y, y) = 0, \ \text{for all } x \in \mbp^n\}.
	\]
\end{definition}

%In the case $n = 2$ and $m = 3$, we have that $X$ is a plane quartic and $\mathfrak{C}$ is of dimension $0$ and degree $8$. This cluster of $8$ points was originally studied by Cayley and thus this particular Cayley variety is called the \emph{Cayley octad}. 

\bigskip
The \emph{automorphism group} of $\mathcal{A} \in k^{n+1} \otimes \Sym_2 k^{m+1}$ is the subgroup of $\GL_{n+1} \times \GL_{m+1}$ defined by
	\[
	\Aut(\mathcal{A}) := \{ (g,h) \in \GL_{n+1} \times \GL_{m+1} : (g, h) \cdot \mathcal{A} = \mathcal{A}\}.
	\]
On each of the slices $A_j$, we have that $h$ acts by $A_j \mapsto h^T A_j h$. In particular, the image of $\Aut(\mathcal{A})$ in $\PGL_{n+1}$ is a group of linear automorphisms of $\mbp^n$ leaving $Z(\det \mathcal{A}(\mathbf{x}, \cdot, \cdot))$ invariant. The image of $\Aut(\mathcal{A})$ in $\PGL_{m+1}$ is a subgroup of linear automorphisms of $\mbp^m$ leaving the Cayley variety invariant, which is also the group of linear automorphisms of $\mbp^m$ such that $\psi \circ g = \psi$. If $\mathcal{A}$ is generic and $\mathfrak{C}$ is nonempty, then the image of $\Aut(\mathcal{A})$ in $\PGL_{m+1}$ is precisely the group of linear transformations preserving $\mathfrak{C}$.

\bigskip
We obtain a hypersurface in $\mbp^n$ with a symmetric determinantal representation by $X := Z(\det \mathcal{A}(\mathbf{x}, \cdot, \cdot))$. 
The tensor $\mathcal{A}$ also determines natural rational maps. Here, we discuss these maps in the case that $X$ is a hypersurface. The analogous maps for an intersection of symmetroid hypersurfaces are discussed in Section~\ref{sec: sub: block-diagonal}. The tensor $\mathcal{A}$ defines a rational map by
	\[
		\begin{tabu}{cccc}
			\varphi\: & X &\dashrightarrow & \mbp^m \\
			& x & \mapsto & \ker \mathcal{A}(x, \cdot, \cdot)
		\end{tabu}.
	\]
Note that $\varphi$ is well-defined for all $x \in X$ such that $\corank \mathcal{A}(x, \cdot, \cdot) = 1$. We refer to this map as the \emph{kernel map}. If $A$ is a non-zero symmetric matrix, then the singular locus of the quadric $Z(\mathbf{y}^TA\mathbf{y}) \subseteq \mbp^m$ is precisely $\mbp(\ker A)$. In other words, when $x$ is a smooth point of $X$, we see that the kernel of $\mathcal{A}(x, \cdot, \cdot)$ is 1-dimensional, i.e, the singular locus of the associated quadric hypersurface is a single point; in particular, we may view $\varphi$ as sending a point $x \in X(\bar k)$ to the vertex of the cone $Z(\mathcal{A}(x, \mathbf{y}, \mathbf{y})) \subset \mbp^m$. %In some cases, such as when $X$ is a quintic symmetroid, the image of the kernel map is the normalization of $X$, and in this situation the kernel map is a birational inverse to the normalization map.

For a hypersurface $X$, its dual variety is denoted $\dual X$; it is the subvariety of $\dual \mbp^n$ parameterizing the planes tangent to $X$ at some point. The \emph{Gauss map} $\theta_X$ is the rational map defined by
	\[
		\begin{tabu}{cccc}
			\theta_X\: & X &\dashrightarrow & \dual \mbp^n \\
			& x & \mapsto & T_xX
		\end{tabu}.
	\]
The Gauss map is defined on the smooth subscheme of $X$. The dual variety $\dual X$ is defined to be the Zariski closure of $\theta_X(X)$. The fundamental relationship between $\theta_X$, $\psi$, and $\varphi$ is:

\begin{proposition} \label{prop: main diagram hypersurfaces}
	Let $\mathcal{A} \in k^{n+1} \otimes \Sym_2 k^{m+1}$ be a tensor and let $X$ be the associated symmetroid hypersurface. Then the following diagram commutes:
	\[
	\begin{tikzcd}
		X \arrow[hook]{r}{} \arrow[dashed]{d}[left]{\theta_X} \arrow[swap, dashed]{drr}[above]{\varphi} & \mbp^n \arrow{r}{\dual \psi} & \mbp\H^0(\mbp^m, \mathcal{O}_{\mbp^m}(2))  \\
		\dual X \arrow[hook]{r}{} & \dual \mbp^n & \mbp^m \arrow[dashed]{l}{\psi}
	\end{tikzcd}.	
	\]
	For $x \in \mbp^n(\bar k)$, the quadric $Z(\dual \psi(x))$ is singular if and only if $x \in X(\bar k)$. Furthermore, for any $x \in X(\bar k)$ the quadric $Z(\dual \psi(x)) \subset \mbp^m$ is singular at the point $\varphi(x)$. For any $[H] = z \in \dual \mbp^n(\bar k)$, the fibre $\psi^{-1}(z)$ is the base locus of the linear space of quadrics $\dual \psi_{\mid H} \colon H \rightarrow \mbp\!\H^0(\mbp^m, \mathcal{O}_{\mbp^m}(2))$.
\end{proposition}

\begin{proof}
	Choose $p \in X^{\sm}$. Without loss of generality, we may assume that $k$ is algebraically closed and we may act by $\PGL_{n+1}(k) \times \PGL_{m+1}(k)$ to suppose $p = [1 \colon 0 \colon \dots \colon 0]$ and $\varphi(p) = [1 \colon 0 \colon \dots \colon 0]$. 
	Denote the slices of $\mathcal{A}$ by $A_0,\dots,A_n$. Let $a_{i,j,\lindex}$ denote the element of the $j$-th row and $\lindex$-th column of $A_i$. Let $\mathcal{M}_{j,\lindex}(\mathbf{x})$ denote the $j,\lindex$-th minor of $\mathcal{A}(\mathbf{x}, \cdot, \cdot)$. Let $M_{i,j,\lindex}$ denote the $j,\lindex$-th minor of $A_i$. Then 
	\begin{gather*}
		\psi( \varphi(p)) = \mathcal{A}(\cdot, \varphi(p), \varphi(p)) = [a_{0,1,1} \colon \dots \colon a_{n, 1, 1}] 
		\qquad \text{and} \\[1.5ex]
		\frac{\partial \det(\mathcal{A}(\mathbf{x}, \cdot, \cdot))}{\partial x_i} = \sum_{j,\lindex}a_{i,j,\lindex} \mathcal{M}_{j,\lindex}(\mathbf{x}).
	\end{gather*}
	By our choice of $p$ and $\varphi(p)$ the first column and first row of $A_0$ are $0$. Evaluating at $p$, we see that 
	\[
		\frac{\partial \det(\mathcal{A}(\mathbf{x}, \cdot, \cdot))}{\partial x_i}(p) = \sum_{j,\lindex}a_{i,j,\lindex}\mathcal{M}_{j,\lindex}(p) = a_{i,1,1} M_{0,1,1}.
	\]
	Because $p$ is a smooth point, the matrix $A_0$ must have corank $1$; moreover, the first row and column of $A_0$ are $0$. Therefore $M_{0,1,1} \neq 0$, and thus
	\[
		\psi( \varphi(p)) = \bigg[\frac{\partial \det(\mathcal{A}(\mathbf{x}, \cdot, \cdot))}{\partial x_0}(p) \colon  \dots \colon  \frac{\partial \det(\mathcal{A}(\mathbf{x}, \cdot, \cdot))}{\partial x_n}(p)\bigg ] = \theta_X(p).
	\]
	By the definition of $X$ and $\varphi$, it is clear that for $Z(\dual \psi(p))$ is singular if and only if $p \in X(\bar k)$. Similarly one sees that for $x \in X(\bar k)$, the quadric $Z(\dual \psi(x)) \subset \mbp^m$ is singular at the point $\varphi(x)$. 
	For $z \in \dual\mbp^n$, we compute $\psi^{-1}(z)$. Acting by $\PGL_{n+1}$, we may assume $z = [1 \colon 0 \colon \dots \colon 0]$. Thus, $\psi^{-1}(z) = Z(\mathbf{y}^TA_0\mathbf{y}) \cap \dots \cap Z(\mathbf{y}^TA_n\mathbf{y})$, proving our last assertion.
\end{proof}

%%% SUBSECTION
\subsection{Block-diagonal tensors and intersections of symmetroids} \label{sec: sub: block-diagonal}

Let $d_1, \ldots, d_r \geq 2$ be integers and set $m = (\sum d_{\lindex}) -1$. Observe that $k^{n+1} \otimes (\bigoplus_{\lindex = 1}^r \Sym_2 k^{d_\lindex})$ embeds block-diagonally in $k^{n+1} \otimes \Sym_2 k^{m+1}$ (see Example~\ref{ex: block-diagonal example}). For a tensor $\mathcal{A} \in k^{n+1} \otimes (\bigoplus_{\lindex = 1}^r \Sym_2 k^{d_\lindex})$, we may view $\mathcal{A}\: k^{n+1} \otimes \Sym_2 k^{m+1} \rightarrow k$ as a multilinear map. The restrictions to each $k^{n+1} \otimes \Sym_2 k^{d_\lindex}$ define $r$ multilinear maps, i.e. tensors, $\mathcal{A}^{(1)}, \ldots, \mathcal{A}^{(r)}$; these tensors are called the \emph{blocks} of $\mathcal{A}$. The \emph{size} of the block $\mathcal{A}^{(\lindex)}$ is $d_\lindex$. A \emph{block-diagonal tensor} with $r$ blocks is a tensor $\mathcal{A} \in k^{n+1} \otimes (\bigoplus_{\lindex = 1}^r \Sym_2 k^{d_\lindex})$ for some $d_\lindex \geq 1$ such that each $\det \mathcal{A}^{(\lindex)}(\mathbf{x}, \cdot, \cdot)$ is reduced and irreducible. As a convention, we will always assume that the size of a block is at least $2$ and we will never consider a tensor where one of the blocks is identically zero.

Every block-diagonal tensor $\mathcal{A} \in k^{n+1} \otimes (\bigoplus_{\lindex = 1}^r \Sym_2 k^{d_\lindex})$ defines an intersection of $r$ symmetroids $X := X_{1} \cap \ldots \cap X_{r}$, where $X_\lindex$ is the symmetroid associated to the $\lindex$-th block. Additionally, $Z(\det \mathcal{A}(\mathbf{x}, \cdot, \cdot)) \subseteq \mbp^n$ is the reducible determinantal hypersurface $X_1 \cup \ldots \cup X_r$. The \emph{block spaces} of $\mathcal{A}$ are the linear subspaces of $\mbp^m$ defined by the natural inclusions $k^{d_\lindex} \inj k^{m+1}$ corresponding to the decomposition $\bigoplus_{\lindex = 1}^r \Sym_2 k^{d_\lindex}$.

\begin{example}
\label{ex: block-diagonal example}
	Let $a,b,c,d,e,f,g,h,i \in k[x_0, x_1, x_2, x_3]$ be homogeneous linear forms. Define the tensors $\mathcal{A}^{(1)}$ and $\mathcal{A}^{(2)}$ by forming the two matrices
	\[
	\mathcal{A}^{(1)}(\mathbf{x}, \cdot, \cdot) := 
	\begin{bmatrix}
	a & c \\
	c & b 
	\end{bmatrix}, \quad
	\mathcal{A}^{(2)}(\mathbf{x}, \cdot, \cdot) := 
	\begin{bmatrix}
	d & g & h \\
	g & e & i \\
	h & i & f
	\end{bmatrix}.
	\]
	Specifically, each matrix of linear forms defines a $4 \times d \times d$ tensor by declaring the $j$-th slice of $\mathcal{A}^{(\lindex)}$ to be the matrix of coefficients of $x_j$. The two tensors $\mathcal{A}^{(1)}, \mathcal{A}^{(2)}$ can be joined diagonally along the first axis to create the $4 \times 5 \times 5$ tensor
	\[
	\mathcal{A}(\mathbf{x}, \cdot, \cdot) := 
	\begin{bmatrix}
	a & c & 0 & 0 & 0\\
	c & b & 0 & 0 & 0 \\
	0 & 0 & d & g & h \\
	0 & 0 & g & e & i \\
	0 & 0 & h & i & f
	\end{bmatrix}.
	\]
	The two determinantal varieties $X_1 := Z(\det \mathcal{A}^{(2)}(\mathbf{x}, \cdot, \cdot))$ and $X_2 := Z(\det \mathcal{A}^{(2)}(\mathbf{x}, \cdot, \cdot))$ generically intersect in a smooth canonical curve of genus $4$. The Cayley variety associated to $\mathcal{A}$ is a set of $16$ points in $\mbp^4$.
\end{example}

Every block-diagonal tensor with at least two blocks always has a non-trivial automorphism group. 

\begin{definition}
	Let $\mathcal{A} \in k^{n+1} \otimes (\bigoplus_{\lindex=1}^r \Sym_2 k^{d_\lindex})$ be a block-diagonal tensor with $r$ blocks. The group of \emph{diagonal automorphisms} of $\mathcal{A}$ is the group $\DAut := \Aut(\mathcal{A}) \cap \Lambda$, where $\Lambda := \{I\} \times \Lambda'$ and $\Lambda'$ is the subgroup of diagonal matrices of $\GL_{m+1}$.
\end{definition}

The diagonal automorphisms of a tensor $\mathcal{A}$ characterize the number of blocks; we have that $\mathcal{A}$ has $r$ blocks if and only if $\DAut(\mathcal{A}) \iso \mu_2^r$. The subgroup of $\DAut(\mathcal{A})$ generated by $-I$ is the kernel of $\DAut(\mathcal{A}) \rightarrow \PGL_{m+1}$.

%\bigskip
%Outside the intersection of any two components of $X_1 \cup \ldots \cup X_r$, the Gauss map and kernel map are well-defined. 
For the purposes of studying the intersection of symmetroids $X := X_1 \cap \ldots \cap X_r$, it is helpful to consider the restriction to $X$ of each of the kernel maps $\varphi_{X_1}, \ldots, \varphi_{X_r}$ in tandem. 

%There is a rational map $\theta_X \colon X \dashrightarrow \Gr(n+1-r, n+1)$ sending a smooth point of $x$ of $X$ to the point in $\Gr(n+1-r, n+1)$ representing its tangent space; define $\dual X$ to be the Zariski closure of $\theta_X(X)$. The map $\theta_X$ is the natural generalization of the Gauss map to lower dimensional subvarieties of projective space.

\begin{definition}
	The \emph{based kernel map} is the map
	\[
		\begin{tabu}{cccc}
			\varphi_{d_1,\dots,d_r} \: & X &\rightarrow & \mbp^{d_1 - 1} \times \dots \times \mbp^{d_r-1} \\
				& \mathbf{x} & \mapsto & (\ker \mathcal{A}^{(1)}(\mathbf{x}, \cdot, \cdot),\dots, \ker \mathcal{A}^{(r)}(\mathbf{x}, \cdot, \cdot))
		\end{tabu}
	\]
\end{definition}

\noindent
We discuss the properties of this map in more detail in Section~\ref{sec: general results}. Let $\mathcal{A}$ a block-diagonal tensor with $r$ blocks, let $\varphi_1, \ldots, \varphi_r$ be the kernel maps associated to each block, and let $x$ be a point in the intersection of symmetroids. We may consider each $\varphi_\lindex(x)$ as a point in $\mbp^m$ via the natural inclusions of the block spaces $\mbp^{d_\lindex-1} \inj \mbp^{m}$. We will write $\ker \mathcal{A}(x, \cdot, \cdot) = \gen{\varphi_1(x) , \ldots, \varphi_r(x)}$ to refer to this particular basis for $\ker \mathcal{A}(x, \cdots, \cdot) \subset \mbp^m$.

%%% aCM sheaves stuff.
\subsection{Arithmetically Cohen-Macaulay sheaves}

It is well-understood that the linear determinantal representations of a hypersurface in projective space can be classified in terms of arithmetically Cohen-Macaulay (aCM) sheaves on the projective space; see \cite{Beauville2000}, \cite[Section~4]{Dol2012}. 
A coherent sheaf $\mathcal{E}$ on $\mbp^n$ is \emph{arithmetically Cohen-Macaulay (aCM)} if
\begin{itemize}
	\item 
	Each stalk $\mathcal{E}_x$ is a Cohen-Macaulay $\mathcal{O}_{\mbp^n, x}$-module, and
	\item
	$\H^i(\mbp^n, \mathcal{E}(j)) = 0$ for $1 \leq i \leq  \dim(\supp \mathcal{E})-1$ and $j \in \mbz$.
\end{itemize}
Equivalently, setting $S := k[x_0, \ldots, x_n] = \bigoplus_{j \in \mbz} \H^0(\mbp^n, \mathcal{O}_{\mbp^n}(j))$, the sheaf $\mathcal{E}$ is aCM if and only if the graded $S$-module
\[
	\Gamma_*(\mathcal{E}) := \bigoplus_{j \in \mbz} \H^0(\mbp^n, \mathcal{E}(j))
\]
is Cohen-Macaulay \cite[Proposition~1.2]{Beauville2000}. 

Several aspects of our constructions can be described entirely in terms of aCM sheaves on $\mbp^n$. We let $\iota\: X \inj \mbp^n$ be an embedding of a hypersurface $X$ of degree $d$ with a linear symmetric determinantal representation. The linear symmetric determinantal representation on $X$ is canonically associated to a rank $1$ arithmetically Cohen-Macaulay sheaf $\mathcal{E}$ on $\mbp^n$ supported on $X$, together with a symmetric bilinear form $\eta\: \mathcal{E} \times \mathcal{E} \rightarrow {\iota_*\mathcal{O}_{X}(d-1)}$. The sheaf $\mathcal{E}$ admits a resolution
\[
	\xym{0 \ar[r] & \mathcal{O}_{\mbp^n}^d(-1) \ar[r]^-{M} & \mathcal{O}_{\mbp^n}^d \ar[r] & \mathcal{E} \ar[r] & 0}
\]
where $M$ is a symmetric matrix of linear forms. The projective space $\mbp^m$ in our main discussion can be identified with $\mbp\!\H^0(\mbp^n, \mathcal{E})^\vee \iso \mbp\!\H^0(\mbp^n, \mathcal{O}_{\mbp^n}^d)^\vee$. 

The adaptation to consider intersections of symmetroids is fairly straightforward. We consider $r$ determinantal hypersurfaces $X_{d_1}, \ldots, X_{d_r}$ with associated aCM sheaves $\mathcal{E}_1, \ldots, \mathcal{E}_r$. The natural aCM sheaf on $\mbp^n$ to consider for our purposes is
	\[
		\mathcal{E} := \bigoplus_{j=1}^r \mathcal{E}_j.
	\]
with the obvious resolution being the direct sum of the resolutions for the summands. The associated determinantal variety is $X := \bigcup_{j=1}^r X_{d_j}$. We let 
	\[
		\eta\: \mathcal{E} \times \mathcal{E} \rightarrow \mathcal{O}_{X}(d_1 + \ldots + d_r -1)
	\]
be the natural symmetric bilinear form coming from the determinantal representation of $X$. The bilinear form $\eta$ is compatible with the forms $\eta_1, \ldots, \eta_r$ on the summand bundles $\mathcal{E}_1, \ldots, \mathcal{E}_r$ as follows. 

\begin{proposition}
	Let $f_{1}, \ldots, f_r$ denote the defining equations of $X_{d_1}, \ldots, X_{d_r}$, these being of degrees $d_1, \ldots, d_r$. Let $d := d_1 + \ldots + d_r$. Let $\hat f_j := \frac{1}{f_j}f_1 \cdots f_r$. Consider the inclusions
		\[
		\mapdef{\gamma_j}{\mathcal{O}_{X_{d_j}}(d_j-1)}{\mathcal{O}_{X}(d-1)}{s}{\hat f_j s}
		\]
	Then we have that $\eta|_{\mathcal{E}_j \times \mathcal{E}_j} = \gamma_j\eta_j$.
\end{proposition}

\begin{proof}
	The symmetric bilinear form $\eta$ is canonically equivalent to an identification
	\[
		\alpha\: \mathcal{E} \rightarrow \sheafhom_{\mathcal{O}_X}(\mathcal{E}, \mathcal{O}_X(d-1))
	\]
	We note that any homomorphism from $\mathcal{E}_j$ to $\mathcal{O}_X(d-1)$ must restrict to zero on $X_{d_i}$ when $i \neq j$. In other words, the image of $\mathcal{E}_j$ is contained in the image of the ideal sheaf generated by $\hat f_j$. So $\alpha$ restricts to
	\[
		\alpha | _{\mathcal{E}_j}\: \mathcal{E}_j \rightarrow \sheafhom_{\mathcal{O}_X}(\mathcal{E}_j, \hat f_j\mathcal{O}_{X_{d_j}}(d_j-1))
		= \gamma_j \left( \sheafhom_{\mathcal{O}_{X_{d_j}}}(\mathcal{E}_j, \mathcal{O}_{X_{d_j}}(d_j-1)) \right).
	\]
	For
	$
		\alpha_j\: \mathcal{E}_j \rightarrow \sheafhom_{\mathcal{O}_{X_{d_j}}}(\mathcal{E}_j, \mathcal{O}_{X_{d_j}}(d_j-1)),
	$
	we see that $\alpha|_{\mathcal{E}_j} = \gamma_j \alpha_j$.
\end{proof}

Finally, we show how to describe the $\psi$ map in terms of the aCM sheaf $\mathcal{E}$. To do so, we follow the manipulations of Beauville~\cite{Beauville2000}. First, we may twist the natural resolution to obtain
\[
	\xym{0 \ar[r] & \mathcal{O}_{\mathbb{P}^n}^{m+1} \ar[r]^M & \mathcal{O}_{\mathbb{P}^n}^{m+1}(1) \ar[r] & \mathcal{E}(1) \ar[r] & 0}.
\]
Applying the functor $\sheafhom_{\mathcal{O}_{\mathbb{P}^n}}(-, \mathcal{O}_{\mathbb{P}^n}(1))$ we obtain
\[
	\xym{0 \ar[r] & \mathcal{O}_{\mathbb{P}^n}^{m+1} \ar[r]^{M^T} & \mathcal{O}_{\mathbb{P}^n}^{m+1}(1) \ar[r] & \sheafext_{\mathcal{O}_{\mathbb{P}^n}}^1\left(\mathcal{E}(1), \mathcal{O}_{\mathbb{P}^n}(1)\right) \ar[r] & 0}.
\]
Grothendieck duality provides a canonical isomorphism $\sheafext_{\mathcal{O}_{\mathbb{P}^n}}^1\left(\mathcal{E}(1), \mathcal{O}_{\mathbb{P}^n}(1)\right) \iso \iota_*\sheafhom_{\mathcal{O}_X}(\mathcal{E}(1), \mathcal{O}_X(m))$. Assuming that $\mathcal{E}$ admits a symmetric structure, there is an isomorphism $\alpha\: \mathcal{E} \rightarrow \sheafhom_{\mathcal{O}_X}(\mathcal{E}, \mathcal{O}_X(m))$ and we obtain via pullback
\[
	\xym{
	0 \ar[r] & \mathcal{O}_{\mathbb{P}^n}^{m+1} \ar[r]^M \ar[d]^\alpha & \mathcal{O}_{\mathbb{P}^n}^{m+1}(1) \ar[r] \ar[d]^\alpha & \mathcal{E}(1) \ar[r] \ar[d]^\alpha & 0\phantom{.} \\
	0 \ar[r] & \mathcal{O}_{\mathbb{P}^n}^{m+1} \ar[r]^{M^T} & \mathcal{O}_{\mathbb{P}^n}^{m+1}(1) \ar[r] & \sheafhom_{\mathcal{O}_X}(\mathcal{E}(1), \mathcal{O}_X(m)) \ar[r] & 0.}
\]
We have the canonical isomorphism $\sheafhom_{\mathcal{O}_{\mathbb{P}^n}}(\mathcal{O}_{\mathbb{P}^n}^{m+1}, \mathcal{O}_{\mathbb{P}^n}(1)) \iso\mathcal{O}_{\mathbb{P}^n}^{m+1}(1)$. Thus, we obtain a bilinear map
\[
	\mapdef{\psi}{\mathcal{O}_{\mathbb{P}^n}^{m+1} \times \mathcal{O}_{\mathbb{P}^n}^{m+1}}{\mathcal{O}_{\mathbb{P}^n}^{m+1}(1)}{(y,y')}{\alpha(My)(y')}.
\]
On the level of global sections, this induces the $\psi$ map that we expect. Computationally, we often start with an explicit symmetric matrix $M$ of linear forms, in which case the identification $\alpha\: \mathcal{O}_{\mathbb{P}^n}^{m+1} \rightarrow \mathcal{O}_{\mathbb{P}^n}^{m+1}$ is given by the identity.

%%% SECTION %%%
\section{Genus 3: plane quartics and bitangents} \label{sec: plane quartics}

There is a classical connection between the Cayley octad and the $28$ bitangents of a smooth plane quartic curve; studying this connection and generalizing was a major inspiration for this paper. In this section, we will illustrate the highlights of this classical construction, so that this section can serve as an illustrative example for the rest of the paper. 

The canonical model of a smooth non-hyperelliptic genus $3$ curve $X$ is a smooth quartic curve in $\mbp^2$. Each of the $36$ even theta characteristics on $X$ gives rise to a symmetric $4 \times 4$ matrix $x_0 A_0 + x_1 A_1 + x_2 A_2$ of linear forms of $\mbp^2$, whose determinant is the defining equation of $X$. The three symmetric matrices $A_0, A_1, A_2$ with entries in $k$ define three quadrics in $\mbp^3$, whose intersection is a cluster of $8$ points $\mathfrak{C}$ called the \emph{Cayley octad}. In our notation, we have a tensor $\mathcal{A} \in k^3 \otimes \Sym_2 k^4$, an associated symmetroid hypersurface $X$, and associated Cayley variety $\mathfrak{C}$. 

\begin{remark}[Moduli]
	There is a canonical bijection between pairs $(X, \theta)$ of smooth genus $4$ curves over $k$ with an even theta characteristic $\theta$ defined over $k$ and orbit classes of $k^3 \otimes \Sym_2 k^4$ under $\GL_3 \times \GL_4$ whose determinant defines a smooth curve. See \cite[Section~4]{Dol2012}, or \cite[Theorem~4.10, Theorem~4.12]{Ho2009} for a formulation in terms of stacks.
\end{remark}

By Proposition~\ref{prop: main diagram hypersurfaces}, we have the commutative diagram
	\[
	\begin{tikzcd}
		X \arrow[hook]{r}{} \arrow[dashed]{d}[left]{\theta_X} \arrow[swap, dashed]{drr}[above]{\varphi} & \mbp^2 \arrow{r}{\dual \psi} & \mbp\H^0(\mbp^3, \mathcal{O}_{\mbp^3}(2))  \\
		\dual X \arrow[hook]{r}{} & \dual \mbp^2 & \mbp^3 \arrow[dashed]{l}{\psi}
	\end{tikzcd}.	
	\]

If $\ell(q_1,q_2) \in \mbp^m$ is a secant line of $\mathfrak{C}$ for $q_1, q_2$ distinct geometric points of $\mathfrak{C}$, we have the linear form given by $\psi(\ell(q_1,q_2)) = \mathcal{A}(\mathbf{x}, q_1, q_2)$ defines a bitangent line of $X$ -- see \cite[Section~6]{Dol2012}. A simple way to see this is the argument in \cite[Proposition~3.3]{Plaumann2011}. Act by $\GL_4$ so that $q_1=[1:0:0:0]$ and $q_2 = [0:1:0:0]$. Hence we may write
\[
	\mathcal{A}(x, \cdot, \cdot)
	= 
	\begin{bmatrix}
		0 & h & a & b \\
		h & 0 & c & d \\
		a & c & e & f \\
		b & d & f & g
	\end{bmatrix}
\]
with $a, \ldots, g, h \in k[x_0, x_1, x_3]$ linear forms. Modulo $h$, the determinant of $\mathcal{A}(x, \cdot, \cdot)$ is a square, viz., $Z(h)$ is a bitangent of $X$.

Our results in Section~\ref{sec: general results} describe what is going on in a generalized framework. The restriction of a symmetroid to a hyperplane is again a symmetroid, and we may construct the tensor $\mathcal{B}$ from $\mathcal{A}$ by reduction modulo $h$. Theorem~\ref{thm: singularities of symmetroid intersections with multiplicity} is a general result, which applied in this specific context shows that there is an isomorphism
\[
	\sing(X \cap H) \iso \{(x,q) \in (X \cap H) \times \mathfrak{C}_{\mathcal{B}} : \mathcal{B}(x, q, \cdot) = 0\}.
\]
On the other hand, determining the points of the incidence variety on the right is essentially a linear algebra problem -- there is a solution to $\mathcal{B}(x, q, \cdot) = 0$ if and only if the rank of $\begin{bmatrix} a(x) & b(x) \\ c(x) & d(x) \end{bmatrix}$ is equal to $1$ (since $X$ has no essential singularities, see Section~\ref{sec: general results}). In the generic case, a pencil of $2 \times 2$ matrices degenerates at two points, meaning that $H$ is a bitangent of $X$. For symmetroid hypersurfaces, this second argument generalizes directly, and with some mindfulness toward the details the argument works for intersections of symmetroids as well -- see Proposition~\ref{prop: deg and dim of tangents of intersections of symmetric determinantal hypersurfaces}.
The case of plane quartic curves also provides a nice application of Theorem~\ref{thm: singularities of symmetroid intersections with multiplicity}.

\begin{example}
	A plane quartic with a symmetric determinantal representation $\mathcal{A}$ is singular if and only if the Cayley octad $\mathfrak{C}$ is non-singular, not contained in a union of two hyperplanes, and of dimension 0. If $\mathfrak{C}$ is not a complete intersection of dimension $0$, then $\mathfrak{C}$ is either a quadric surface, a genus $1$ curve, or it is a twisted cubic curve in $\mbp^3$. In the first two cases the determinantal quartic in $\mbp^2$ is reducible, and in the case $\mathfrak{C}$ is a twisted cubic curve it turns out that there is a solution to $\mathcal{A}(p, q, \cdot) = 0, q \in \mathfrak{C}$ on the variety $X \times \mathfrak{C}$.

	If $\mathfrak{C}$ is a complete intersection of dimension $0$, it is singular if and only if $X$ has an accidental singularity or $\mathfrak{C}$ is contained in a union of two hyperplanes, i.e., one of the quadrics in the net has rank $1$. This shows that singularities of symmetric determinantal plane quartics are characterized by the Cayley variety.
\end{example}

\subsection{The number of $k$-rational bitangents of a smooth plane quartic}

Let $X$ be a smooth plane quartic defined over $k$. A symmetric determinantal representation of $X$ can be used to determine the number of $k$-rational bitangents of that smooth plane quartic. Specifically: %We determine precisely which integers occur as the number of $k$-rational bitangents of a plane quartic, even when curve does not have a symmetric determinantal representation over $k$.

\begin{theorem}
\label{thm: number of bitangents}
Let $k$ be an infinite field of characteristic not $2$ or $3$. For any smooth non-hyperelliptic genus $3$ curve $X$ defined over $k$, the number of odd theta characteristics of $X$ defined over $k$ lies in $S := \{28, 16, 10, 8, 6, 4, 3, 2, 1, 0\}$. If $k$ has normal separable extensions of degrees $1, \ldots, 8$, then each of the numbers in $S$ is the number of $k$-rational bitangents of some smooth plane quartic $X$ defined over $k$. 
\end{theorem}

The remainder of this section is devoted to the proof of this theorem. First, we note that if
$\mathcal{A} \in k^3 \otimes \Sym_2 k^4$ is a tensor such that $X = Z(\det \mathcal{A}(\mathbf{x}, \cdot, \cdot))$ is a smooth plane quartic, then for any $p_1, p_2 \in \mathfrak{C}$, the bitangent $\psi(\ell(p_1,p_2))$ is defined over $k$ if and only if the set $\{p_1,p_2\}$ is defined over $k$. We will also use the following theorem that describes the even theta characteristics of $X$. 
\begin{comment}
\begin{lemma}
\label{lem: number of rational bitangents with rep}
Let $\mathcal{A} \in k^3 \otimes \Sym_2 k^4$ be a symmetric determinantal representation of a smooth plane quartic $X$. For any $p_1, p_2 \in \mathfrak{C}$, the bitangent $\psi(\ell(p_1,p_2))$ is defined over $k$ if and only if the set $\{p_1,p_2\}$ is defined over $k$.
\end{lemma}

\begin{proof}
	By definition, $\psi(\ell(p_1,p_2)) = p_1 \mathcal{A} p_2^{T}$. Let $l$ be the (Galois) field of definition of the Cayley octad points. Let $\sigma \in \Gal(l/k)$. Each unordered tuple of Cayley octad points defines a unique bitangent of $X$, so the bitangent $\psi(\ell(p_1,p_2))$ is defined over $k$ if and only if 
	\[
	p_1^{\sigma} \mathcal{A} (p_2^{\sigma})^{T} = p_1 \mathcal{A} p_2^{T}
	\]
	which happens if and only if either $p_1$ and $p_2$ are defined over $k$ or are conjugate to each other in a quadratic extension.
\end{proof}
\end{comment}

%\noindent We require the following theorem.

\begin{theorem}[\cite{Dol2012}, Section~6.3.2, Theorem~6.3.3]
\label{thm: genus 3: even theta presentations}
Write the Cayley octad as $\mathfrak{C} = \{p_1,\dots,p_8\}$. Let $\theta_{i,j}$ be the odd theta characteristic given by $\psi(\ell(p_i,p_j))$. Then for any $\{i,j,k,l\} \subset \{1,\dots,8\}$, the line bundle
\[
	\theta_{i,j} \otimes \theta_{i,k} \otimes \theta_{i,l} \otimes \omega_C^{\vee}
\]
is an even theta characteristic. Moreover, this line bundle is independent of the ordering of $i$, $j$, $k$, and $l$. For $\{i',j',k',l'\} \cap \{i,j,k,l\} = \emptyset$, we have
\[
	\theta_{i,j} \otimes \theta_{i,k} \otimes \theta_{i,l} \otimes \omega_C^{\vee} \cong \theta_{i',j'} \otimes \theta_{i',k'} \otimes \theta_{i',l'} \otimes \omega_C^{\vee}.
\]
\end{theorem}

\begin{lemma}
	\label{lem: noeventheta}
	If $X$ does not have an even theta characteristic defined over $k$, then $X$ has at most $4$ odd theta characteristics defined over $k$.
\end{lemma}

\begin{proof}
	Let $X$ be a smooth plane quartic over $k$ with at least $5$ rational bitangents and let $\mathcal{A} \in \bar k^3 \otimes \Sym_2 \bar k^4$ be a tensor such that $X = Z(\det \mathcal{A}(\mathbf{x}, \cdot, \cdot))$. Let $\{p_1, \ldots, p_8\}$ be the geometric points of the Cayley octad defined by $\mathcal{A}$, and let $\vartheta_{i,j}$ denote the odd theta characteristic given by $\psi(\ell(p_i,p_j))$. Note that because $\mathcal{A}$ is not necessarily defined over $k$, the Galois action on $\{p_1, \ldots, p_8\}$ does not naturally correspond to a Galois action on the odd theta characteristics of $X$; we are simply using the Cayley octad to label the bitangents.
	
	Let $G = (V, E)$ be the graph where $V = \{1,\dots,8\}$ and an edge exists between vertices $i$ and $j$ if and only if $\vartheta_{i,j}$ is defined over $k$. There are at least $5$ edges of $G$, so either $3$ edges meet at a vertex or there are $3$ edges $a$, $b$, $c$ such that $a$ and $b$ are adjacent, $a$ and $c$ are not adjacent, and $b$ and $c$ are not adjacent. We may assume up to relabelling that $a = (1,2), b = (1,3), c = (4,5) \in E$. Then by Theorem~\ref{thm: genus 3: even theta presentations}
		\[
			\vartheta_{1,2,3,4} = \vartheta_{1,2} \otimes \vartheta_{1,3} \otimes \vartheta_{1,4} \otimes \omega_{X}^{\vee} \cong \vartheta_{4,1} \otimes \vartheta_{4,2} \otimes \vartheta_{4,3} \otimes \omega_{X}^{\vee}
		\]
		and thus
		$
			 \vartheta_{1,2} \otimes \vartheta_{1,3} \cong \vartheta_{4,2} \otimes \vartheta_{4,3}.
		$
		Because $\vartheta_{1,2} \otimes \vartheta_{1,3}$ is defined over $k$, so is $\vartheta_{4,2} \otimes \vartheta_{4,3}$. Thus, as $\vartheta_{4,5}$ is defined over $k$, the even theta characteristic $\vartheta_{2,3,4,5} = \vartheta_{4,2} \otimes \vartheta_{4,3} \otimes \vartheta_{4,5} \otimes \omega_{X}^{\vee}$ is as well.
\end{proof}

\begin{proof}[Proof of Theorem~\ref{thm: number of bitangents}]
We first prove that the number of odd theta characteristics defined over $k$ is one of the numbers in $S$. If there does not exist an even theta characteristic defined over $k$, then apply by Lemma~\ref{lem: noeventheta}. Otherwise, choose an even theta characteristic defined over the ground field and a corresponding symmetric determinantal representation of $X$. Suppose the Cayley octad for this representation has $n_1$ points that are $k$-rational and $n_2$ pairs of quadratic conjugate points, with $n_1 + 2n_2 \leq 8$. 
%
%Then by Lemma~\ref{lem: number of rational bitangents with rep}, 
%
Then $X$ has ${\binom{n_1}{2}} + n_2$ bitangents defined over $k$. Note $S = \{\binom{n_1}{2} + n_2 : 0 \leq n_1, n_2 \text{ and } n_1 + 2n_2 \leq 8\}$. 

The second part of the claim follows from the result of \cite{ElsenhansJahnel2019}. Any Galois action on the bitangents of $X$ acts via a subgroup of $\Sp(6, \mathbb{F}_2)$.  This group contains a unique subgroup $U_{36}$ up to conjugacy of index $36$; such a subgroup is the stabilizer of some even theta characteristic of $X$. Furthermore, $U_{36} \iso S_8$. The result:

\begin{theorem}[{\cite[Theorem~1.2]{ElsenhansJahnel2019}}]
	Let $k$ be an infinite field of characteristic not 2, let $L$ be a normal and separable
	extension field, and let $i\: \Gal(L/k) \rightarrow U_{36}$ be an injective group homomorphism.
	Then there exists a nonsingular quartic curve $X$ over $k$ such that $L$ is the field of definition of the $28$ bitangents and each $\sigma \in \Gal(L/k)$ permutes the bitangents
	as described by $i(\sigma) \in U_{36}$.
\end{theorem}

\noindent
finishes the proof of Theorem~\ref{thm: number of bitangents}.
\end{proof}

\section{Singularities of intersections of symmetroids}
\label{sec: general results}
	
	In this section we explore the relationship between singularities of an intersection of symmetroids defined by a tensor $\mathcal{A}$ and the Cayley variety associated to $\mathcal{A}$. First, we describe this relationship set-theoretically and later give a more precise relationship in terms of morphisms of schemes.
	
	\begin{definition}
		Let $\mathcal{A}$ be a block-diagonal tensor with blocks $\mathcal{A}^{(1)}, \ldots, \mathcal{A}^{(r)}$ and let $X$ be the associated intersection of symmetroids. A point $x \in X$ is an \emph{essential singularity} if $x$ is a singular point and $\corank{A}^{(\lindex)}(x, \cdot, \cdot) \geq 2$ for some $\lindex$. If $x \in X$ is singular, but not an essential singularity, it is an \emph{accidental singularity}. The complement of the essential singularities in $X$ is denoted by $X^{\nonessential}$.
	\end{definition}
	
	If $X$ is a complete intersection defined by $\mathcal{A}$ and $\corank{A}^{(\lindex)}(x, \cdot, \cdot) \geq 2$ for some point $x$ and block $\mathcal{A}^{(\lindex)}$, then $x$ is a singular point of $X$.
	
	\begin{definition}
		We say that accidental singularity $x \in X$ is an \emph{inherited singularity} if it is a singularity of $X' \supset X$, where $X'$ is an intersection of symmetroids defined using a strict subset of the blocks of $\mathcal{A}$. A \emph{coincident singularity} of $X$ is an accidental singularity $x \in X$ which is a singular point of two intersections of symmetroids $X', X'' \supset X$ defined by two distinct strict subsets of the blocks of $\mathcal{A}$. 
	\end{definition}

	The corank of the Jacobian matrix of $X$ classifies the coincident singularities of $X$.
	
	\begin{lemma}
		Let $X$ be a complete intersection of symmetroids, let $x \in X$, and let $J(x)$ be the Jacobian matrix for $X$ at $x$. Then $\corank J(x) > 1$ if and only if $x$ is a coincident singularity. 
	\end{lemma}
	
	\begin{proof}
		If $X := Z(f_1, \ldots, f_r)$ is an intersection of the $r$ symmetroids $Z(f_\lindex)$, we may assume that $J(x)$ is the Jacobian matrix associated to $f_1, \ldots, f_r$.
		Since $X$ is a complete intersection of symmetroids, we may assume that the rows of the Jacobian matrix are given by the equations for these symmetroids. 
		Suppose $\corank J(x) > 1$ and choose $2$ linearly independent vectors $v$ and $w$ in the left kernel of $J(x)$. Take a linear combination to force one of the entries of $v$ and a different entry of $w$ to be zero. Then $v$ and $w$ correspond to linear combinations of two distinct subsets of the rows of $J(x)$, which shows that there are two symmetroids $X'$ and $X''$ containing $X$ with an accidental singularity at $x$. The converse follows as well.
	\end{proof}
	
	Every essential singularity of $\mathcal{A}$ is an inherited singularity. Theorem~\ref{thm: singularities of symmetroid intersections} characterizes when an intersection of symmetroids has accidental singularities. Given a singular point $q$ of $\mathfrak{C}$ and a chosen lift $q \in \H^0(\mbp^n, \mathcal{E})$, denote by $q_\lindex$ its restrictions to $\H^0(\mbp^n, \mathcal{E}_\lindex)$, for $1 \leq \lindex \leq r$.

	\begin{theorem} \label{thm: singularities of symmetroid intersections}
		Let $\mathcal{A} \in k^{n+1} \otimes \Sym_2 k^{m+1}$ be a block-diagonal tensor over a field $k$ and let $X$ be the associated intersection of symmetroids. Let $q \in \H^0(\mbp^n, \mathcal{E})$ represent a singular point of the Cayley variety associated to $\mathcal{A}$, with restrictions $q_\lindex$ to the summands of $\mathcal{E}$. Then:
		
		\begin{enumerate}[(a)]
		\item
		If $\mathfrak{C}$ is a complete intersection of $n+1$ quadrics, then there exists a $p \in k^{n+1}$ such that $\mathcal{A}(p, q, \cdot) = 0$ if and only if $q$ represents a singular point of $\mathfrak{C}$.
		
		\item
		If $\mathcal{A}(p, q, \cdot) = 0$ and each $q_\lindex$ is non-zero, then $p$ is a singular point of $X$. If $X$ is a complete intersection and $p \in X(\bar k)$ is an accidental singularity of $X$, then there exists a $q = (q_1, \ldots, q_r)$ representing a point of $\mathfrak{C}(\bar k)$ such that each $q_\lindex = \varphi_\lindex(p)$ and $\mathcal{A}(p, q, \cdot) = 0$. 
		
		\item
		If some $q_\lindex = 0$, then $(q_i)_{i \neq \lindex}$ represents a point of the Cayley variety for the tensor $\mathcal{A}^{(\neg \lindex)}$ obtained by deleting the $\lindex$-th block of $\mathcal{A}$. Conversely, if $q' := (q_1', \ldots, q_{r-1}')$ is a point of the Cayley variety obtained by deleting the $\lindex$-th block, then $q := (q_1', \ldots, 0, \ldots, q_{r-1}')$ is a point of $\mathfrak{C}$, where the $0$ vector is spliced into the $\lindex$-th component. Finally, $\mathcal{A}(p, q, \cdot) = 0$ if and only if $\mathcal{A}^{(\neg \lindex)}(p, q', \cdot) = 0$.
		\end{enumerate}
	\end{theorem}
	
	\begin{proof}
		Denote the blocks of $\mathcal{A}$ by $\mathcal{A}^{(1)}, \ldots, \mathcal{A}^{(r)}$. Let $A_0, \ldots, A_n$ denote the slices of $\mathcal{A}$ corresponding to the standard basis vectors and let $Q_0, \ldots, Q_n$ be the associated quadratic forms. 
		
		If the Cayley variety $\mathfrak{C}$ associated to $\mathcal{A}$ has a singular point $q$, then by definition the Jacobian matrix for $Z(Q_0, \ldots, Q_n)$ has a rank degeneracy. If $\mathfrak{C}$ is a complete intersection, this is equivalent to the existence of a non-zero vector $p$ such that
		\[
			p \cdot \left(\frac{\der Q_i}{\der y_j}(q) \right)_{i,j} = 0.
		\]
		Identify the linear forms $\frac{\der Q_i}{\der y_j}$ with their vector of coefficients. Thus, the equation above is equivalent to
		$
			0 = 2 \cdot \mathcal{A}(p, q, \cdot),
		$
		which completes the proof of part (a).
		
		Since $\mathcal{A}$ is block-diagonal, the projections $q_1, \ldots, q_r$ to the $r$ subspaces corresponding to the blocks yields the equations $0 = 2 \cdot \mathcal{A}^{(\lindex)}(p, q_\lindex, \cdot)$. As $q \in \mathfrak{C}$, we have that $\mathcal{A}^{(\lindex)}(\cdot, q, q) = 0$. From the block-diagonal structure we have that $\sum_{\lindex=1}^r \mathcal{A}^{(\lindex)}(\cdot, q_\lindex, q_\lindex) = \mathcal{A}(\cdot, q, q) = 0$.
		
		If each $q_\lindex$ is non-zero, then by definition of $X$, we see that the point in $\mathbb{P}^n$ associated to $p$ lies on $X$. If any $\mathcal{A}^{(\lindex)}(p, \cdot, \cdot)$ has corank greater than $1$, then $p$ is a singular point of $Z(\det \mathcal{A}^{(\lindex)}(\mathbf{x}, \cdot, \cdot))$, so $X$ is singular at $p$. Alternatively, each block has corank $1$ at $p$, so by definition of the kernel map $q_\lindex = \varphi_\lindex(p)$. We now consider the Jacobian matrix for $X$ at $p$, which is given by
		\[
			\left(\frac{\der \det \mathcal{A}^{(\lindex)}(\mathbf{x}, \cdot, \cdot)}{\der x_j}(p) \right)_{\lindex,j}.
		\]
		Jacobi's formula gives that $\frac{\der \det \mathcal{A}^{(\lindex)}(\mathbf{x}, \cdot, \cdot)}{\der x_j}(p) = \trace \left( \operatorname{adj}\mathcal{A}^{(\lindex)}(p, \cdot, \cdot) \cdot A_j^{(\lindex)} \right)$. Since each $\mathcal{A}_j^{(\lindex)}$ has corank $1$ at $p$, we have $\operatorname{adj}\mathcal{A}_j^{(\lindex)}(p, \cdot, \cdot)$ is a scalar multiple of $q_\lindex q_\lindex^T$. But as $q \in \mathfrak{C}$ we have
		\[
			\sum_{\lindex=1}^r \trace \left( q_\lindex q_\lindex^T \cdot A_j^{(\lindex)} \right) = \sum_{\lindex=1}^r q_\lindex^T A_j^{(\lindex)} q_\lindex = q^T A_j q = 0.
		\]
		Thus, there is a nontrivial relation among the rows of the Jacobian matrix $J(p)$ for $X$ at $p$, so $X$ is singular at $p$. Conversely, suppose that $p$ is a singular point of $X$ and each $\mathcal{A}^{(\lindex)}(p, \cdot, \cdot)$ has corank $1$. Since $X$ is a complete intersection, the singularity at $p$ ensures there is a solution to $0 = v \cdot J(p)$. From the corank assumption, we define $q_\lindex := \varphi_\lindex(p) \neq 0$. The Jacobi formula then shows that $q := (\sqrt{v_1} \ q_1, \ldots, \sqrt{v_r} \ q_r)$ satisfies $\mathcal{A}(\cdot, q, q) = 0$. On the other hand, $\mathcal{A}(p, q, \cdot) = 0$. Note that different choices of signs for the square roots of the $v_\lindex$ correspond to applying automorphisms of $\mathcal{A}$ to $q$.
		
		If some $q_\lindex = 0$, let $\mathcal{A}^{(\neg \lindex)}$ be the tensor obtained from $\mathcal{A}$ by deleting the $\lindex$-th block. Then $\sum_{\lindex' \neq \lindex} \mathcal{A}^{(\lindex')}(\cdot, q_{\lindex'}, q_{\lindex'}) = 0$, which by definition means $q' := (q_{\lindex'})_{\lindex' \neq \lindex}$ is a point on the Cayley variety $\mathfrak{C}_{\mathcal{A}^{(\neg \lindex)}}$ for $\mathcal{A}^{(\neg \lindex)}$. Furthermore, we see that $\ker \mathcal{A}^{(\neg \lindex)}(\cdot, q', \cdot)$ is non-trivial, so $q'$ is a point of $\mathfrak{C}_{\mathcal{A}^{(\neg \lindex)}}$. The converse also follows from the previous calculation, and the last statement is clear.
	\end{proof}

	One cannot replace Theorem~\ref{thm: singularities of symmetroid intersections}(b) with \emph{``If $\mathcal{A}(p, q, \cdot) = 0$ and each $q_\lindex$ is non-zero, then $p$ is an \underline{accidental} singularity of $X$ ...''}. This is because it is possible for the closed conditions $\mathcal{A}(p, q, \cdot) = 0, \mathcal{A}(\cdot, q, q) = 0, \corank \mathcal{A}(p, \cdot, \cdot) \geq 2$ to overlap. For example, with
	\[
		\mathcal{A}(\mathbf{x}, \cdot, \cdot) := 
		x_0
		\begin{bmatrix}
			0 & 0 & 0 \\
			0 & 0 & 0 \\
			0 & 0 & 1
		\end{bmatrix}
		+
		x_1
		\begin{bmatrix}
			0 & 1 & 2 \\
			1 & 3 & 4 \\
			2 & 4 & 5
		\end{bmatrix}
		+
		x_2
		\begin{bmatrix}
			0 & 6 & 7 \\
			6 & 8 & 9 \\
			7 & 9 & 10
		\end{bmatrix}, 
	\]
	the point $p := [1:0:0] \in Z(\det \mathcal{A}(\mathbf{x}, \cdot, \cdot))$ is an essential singularity. However, $q := [1:0:0] \in \mathfrak{C}(\bar k)$ and $\mathcal{A}(p, q, \cdot) = 0$. The singularity at $p$ is in some sense ``accidentally essential''.
	
	\begin{corollary}
		Let $\mathcal{A} \in k^{n+1} \otimes \Sym_2 k^{m+1}$ be a tensor with a single block, let $X$ be the corresponding determinantal hypersurface, and assume that the Cayley variety associated to $\mathcal{A}$ is a complete intersection of $n+1$ quadrics. If $p$ is an accidental singularity of $X$, then $\varphi(p)$ is a singular point of the Cayley variety. Conversely, if $\varphi(p)$ is defined and a singular point of $\mathfrak{C}$, then $p$ is an accidental singularity. Finally, if $\varphi$ is not defined at $p$, then $p$ is an essential singularity.
	\end{corollary}
	
	\begin{proof}
		Immediate from Theorem~\ref{thm: singularities of symmetroid intersections}(a).
	\end{proof}
		
	If $\mathcal{A}$ is a tensor whose associated intersection of symmetroids $X$ is non-singular, Theorem~\ref{thm: singularities of symmetroid intersections}(b) shows that it is still possible for the Cayley variety to be singular; in this case, Theorem~\ref{thm: singularities of symmetroid intersections}(b) indicates that the singularity on the Cayley variety corresponds to an accidental singularity on some intersection of symmetroids produced by deleting blocks of $\mathcal{A}$. Conversely, a determinantal hypersurface with only essential singularities will generically have a smooth Cayley variety -- examples include a generic quintic symmetroid surface in $\mathbb{P}^3$, which has $20$ essential singularities and a smooth Cayley variety, or a generic cubic symmetroid in $\mathbb{P}^3$, which has $4$ essential singularities and an empty Cayley variety.  
		
	\begin{corollary}
		Let $\mathcal{A} \in k^{n+1} \otimes \Sym_2 k^{m+1}$ be a tensor, let $X$ be the associated intersection of symmetroids, and let $U$ be the open subscheme of $\mbp^m = \mbp\!\H^0(\mbp^n, \mathcal{E})$ obtained by removing $\mbp\!\H^0(\mbp^n, \mathcal{E}_1), \ldots, \mbp\!\H^0(\mbp^n, \mathcal{E}_r)$. Further assume that $\mathfrak{C}$ is a complete intersection of $n+1$ quadrics. If $\mathfrak{C} \cap U$ is non-singular then $X$ has no accidental singularities. If $\mathfrak{C} \cap U$ is singular, then $X$ is singular.
	\end{corollary}
	
	\begin{proof}
		Follows from Theorem~\ref{thm: singularities of symmetroid intersections}(a).
	\end{proof}
		
	The Cayley variety associated to a symmetroid hypersurface of degree $2$ or $3$ in $\mbp^n$ for $n \geq 2$ is generically empty. These hypersurfaces are especially singular if it is not.
	
	\begin{lemma} \label{lem: cayley variety for quadrics}
		Let $n \geq 2$ and let $\mathcal{A} \in k^{n+1} \otimes \Sym_2 k^2$, and let $X$ be the associated symmetroid hypersurface of degree $2$ in $\mathbb{P}^n$. If the Cayley variety for $\mathcal{A}$ is non-empty, then $X$ is non-reduced.
	\end{lemma}
	
	\begin{proof}
		Applying an action of $\GL_{n+1} \times \GL_{2}$, we may write $\mathcal{A}(\mathbf{x}, \cdot, \cdot) = \begin{bmatrix} 0 & a \\ a & b \end{bmatrix}$. The determinant is $-a^2$. 
	\end{proof}

	\begin{lemma} \label{lem: cayley variety for cubics}
		Let $n \geq 2$ and let $\mathcal{A} \in k^{n+1} \otimes \Sym_2 k^3$, and let $X$ be the associated symmetroid hypersurface of degree $3$ in $\mathbb{P}^n$. If the Cayley variety for $\mathcal{A}$ is non-empty, then $X$ contains a linear subspace of $\mathbb{P}^n$ of dimension at least $n-2$, along which it is singular.
	\end{lemma}
	
	\begin{proof}
		Applying an action of $\GL_{n+1} \times \GL_{3}$, we may write $\mathcal{A}(\mathbf{x}, \cdot, \cdot) = \begin{bmatrix} 0 & x_0 & x_1 \\ x_0 & a & b \\ x_1 & b & c \end{bmatrix}$. The determinant is $2bx_0x_1 - c x_1^2 - a x_0^2$, which is singular along $Z(x_0, x_1)$. 
	\end{proof}

%%% SUBSECTION
\subsection{Properties of the Gauss and kernel maps}

Here, we discuss the relationship between the Gauss and kernel maps for intersections of symmetroids. Choose $\mathcal{A} \in k^{n+1} \otimes (\bigoplus_{\lindex=1}^r \Sym_2 k^{d_\lindex})$ to be a block-diagonal tensor with $r$ blocks, and denote by $X$ the intersection $X_1 \cap \ldots \cap X_r$ of $r$ symmetroid hypersurfaces (of degrees $d_1, \ldots, d_r$). On each hypersurfaces, there is an exact sequence of sheaves
	\[
		\xymatrix@C+10pt{
		0 \ar[r] & \mathcal{K}_\lindex \ar[r] & \mathcal{O}_{X_\lindex}^{d_\lindex}(d_\lindex-1) \ar[rr]^-{\mathcal{A}^{(\lindex)}(\mathbf{x}, \cdot, \cdot)} & & \mathcal{O}_{X_\lindex}^{d_\lindex}(d_\lindex) \ar[r] & \mathcal{E}_\lindex(d_\lindex) \ar[r] & 0.
		}
	\]
If $X$ is a hypersurface, $\mathcal{K}_\lindex$ defines a line bundle on $X_\lindex^\nonessential$, see \cite[Section~2.4]{Kerner2012}. In particular, $\mathcal{K} := \bigoplus_{\lindex=1}^r \mathcal{K}_\lindex$ is a vector bundle on $X^\nonessential$ of rank $r$. Because it is also a subbundle of the rank $m+1$ bundle $\bigoplus_{\lindex=1}^r \mathcal{O}_{X}^{d_\lindex}(d_\lindex-1)$ by definition, it defines a morphism $X \rightarrow \Gr(r, m+1)$. Because of the direct sum decomposition of $\mathcal{K}$, this map factors through $\prod_{\lindex=1}^r \Gr(1, d_\lindex) \rightarrow \prod_{\lindex=1}^r \Gr(1, m+1) \dashrightarrow \Gr(r, m+1)$; the morphism $X \dashrightarrow \prod_{\lindex=1}^r \Gr(1, d_\lindex)$ is the based kernel map.

On the other hand, the Gauss map $\theta_X\: X \dashrightarrow \Gr(r, n+1)$ sends a smooth point of $x$ to the moduli point of the $(r-1)$-dimensional subspace of $\dual \mbp^n$ of hyperplanes containing the tangent space $T_xX$. It is dual to the map which sends a point $x$ to the moduli point of its tangent space, as an $(r-1)$-dimensional subspace of $\mbp^n$. As in the case of hypersurfaces, there is a natural relation between $\theta_X$, $\psi$, and the based kernel map $\varphi_{d_1, \ldots, d_r}$.

\begin{proposition} \label{prop: psi sends sing locus to planes}
     Let $\mathcal{A}$ be a block-diagonal tensor with $r$ blocks and let $X$ be the associated complete intersection of symmetroids. Let $x$ be a smooth (not necessarily closed) point of $X$. Then the image of $\ker \mathcal{A}(x, \cdot, \cdot)$ under $\psi$ is a linear space in $\dual \mbp^n$ of dimension $(r-1)$. This linear space is precisely the moduli space of hyperplanes in $\mbp^n$ containing $T_x X$, i.e., is the fibre of the tautological bundle over $\theta_X(x)$.
\end{proposition}

\begin{proof}
Let $X_1, \ldots, X_r$ be the symmetroid hypersurfaces defined by the blocks of $\mathcal{A}$, of degrees $d_1, \ldots, d_r$, and let $\varphi_\lindex\: X \rightarrow \mbp^{d_\lindex-1}$ denote the associated kernel maps. Let $\iota_\lindex$ denote the inclusion of $\mbp^{d_\lindex-1}$ into $\mbp^m$.

If $x \in X$ is a smooth point, we may choose lifts $v_\lindex := \iota_\lindex\varphi_\lindex(x)$ in $k(x)^{m+1}$. Note that $\ker \mathcal{A}(x,\cdot, \cdot)$ is spanned by $v_1, \ldots, v_r$ over $k(x)$.  
For $a_1,\dots,a_r \in k(x)$ we have that
	\begin{align*}
        \mathcal{A}\bigg(\,\cdot\,, \sum_{i = 1}^r a_i v_i, \sum_{i = 1}^r a_i v_i\bigg ) = 
        \sum_{i = 1}^r a_i^2 \mathcal{A}(\cdot, v_i, v_i) + \sum_{i\neq j}a_ia_j \mathcal{A}(\cdot, v_i, v_j).
    \end{align*}
By the block-diagonal form of $\mathcal{A}$, we see $\mathcal{A}(\cdot, v_i, v_j) = 0$ for $i \neq j$. Thus $\psi(\ker \mathcal{A}(x, \cdot, \cdot))$ is the $(r-1)$-dimensional plane spanned by $\mathcal{A}(\cdot, v_1, v_1),\dots, \mathcal{A}(\cdot, v_r, v_r)$. The Jacobi formula gives that $\mathcal{A}(\cdot, v_\lindex, v_\lindex)$ is the $\lindex$-th row of the Jacobian matrix of $X$ at $x$, up to scaling.
\end{proof}

\subsection{Singular subschemes of intersections of symmetroids}

In this subsection we prove the main theorem on accidental singularities of intersections of symmetroids. Given a block-diagonal tensor $\mathcal{A} \in k^{n+1} \otimes (\bigoplus_{\lindex =1}^r \Sym_2 k^{d_\lindex})$, its intersection of symmetroids $X$, and associated Cayley variety $\mathfrak{C}$, Theorem~\ref{thm: singularities of symmetroid intersections} already lets us determine the accidental singularities of $X$ in terms of the Cayley variety associated to $\mathcal{A}$. The improvement obtained in this subsection is that this set-theoretic association is upgraded to an explicit finite cover of $\sing X^\nonessential$ by a subscheme of $X^\nonessential \times \mathfrak{C}$ defined by bilinear conditions. The procedure of the proof is to replace global objects with sheaves, and then to work locally. We make essential use of the kernel sheaf on $X^\nonessential$. To begin the proof, we must start with local linear algebra.

\begin{lemma} \label{lem: adjugate is locally defined by the kernel}
	Let $A$ be a symmetric matrix over a local ring $R$ such that $\ker A$ is a free $R$-module of rank $1$ and generated by $y$. Then $yy^T = u \cdot \adj A$ for some $u \in R^\times$.
\end{lemma}

\begin{proof}
	Let $\mathfrak{m}$ be the maximal ideal of $R$. First, $A \cdot \adj A = \det A \cdot I$. On one hand, $y^T A \cdot \adj A = 0 = \det A \cdot  y^T$, which means $\det A = 0$ as the kernel is free and thus has trivial annihilator. Reducing to $R/\mathfrak{m}$ shows that $\adj A$ has at least one entry which is a unit. Freeness then gives us that each column of $\adj A$ is an $R$-multiple of $y$, so by symmetry we may write $\adj A = u \cdot yy^T$ for some $u \in R$. Because at least one entry of $\adj A$ is a unit, we have that $u \in R^\times$.
\end{proof}

\begin{lemma} \label{lem: minors vanish implies kernel}
	Let $r \leq n$ and let $A$ be an $r \times n$ matrix over a local ring $(R, \mathfrak{m})$. If the $r \times r$ minors of $A$ are identically $0$, and the corank of $A \pmod {\mathfrak{m}}$ is $1$, then the left kernel of $A$ is a free $R$-module of rank $1$.
\end{lemma}

\begin{proof}
	Up to changing the order of the rows and columns of $A$, we may assume that the first $(r-1) \times (r-1)$ block has a minor of minimal $\mathfrak{m}$-valuation. Our assumption about the corank of $A \pmod{\mathfrak{m}}$ implies that there is some $g \in \GL_r(R)$ and some $A' \in \GL_{r-1}(R)$ such that
	\[
		gA = 
		\left[
		\begin{array}{c|ccc}
		A' & & B' \\ \hline
		0 & v_{r+1, r+1} & v_{r+1, r+2} & \cdots
		\end{array}
		\right]
	\]
	for some $(r-1) \times (n-r)$ matrix $B'$ and $v_{r+1, r+1}, \ldots, v_{r+1, n} \in R$; namely, one can take $g$ to be a lift of the element of $\GL_r(R/\mathfrak{m})$ sending $A \pmod{\mathfrak{m}}$ to its echelon form. As all of the $r \times r$ minors are identically zero, we have $v_{r+1, r+1} = \ldots = v_{r+1, n} = 0$ and that the kernel of $gA$ is free of rank $1$.
\end{proof}

We may associate to a symmetroid hypersurface $X$ defined by a tensor $\mathcal{A} \in k^{n+1} \otimes \Sym_2 k^{m+1}$ two naturally defined sheaves which restrict to line bundles on $X^\nonessential$. The first is the kernel bundle defined by the exact sequence
	\[
		\xymatrix@C+10pt{
		0 \ar[r] & \mathcal{K} \ar[r] & \mathcal{O}_{X}^{m+1}(m) \ar[r]^-{\mathcal{A}(\mathbf{x}, \cdot, \cdot)} & \mathcal{O}_{X}^{m+1}(m+1) \ar[r] & \mathcal{E}_\lindex(m) \ar[r] & 0.
		}
	\]
The second is the free $\mathcal{O}_{X^\nonessential}$-module $\Adj := \adj \mathcal{A}(\mathbf{x}, \cdot, \cdot) \cdot \mathcal{O}_{X^\nonessential}$. The bundle $\Adj$ has the natural structure of a subbundle of $\mathcal{O}_{X^\nonessential}^{(m+1)^2}(m)$ and $\mathcal{K}^{\otimes 2}$ has the natural structure of a subbundle of $\mathcal{O}_{X^\nonessential}^{(m+1)^2}(2m)$. We can globalize the local relationship of these bundles given by Lemma~\ref{lem: adjugate is locally defined by the kernel}.

\begin{corollary} \label{cor: the projective bundles are equal}
	Let $X$ be a symmetroid hypersurface of degree $m+1$ defined by a tensor $\mathcal{A} \in k^{n+1} \otimes \Sym_2 k^{m+1}$ and let $\mathcal{K}$ be the kernel sheaf defined by $\mathcal{A}$. Then the line bundles $\mathcal{K}^{\otimes 2}$ and $\Adj$ on $X^\nonessential$ define identical morphisms $X^\nonessential \rightarrow \Gr(1, (m+1)^2)$. 
\end{corollary}

\begin{proof}
	Both $\Adj$ and $\mathcal{K}$ define a morphism to $\Gr(1, (m+1)^2)$, which we denote by $f$ and $g$ respectively.
	For any point $x \in X$, we have that the stalks of the rank $1$ bundles $\mathcal{K}^{\otimes 2}$ and $\Adj$ are isomorphic as $\mathcal{O}_{X, x}$-modules by Lemma~\ref{lem: adjugate is locally defined by the kernel}. Thus, $f$ and $g$ agree as morphisms on an open neighbourhood of $x$. Since this is true everywhere locally, $f=g$.
\end{proof}

The nonlinear map $\nu_2\: y \mapsto yy^T$ from $k^{m+1}$ to $k^{(m+1)^2}$ defines a morphism of projective spaces. The image lies in the linear subspace of symmetric matrices of dimension $\binom{m+1}{2}-1$, and the induced morphism $\nu_2\: \mbp^m \rightarrow \mbp^{\binom{m+1}{2}-1}$ is simply the Veronese embedding. The natural analogue of this map when considering tensors with multiple blocks is defined as follows. Let $V := k^{m+1}$, and let $V = \bigoplus_{\lindex=1}^r V_\lindex$ be a fixed direct sum decomposition, with $V_\lindex \subset V$. Then the nonlinear map
\[
	\mapdef{\nu_2'}{V}{V^{\otimes 2}}{v = v_1 + \ldots + v_r}{v_1v_1^T + \ldots + v_rv_r^T}
\]
induces a well-defined morphism of projective spaces $\nu_2'\: \mbp V \rightarrow \mbp V^{\otimes 2}$. The morphism $\nu_2'$ is finite onto its image and of degree $2^{r-1}$; the automorphism group of the covering is generated by the maps $\sigma_\lindex(v_1 + \ldots + v_r) := v_1 + \ldots + (-v_\lindex) +  \ldots + v_r$. With all of the preliminaries in place, we may now prove:

\accidentalsingularitytheorem*

\begin{proof}
	Let $\pi_1\: \mbp^n \times \mbp^{(m+1)^2-1} \rightarrow \mbp^n$ denote the projection onto the first factor, and let $\pi_2\: \mbp^n \times \mbp^{(m+1)^2-1} \rightarrow \mbp^{(m+1)^2-1}$ be the projection onto the second factor.
	We consider the free $\mathcal{O}_{X^\nonessential}$-module of rank $r$ defined by 
	$
		\Adj := \bigoplus_{\lindex=1}^r \gen{\adj \mathcal{A}^{(\lindex)}(\mathbf{x}, \cdot, \cdot)}
	$
	as well as the free $\mathcal{O}_{X^\nonessential}$-module generated by the rows of the Jacobian matrix, i.e., if $f_\lindex := \det \mathcal{A}^{(\lindex)}(\mathbf{x}, \cdot, \cdot)$ then 
	\[
		\mathcal{J}_{X^\nonessential/\mbp^n} := \bigoplus_{\lindex=1}^r \langle{\left(\frac{\der{f_\lindex}}{\der x_0}, \cdots, \frac{\der{f_\lindex}}{\der x_n}\right)} \rangle.
	\]
	By construction, the Jacobian matrix $J(\mathbf{x})$ for $f_1, \ldots, f_r$ is a global section of $\mathcal{J}_{X^\nonessential/\mbp^n}$. We define
	\[
		\begin{tabu}{cccc}
			\psi\: & \Adj & \ra & \mathcal{J}_{X^\nonessential/\mbp^n} \vspace{4pt} \\
			& (M_1, \ldots, M_r) & \mapsto & {
					\begin{bmatrix}
					\trace(M_1 \cdot A_0) & \cdots & \trace(M_1 \cdot A_n) \\
					& \cdots \\
					\trace(M_r \cdot A_0) & \cdots & \trace(M_r \cdot A_n)
					\end{bmatrix}}
		\end{tabu}	
		.
	\]
	By Jacobi's formula, we see that $\psi$ is surjective.
	We denote by $\mbp\!\Adj$ the projective bundle over $X^\nonessential$ associated to $\Adj$, which is a well-defined subbundle of $X^{\nonessential} \times \mbp^{(m+1)^2-1}$ of rank $r$. 
	%
	%%% Proving that $\sing X^{\nonessential}$ is defined by an incidence condition. We should assume that $Z$ is $0$-dimensional for now. The kernel of $J$ may fail to have global sections.
	%
	Let $Z := \sing X^\nonessential$. Since $X$ has no coincident singularities, $\mathcal{J}_{X^\nonessential/\mbp^n}\!\mid_Z$ is a free vector bundle of rank $1$. Define
	\begin{align*}
		\mbp\!\Adj_Z &:= \left\{ (z, M) \in X^\nonessential \times \mbp^{(m+1)^2-1} : M \in \mbp\!\Adj, \ \ \pi_1(z) \in Z \right\}, \qquad \text{and} \\
		W &:= \left\{M \in \mbp^{(m+1)^2-1} : \trace(M \cdot A_j) = 0, \quad j = 0, \cdots, n \right\}.
	\end{align*}
	We claim that $\pi_1$ restricts to an isomorphism $\pi_1\: \mbp\!\Adj_Z \cap \pi_2^{-1}W \rightarrow Z$.
	Let $J(\mathbf{z})$ denote the restriction of $J(\mathbf{x})$ to $Z$. If $Z = \Spec R$ is affine of dimension $0$, it suffices to prove the claim assuming that $R$ is local. Lemma~\ref{lem: minors vanish implies kernel} gives that there is some $v(\mathbf{z}) = (v_1(\mathbf{z}), \ldots, v_r(\mathbf{z})) \in \mathcal{O}_Z^r$ that generates the kernel of $J(\mathbf{z})$. We construct the section $(\mathbf{z}, M(\mathbf{z}))\: \mathbf{z} \mapsto \left(\mathbf{z}, \sum_{\lindex = 1}^r v_\lindex(\mathbf{z}) \cdot \adj\mathcal{A}^{(\lindex)}(\mathbf{z}, \cdot, \cdot) \right)$. We see by $\mathcal{O}_{X^\nonessential}$-linearity of $\psi$ that the image of $(\mathbf{z}, M(\mathbf{z}))$ lies in the kernel of $\psi$. That is, the image lies in $\pi_2^{-1}(W)$. 
	
	Next, observe that $W$ is a linear subspace over $\Spec k$, so its base change to $X^\nonessential$ is a projective bundle. Thus, $\mbp\Adj_Z \cap \pi_2^{-1}(W)$ is an intersection of projective bundles over $Z$. Explicitly, we have that
	\[
		\mbp\Adj_Z \cap \pi_2^{-1}(W) = \{(z, M) \in Z \times \mbp^{(m+1)^2-1} : M = M^T, \ \ \mathcal{A}(z, \cdot, \cdot) \cdot M = 0, \ \ M \in W\}.
	\]
	Notice that the conditions are linear in the entries of $M$. 
	
	Let $z \in Z$ be a point and let $(\mathbf{z}, M'(\mathbf{z}))$ be a point of $\mbp\Adj_Z \cap \pi_2^{-1}(W)$ lying over $\Spec \mathcal{O}_{Z,z}$. As $(\mathbf{z}, M'(\mathbf{z}))$ lies in $\mbp\Adj_Z$, we have that
	\[
		M'(\mathbf{z}) = \sum_{\lindex=1}^r v_\lindex'(\mathbf{z}) \adj \mathcal{A}^{(\lindex)}(\mathbf{z}, \cdot, \cdot)
	\]
	for some $v_\lindex'(\mathbf{z}) \in \mathcal{O}_{Z,z}$ with at least one element a unit. Let $v'(\mathbf{z}) := (v_1'(\mathbf{z}), \cdots, v_r'(\mathbf{z}))$, and observe that $\psi(\mathbf{z}, M'(\mathbf{z})) = v'(\mathbf{z}) \cdot J(\mathbf{z}) = 0$ from the fact that $\psi$ is a morphism of $\mathcal{O}_Z$-modules. As $v'(\mathbf{z})$ lies in the kernel of $J(\mathbf{z})$, which is locally free of rank $1$, we have that $v'(\mathbf{z}) = f(\mathbf{z}) \cdot v(\mathbf{z})$ for some $f(\mathbf{z}) \in \mathcal{O}_{Z, z}^\times$. In particular, $(\mathbf{z}, M'(\mathbf{z}))$ and $(\mathbf{z}, M(\mathbf{z}))$ represent the same point of $\mbp^{(m+1)^2-1}(\mathcal{O}_{Z,z})$. The same argument applies over the residue field of $\mathcal{O}_{Z,z}$. Thus, $\mbp\Adj_Z \cap \pi_2^{-1}(W)$ is a projective bundle over $Z$ whose fibres are zero dimensional. In particular, $\pi_1\: \mbp\Adj_Z \cap \pi_2^{-1}(W) \rightarrow Z$ is an isomorphism with inverse $(\mathbf{z}, M(\mathbf{z}))\: \mathbf{z} \mapsto (\mathbf{z}, M(\mathbf{z}))$.
		
	Next, for $\lindex=1, \ldots, r$ define the bundle $\mathcal{K}_\lindex$ by the exact sequence
	\[
		\xymatrix@C+10pt{
		0 \ar[r] & \mathcal{K}_\lindex \ar[r] & \mathcal{O}_{X}^{d_\lindex}(d_\lindex-1) \ar[r]^-{\mathcal{A}^{(\lindex)}(\mathbf{x}, \cdot, \cdot)} & \mathcal{O}_{X}^{d_{\lindex}}(d_\lindex) \ar[r] & \mathcal{E}_\lindex(d_\lindex) \ar[r] & 0
		}
	\]
	of bundles on $X^\nonessential$. From the block-diagonal structure of $\mathcal{A}$, we define $\mathcal{K} := \bigoplus_{\lindex=1}^r \mathcal{K}_\lindex$ and $\mathcal{H} := \bigoplus_{\lindex=1}^r \mathcal{K}_\lindex^{\otimes 2}$. Observe that the projective bundle $\mbp\mathcal{K}$ is realized as a scheme by
	\[
		\mbp\mathcal{K} = \{(x, y) \in X^\nonessential \times \mbp^m : \mathcal{A}(x, y, \cdot) = 0\}.
	\]	
	By Corollary~\ref{cor: the projective bundles are equal} the two bundles $\mathcal{H}$ and $\Adj$ define identical morphisms from $X$ to $\Gr(r, (m+1)^2)$. Let $\pi_2'\: \mbp^n \times \mbp^m \rightarrow \mbp^m$ be the projection and consider the commutative diagram
	\[
		\xym{
			\mbp\mathcal{K} \ar[d] \ar[r]^{\pi_2'} & \mbp^m \ar[d]^{\nu_2'} \\
			\mbp\mathcal{H} \ar[d]^{\pi_1} \ar[r] & \mbp^{(m+1)^2-1} \\
			 X^\nonessential.
		}
	\]
	%\bigskip
	% Determining the degree
	Because $\trace(yy^T \cdot A) = y^T A y$ for any symmetric matrix $A$ and vector $y$, we have that $\mathfrak{C} = \nu_2'^{-1}(W)$. Since $\nu'_2$ is a finite morphism onto its image of degree $2^{r-1}$, the same is true for the morphism $\mbp\mathcal{K} \rightarrow \mbp\mathcal{H}$. In other words, the pullback of $\mathcal{H} \cap \pi_2^{-1}(W)$ by $\nu_2'$ is defined by
	\[
	\{(x,y) \in X^{\nonessential} \times \mbp^m : \mathcal{A}(x, y, \cdot) = 0, \ \  \mathcal{A}(\cdot, y, y) = 0\} \subseteq \mbp\mathcal{K}
	\]
	whose image is the section over $Z = \sing X^\nonessential$.
\end{proof}

\begin{remark}
	Note that $\DAut(\mathcal{A})$ acts on the incidence variety 
	\[
		\{(x,y) \in X^{\nonessential} \times \mbp^m : \mathcal{A}(x, y, \cdot) = 0,  \ y \in \mathfrak{C}\}
	\]
	and the action of any element induces an automorphism of $\pi_1$.
\end{remark}

%%% SUBSECTION %%%
\section{Fibres of $\psi$, secants, and hyperplane sections}\label{sec: fibres}

	In this section, we comment on how one can study the geometry of hyperplane sections via the fibres of $\psi$ and vice-versa. The morphism $\dual \psi\: \mbp^n \rightarrow \mbp\!\H^0(\mbp^m, \mathcal{O}_{\mbp^m}(2))$ allows us to identify $\mbp^n$ as a linear space of quadrics in $\mbp^m$. Given a linear subspace $H \subset \mbp^n$, we denote by $H^\vee$ the linear subspace in $\dual \mbp^n$ of hyperplanes vanishing on $H$. The fibre of $\psi$ over $H^\vee \in \dual \mbp^n$ is the common zero locus of the subspace $\dual \psi(H)$ of quadrics within $\dual \psi(\mbp^n)$. 
	Viewing $\mathcal{A} \in k^{n+1} \otimes \Sym_2 k^{m+1}$ as a multilinear map, we can use the restriction map $\res\: \H^0(\mbp^n, \mathcal{O}_{\mbp^n}(1)) \rightarrow \H^0(H, \mathcal{O}_H(1))$ to define a new tensor $\mathcal{B} := \mathcal{A}(\pi(\cdot), \cdot, \cdot) \in k^{\dim{H}+1} \otimes \Sym_2 k^{m+1}$ via contraction. The determinantal variety defined by $\mathcal{B}$ allows us to study the fibre of $\psi$ over $H^\vee$. We will say that the tensor $\mathcal{B}$ is the \emph{restriction} of $\mathcal{A}$ to $H$. If $\mathcal{E}$ is the aCM sheaf on $\mbp^n$ associated to $\mathcal{A}$, then $\mathcal{E}\!\mid_H$ is the aCM sheaf on $H$ associated to $\mathcal{B}$.
	
	\begin{proposition} \label{prop: restrictions of tensors}
		Let $\mathcal{A} \in k^{n+1} \otimes (\bigoplus_{\lindex =1}^r \Sym_2 k^{d_\lindex})$ be a block-diagonal tensor, let $X$ the associated intersection of symmetroids, and consider the maps $\psi, \dual \psi$ as in the standard setup. Let $H \subset \mbp^n$ be a linear subspace and let $H^{\vee} \subset \dual \mbp^n$ be the linear subspace parameterizing the hyperplanes containing $H$. Finally, let $\mathcal{B}$ be the restriction of $\mathcal{A}$ to $H$. Then
		\begin{enumerate}[(a)]
					
			\item
				Let $\pi\: \dual{\mathbb{P}}^n \dashrightarrow \mbp\!\H^0(H, \mathcal{O}_H(1))$ denoting the projection away from $H^\vee$ and let $\psi_{\mathcal{B}}\: y \mapsto \mathcal{B}(\cdot, y, y)$. Then
				$
					\psi_{\mathcal{B}} = \pi \circ \psi.
				$
			
			\item
			Let $\mathbf{x}' = (x'_0,\dots,x'_n)$ denote a choice of coordinates for $\mbp\!\H^0(H, \mathcal{O}_H(1))$. The determinantal variety 
				\[
				X' := Z(\det \mathcal{B}(\mathbf{x}', \cdot, \cdot)) \subseteq H
				\]
			is canonically identified with $X \cap H$.

			\item
				The preimage $\psi^{-1}(H^\vee)$ is the common zero locus of the quadrics in $\dual \psi(H)$. Equivalently, $\psi^{-1}(H^\vee)$ is the Cayley variety for $\mathcal{B}$. %In particular, the Cayley variety for $\mathcal{B}$ the fibre over a point in $\dual \mbp^n$ is defined by an intersection of $\dim H + 1$ quadrics.
				
			\item
				If $\mathcal{A}$ has only one block, then the Gauss map associated to $X' \subseteq H$ is given by $\psi_{\mathcal{B}} \circ \varphi_{\mathcal{B}}$, where $\varphi_{\mathcal{B}}$ is the kernel map defined by $\mathcal{B}$.
		\end{enumerate}
		  
	\end{proposition}
	
	\begin{proof}
		Since $\dual \mbp^n = \mbp\!\H^0(\mbp^n, \mathcal{O}_{\mbp^n}(1))$ by definition, we see that parts (a) and (b) follow from the fact that $\mathcal{A}$ is defined by a matrix of forms in $\mathcal{O}_{\mbp^n}(1)$, and $\mathcal{B}$ is defined by a matrix of forms in $\mathcal{O}_H(1)$. Part (c) follows from part (a). If $x \in H$, then $\mathcal{A}(x, \cdot, \cdot) = \mathcal{B}(x, \cdot, \cdot)$. Thus, $\varphi_{\mathcal{B}} = \varphi|_{H}$, and so part (d) follows from parts (a) and~(b).
	\end{proof}
	
	\begin{example}
		We consider the tensor $\mathcal{A}$ defined by
		\[
			\mathcal{A}(\mathbf{x}, \cdot, \cdot) :=
			x_0 \begin{bmatrix} 1 & 0 & 0 \\ 0 & 0 & 0 \\ 0 & 0 & 0 \end{bmatrix}
			+ x_1 \begin{bmatrix} 0 & 0 & 0 \\ 0 & 1 & 0 \\ 0 & 0 & 0 \end{bmatrix}
			+ x_2 \begin{bmatrix} 0 & 0 & 0 \\ 0 & 0 & 0 \\ 0 & 0 & 1 \end{bmatrix}
			+ x_3 \begin{bmatrix} 1 & 1 & 1 \\ 1 & 1 & 1 \\ 1 & 1 & 1 \end{bmatrix}.
		\]
		The determinantal variety $X_3 \subset \mbp^3$ is a Cayley cubic, and the Cayley variety in $\mbp^2$ is empty. Let $H = Z(x_3) \subset \mbp^3$. We see that $H^\vee$ is the single point $[0:0:0:1] \in \dual \mbp^3$, and we can identify the coordinate ring of $H$ with $k[x_0', x_1', x_2']$. We see
		\[
			\mathcal{B}(\mathbf{x}', \cdot, \cdot) =
			x_0' \begin{bmatrix} 1 & 0 & 0 \\ 0 & 0 & 0 \\ 0 & 0 & 0 \end{bmatrix}
			+ x_1' \begin{bmatrix} 0 & 0 & 0 \\ 0 & 1 & 0 \\ 0 & 0 & 0 \end{bmatrix}
			+ x_2' \begin{bmatrix} 0 & 0 & 0 \\ 0 & 0 & 0 \\ 0 & 0 & 1 \end{bmatrix}.
		\]
		The determinantal variety defined by $\mathcal{B}$ is $Z(x_0'x_1'x_2')$, which is precisely the curve obtained by intersecting $X_3$ and $H$. The preimage $\psi^{-1}(H^\vee)$ is the fibre of $\psi$ over a single point, and is the intersection of the three quadrics
		\[
			\mathbf{y}^T \begin{bmatrix} 1 & 0 & 0 \\ 0 & 0 & 0 \\ 0 & 0 & 0 \end{bmatrix} \mathbf{y}, \quad 
			\mathbf{y}^T\begin{bmatrix} 0 & 0 & 0 \\ 0 & 1 & 0 \\ 0 & 0 & 0 \end{bmatrix} \mathbf{y}, \quad 
			\mathbf{y}^T\begin{bmatrix} 0 & 0 & 0 \\ 0 & 0 & 0 \\ 0 & 0 & 1 \end{bmatrix} \mathbf{y} \ \  \subset \mbp^2
		\]
		which in this case is empty, as $[0:0:0:1]$ does not lie in the image of $\psi$. Finally, we see that the dual variety of $Z(x_0'x_1'x_2')$ in $\dual \mbp^2$ is the set of three points $\{[1:0:0], [0:1:0], [0:0:1]\}$. The image of $Z(x_0'x_1'x_2')$ under the kernel map is a set of three points, and it is easy to check $\psi_{\mathcal{B}} \circ \varphi_{\mathcal{B}}$ is the Gauss map for $Z(x_0'x_1'x_2') \subset \mbp^2$.
	\end{example}

	In particular, we can relate singularities of a hyperplane section of a determinantal variety with fibres of $\psi$ over the corresponding point.

%%% SUBSECTION %%%
\subsection{Secants of Cayley varieties of dimension zero}
	
	Here, we show how the Cayley variety associated to a tensor $\mathcal{A}$ is related to interesting hyperplane sections of the intersection of symmetroids determined by $\mathcal{A}$. First, we look at a necessary and sufficient condition for a fibre of $\psi$ to contain a line.

	\begin{lemma}\label{lem: cayley line contracts}
		Let $\mathfrak{C}$ be the Cayley variety associated to a tensor $\mathcal{A}$. Let $\ell$ be a secant line of $\mathfrak{C}$ not contained in $\mathfrak{C}$. Then $\psi(\ell)$ is a single point in $\dual \mbp^n$. Conversely, if $\ell$ is a line such that $\psi(\ell)$ is a single point, then it is a secant line of the Cayley variety.
	\end{lemma}
	
	\begin{proof}
		We first consider the case where the two intersection points of $\ell$ with the Cayley variety are distinct. Let $p, q \in k^{m+1}$ be two vectors representing the intersection points of $\ell$ with $\mathfrak{C}$. Then
			$
	            \ell = \mbp \{ a p + b q : a,b \in k \}.
	        $
	    For $(a:b) \in \mbp^1$ we have
	        \begin{align*}
	            \mathcal{A}(\cdot, a p + b q, a p + b q) &=
	            a^2\mathcal{A}(\cdot, p, p) + 2ab\mathcal{A}(\cdot, p, q) + b^2\mathcal{A}(\cdot, q, q) \\
	            &= 2ab\mathcal{A}(\cdot, p, q)
	        \end{align*}
		which defines a constant linear form. We now consider the case when $\ell$ is tangent to $\mathfrak{C}$. Let $p \in k^{m+1}$ represent the point of tangency, and let $v \in k^{m+1}$ be a vector such that $\ell = \mbp\{ap + bv : a,b \in k\}$. Since $\ell$ is tangent to $\mathfrak{C}$ at $p$, we have the relation
	        \begin{align*}
	            \mathcal{A}(\cdot, p + \epsilon v, p + \epsilon v) &\equiv
	            2 \epsilon \mathcal{A}(\cdot, p, v)  \equiv 0 \pmod {\epsilon^2}
	        \end{align*}
	    over the ring $k[\epsilon]/\gen{\epsilon^2}$. But then $\mathcal{A}(\cdot, p, v) = 0$, so $\psi$ is constant along $\ell$.
	    
	    We now prove the converse claim. Let $\ell$ be a line in $\mbp^m$ not contained in $\mathfrak{C}$ where $\psi(\ell)$ is a single point. Up to linear transformations of $\dual \mbp^n$ we may assume that $\psi(\ell) = [1:0:\ldots:0]$. Thus, if $A_0$ is the $0$-th slice of $\mathcal{A}$ and $Q_0$ the corresponding quadric in $\mbp^m$, we see $\ell$ is not contained in $Q_0$. As $Q_0$ is a hypersurface of degree $2$, we have that $\ell$ meets $Q_0$ in two points (counting multiplicity). Observe that the intersection $\ell \cap \mathfrak{C}$ is precisely where $\ell$ meets $Q_0$. In particular, either $\ell$ is tangent to $\mathfrak{C}$ at a single point or $\ell \cap Q_0$ consists of two distinct points of $\mathfrak{C}$.
	\end{proof}

The fact that secants of the Cayley variety are constant under $\psi$ gives rise to a remarkable coincidence when the Cayley variety is a complete intersection of dimension $0$. Consider a tensor $\mathcal{A} \in k^{n+1} \otimes \Sym_2 k^{m+1}$ defining a determinantal hypersurface $X$ whose associated Cayley variety $\mathfrak{C}$ has a secant line $\ell \not \subseteq \mathfrak{C}$. If $H$ is the hyperplane defined by $\psi(\ell)$, Proposition~\ref{prop: restrictions of tensors} allows us to study the intersection $X \cap H$ as a determinantal variety in its own right, specifically, the determinantal variety defined by the restriction of $\mathcal{A}$ to $H$, which we denote by $\mathcal{B}$. Proposition~\ref{prop: restrictions of tensors}(c) gives that the Cayley variety for $\mathcal{B}$ is the fibre of $\psi$ over $\psi(\ell)$. If we change coordinates of $\mbp^m$ so that $\ell$ is the line between $[1 : 0 : \cdots : 0]$ and $[0 : 1 : \cdots : 0]$, we see that $\mathcal{B}$ has the form
	\[
		\mathcal{B}(\mathbf{x}, \cdot, \cdot) =
		\left[
		\begin{array}{c|c}
		0 & M(\mathbf{x}) \\
		\hline
		\rule{0pt}{1.1\normalbaselineskip}  M(\mathbf{x})^T & N(\mathbf{x})
		\end{array}
		\right]
	\]
for some matrices $M(\mathbf{x}), N(\mathbf{x})$ with entries in $\mathcal{O}_{\mbp^{n-1}}(1)$ of dimensions $2 \times (m-1)$ and $(m-1) \times (m-1)$ respectively. Another way of stating this is that $\ell$ is a common isotropic subspace to the quadrics formed from the slices of $\mathcal{B}$. If $p \in \mbp^{n-1}$ is a point such that $\rank M(p) \leq 1$, then some linear combination of the first two columns of $\mathcal{B}(p, \cdot, \cdot)$ is zero, viz., there is a point $q \in \ell \subseteq \mathfrak{C}_{\mathcal{B}}$ such that $\mathcal{B}(p, q, \cdot) = 0$. By Theorem~\ref{thm: singularities of symmetroid intersections}, we see that if $p \in X(\bar k)$ then it must be a singular point. In particular, counting the number of singular points on a hyperplane section defined by $\psi(\ell)$ comes down to counting the number of points in $\{p \in \mbp^{n-1} : \rank M(p) \leq 2\}$.

\begin{proposition} \label{prop: deg and dim of tangents of symmetric determinantal hypersurfaces}
	Let $\mathcal{A} \in k^{n+1} \otimes \Sym_2 k^{n+2}$ be a tensor, let $X$ be the associated determinantal hypersurface, and let $\ell$ be a secant line of the Cayley variety of $\mathcal{A}$. Then the hyperplane $H$ defined by $\psi(\ell)$ is tangent to $X$ along a scheme of dimension at least $0$. Furthermore, if $X \cap H$ contains no essential singularities of $X$ and the dimension of $\sing(X \cap H)$ is $0$, then $\deg \sing(X \cap H) = n+1$.
\end{proposition}

\begin{proof}
	Let $m = n+1$ and let $\mathcal{B}$ be the restriction of $\mathcal{A}$ to $H$. Let $L \subset k^{m+1}$ denote the space such that $\mbp\!L = \ell$. Viewing $\mathcal{B}$ as a multilinear map, we define the tensor $\mathcal{M}\: k^{n+1} \otimes L \otimes k^{m+1} \rightarrow k$ by restriction of $\mathcal{B}$. Note that $\mathcal{M}$ restricted to $k^{n+1} \otimes L \otimes L$ is identically $0$ by the fact that $\ell \subseteq \mathfrak{C}_{\mathcal{B}}$. Thus, we obtain a multilinear map $\mathcal{M}'\: k^{n+1} \otimes L \otimes (k^{m+1}/L) \rightarrow k$.
	
	The subscheme of $\mbp\mathrm{M}_{2 \times m-1}$ of $2 \times (m-1)$ matrices of corank $1$ is the image of the Segre embedding
	\[
		\mapdef{\mathrm{seg}}{\mbp^1 \times \mbp^{m-2}}{\mbp\mathrm{M}_{2 \times m-1}}{(u,v)}{uv^T}.
	\]
	The image of this Segre embedding has degree $m-1$ and codimension $m-2$. Consider the linear space $Y := \{\mathcal{M}'(p, \cdot, \cdot) : p \in H\}$. If $\dim Y < n-1$, then there is a $p$ such that $\mathcal{M}'(p, \cdot, \cdot) = 0$. This implies that $p$ is an essential singularity of $X$. Otherwise, we see that $Y$ meets the image of $\mathrm{seg}$ along a positive dimensional scheme or a scheme of dimension $0$ and degree $n$. By Theorem~\ref{thm: singularities of symmetroid intersections}, we see that $\{p \in H : \rank \mathcal{M}'(p, \cdot, \cdot) = 1\}$ is a subscheme of $\sing(X \cap H)$. If $X \cap H$ contains no essential singularities, then we have equality.
\end{proof}

The remainder of this section is devoted to establishing an analogous result for an intersection of symmetroids. The main difference in the proof is that the calculation to identify the singularities of a fibre of $\psi$ along a secant line is more technical. We give an example illustrating the nuances of this calculation. 

\begin{example} \label{ex: genus 4 secant contraction example}
Let $k$ be algebraically closed and let $\mathcal{A} \in k^4 \otimes (\Sym_2 k^2 \oplus \Sym_2 k^3)$ be a generic block-diagonal tensor. Let $\ell$ be a secant line of $\mathfrak{C}$ which is not of the form $\{p, \sigma(p)\}$ for some $\sigma \in \DAut(\mathcal{A})$. Let $[H] = \psi(\ell)$. Using the actions of $\GL_4 \times \GL_2 \times \GL_3$, we may assume $H = Z(x_0)$ and $\ell$ is spanned by $[1 \colon 0 \colon 1 \colon 0 \colon 0]$ and $[0 \colon 1 \colon 0 \colon 1 \colon 0]$. Let $\mathcal{B} = \mathcal{A}_{\mid H}$. Because $H = Z(x_0)$ and $\ell \subseteq \mathfrak{C}_{\mathcal{B}}$, we have
\[
	\mathcal{B}(\mathbf{x}, \cdot, \cdot) := 
	\begin{bmatrix}
	a & c & 0 & 0 & 0\\
	c & b & 0 & 0 & 0 \\
	0 & 0 & -a & -c & d \\
	0 & 0 & -c & -b & e \\
	0 & 0 & d & e & f
	\end{bmatrix}
\]
for $a,b,c,d,e,f \in k[x_1,x_2,x_3]$ homogeneous linear forms. Notice for $p \in H$ that $\ker \mathcal{B}(p, \cdot, \cdot) \cap \ell \neq \emptyset$ if and only if the points
\begin{alignat*}{3}
	\mathcal{B}(p, \cdot, [1 \colon 0 \colon 1 \colon 0 \colon 0]) &= [a(p) \colon c(p) \colon -a(p) \colon -c(p) \colon d(p)] && \quad \text{and}\\
	\mathcal{B}(p, \cdot, [0 \colon 1 \colon 0 \colon 1 \colon 0]) &= [c(p) \colon b(p) \colon -c(p) \colon -b(p) \colon e(p)]
\end{alignat*}
are equal. This occurs if and only if 
$
\rank \begin{bmatrix}
a(p) & c(p) & d(p) \\
c(p) & b(p) & e(p)
\end{bmatrix} \leq 1.
$
The slices of the $2 \times 3$ matrix of linear forms span a plane in $\mbp\mathrm{M}_{2 \times 3}$ of dimension $3$. Despite the relation among the entries of this matrix, a generic such plane will meet the rank degeneracy locus (the image of a Segre embedding) transversely along a subscheme of degree $3$ and dimension $0$. 

Alternatively, if $\ell$ is the secant between $p, \sigma(p) \in \mathfrak{C}$ and $p \neq \sigma(p)$, we can use the action of $\GL_2 \times \GL_3$ to move $p$ to $[1:0:1:0:0]$, which implies $\sigma(p)$ is sent to $[1:0:-1:0:0]$. As before, we may apply $\GL_4$ to assume that $H = Z(x_0)$ and set $\mathcal{B} = \mathcal{A}_{\mid H}$. In this case,
\[
	\mathcal{B}(\mathbf{x}, \cdot, \cdot) := 
	\begin{bmatrix}
	0 & c & 0 & 0 & 0\\
	c & b & 0 & 0 & 0 \\
	0 & 0 & 0 & g & d \\
	0 & 0 & g & h & e \\
	0 & 0 & d & e & f
	\end{bmatrix}
\]
for $b,c,d,e,f,g,h \in k[x_1,x_2,x_3]$ homogeneous linear forms. In this case, we see that $H$ is tangent to the quadric cone defined by $Z(x_0b - c^2)$ along the line $Z(x_0, c)$.
\end{example}

Let $\mathcal{A}$ be a block-diagonal tensor and $X$ the associated intersection of symmetroids. Unlike the case of hypersurfaces, it is possible that
\[
	\{(x, y) \in X \times \mbp^m : \mathcal{A}(x, y, \cdot) =0, \ \ y \in \mathfrak{C}\} \neq \{(x, y) \in \mbp^n \times \mbp^m : \mathcal{A}(x, y, \cdot) =0, \ \ y \in \mathfrak{C}\}.
\]
We give an example of such a case in Example~\ref{ex: intersection of cubic curves}.
Theorem~\ref{thm: singularities of symmetroid intersections}(b) shows us that the difference is explained by the block spaces. This subtlety is why we assume $\ell$ does not intersect the block spaces in Proposition~\ref{prop: deg and dim of tangents of intersections of symmetric determinantal hypersurfaces}.

\begin{example} \label{ex: intersection of cubic curves}
	Consider the intersection $X$ of symmetric determinantal cubic curves in $\mathbb{P}^2$ defined by
	\[
		\mathcal{B}(\mathbf{x}, \cdot, \cdot) := 
		\begin{bmatrix}
		0 & b & c & 0 & 0 & 0\\
		b & d & e & 0 & 0 & 0\\
		c & e & f & 0 & 0 & 0 \\
		0 & 0 & 0 & 0 & b' & c' \\
		0 & 0 & 0 & b' & d' & e' \\
		0 & 0 & 0 & c' & e' & f'
		\end{bmatrix}
	\]
	for $b, \ldots, f' \in k[x_1,x_2,x_3]$ homogeneous linear forms. With $q := [1:0:0:1:0:0]$ and $\sigma$ the nontrivial element of $\DAut(\mathcal{B})/\gen{\pm I}$, we see that the line between $q$ and $\sigma(q)$ lies inside $\mathfrak{C}$. The point $p$ such that $b(p) = c(p) = 0$ is a point on the hypersurface defined by the first $3 \times 3$ block, but it is not necessarily a point on $X$. For a general choice of linear forms, $X$ is a collection of $9$ reduced points in $\mbp^2$.
\end{example}

\secantsandhyperplanestheorem*

\begin{proof}
	Let $\mathcal{B}$ be the restriction of $\mathcal{A}$ to $H$ and let $L \subset k^{m+1}$ denote the space such that $\mbp\!L = \ell$. Viewing $\mathcal{B}$ as a multilinear map, we define the tensor $\mathcal{M}\: k^{n+1} \otimes L \otimes k^{m+1} \rightarrow k$ by restriction of $\mathcal{B}$. Note that $\mathcal{M}$ restricted to $k^{n+1} \otimes L \otimes L$ is identically $0$ by the fact that $\ell \subseteq \mathfrak{C}_{\mathcal{B}}$. Thus, we obtain a multilinear map $\mathcal{M}'\: k^{n+1} \otimes L \otimes (k^{m+1}/L) \rightarrow k$.
	
	The subscheme of $\mbp\mathrm{M}_{2 \times m-1}$ of $2 \times (m-1)$ matrices of corank $1$ is the image of the Segre embedding
	\[
		\mapdef{\mathrm{seg}}{\mbp^1 \times \mbp^{m-2}}{\mbp\mathrm{M}_{2 \times m-1}}{(u,v)}{uv^T}.
	\]
	The image of this Segre embedding has degree $m-1$ and codimension $m-2$. Consider the linear space $Y := \{\mathcal{M}'(p, \cdot, \cdot) : p \in H\}$. Because $\ell$ does not meet the block spaces, if $\mathcal{M}'(p, q, \cdot) = 0$ for some $q \in \ell$, then $p$ is a singular point of $X$ by Theorem~\ref{thm: singularities of symmetroid intersections}. 
		
	We consider the embedding of $L$ into $\bigoplus_{\lindex=1}^r k^{d_\lindex}$ given by each of the projections. If $\dim Y < n-1$, then there is a $p$ such that $\mathcal{M}'(p, \cdot, \cdot) = 0$. We then see that $\mathcal{M}(p, \cdot, \cdot)$ is itself the zero map, and each restriction $\mathcal{M}(p, \cdot, \cdot)\: L \otimes k^{d_\lindex} \rightarrow k$ is also the zero map. In particular, each block of $\mathcal{A}(p, \cdot, \cdot)$ of size at least $2$ has corank at least $2$, so $X$ has an essential singularity at $p$. 
	If $\dim(Y \cap \Im(\mathrm{seg})) = 0$, then $\deg(Y \cap \Im(\mathrm{seg})) = n$. If $\dim \sing(X \cap H) = 0$, then Theorem~\ref{thm: singularities of symmetroid intersections with multiplicity} shows that $Y \cap \Im(\mathrm{seg})$ is contained in $\sing(X \cap H)$. Otherwise, $\sing(X \cap H)$ is positive dimensional, and so contains a subscheme of dimension $0$ and degree $n$. 
\end{proof}

In the specific cases covered in the remainder of this paper -- canonical curves of genera $3,4,5$ and symmetroid quintics in $\mbp^3$, the conclusion of the preceding proposition is sharp, in the sense that if $X$ is a generic such object and $H$ is a hyperplane arising from a secant of the Cayley variety, then $\sing(X \cap H)$ is of dimension $0$ and degree $n$, that is, $H$ is a hyperplane tangent to $X$ at $n$ points. In the case that $\sing(X \cap H)$ is reduced, we can determine the multiplicity of the intersection of $X$ and $H$. The discrepancy between the degree of a fat point $Z \subseteq \sing(X \cap H)$ and the multiplicity of $Z$ in the intersection is why we need to use a genericity hypothesis in the case of a quintic symmetroid. For canonical curves, we can remove the genericity hypothesis.

\begin{proposition} \label{prop: tangency for curves}
	Let $\mathcal{A} \in k^{n+1} \otimes (\bigoplus_{\lindex=1}^r \Sym_2 k^{d_\lindex})$ be a block-diagonal tensor, let $X$ be the associated intersection of symmetroids, and let $\mathfrak{C}$ be the associated Cayley variety. Assume that $X$ is a smooth complete intersection of dimension $1$ and that $\mathfrak{C}$ is a complete intersection of dimension $0$. Then:
	\begin{enumerate}[(a)]
		\item $X$ is either a canonical genus $3$ curve, a canonical genus $4$ curve lying on a cone, or a canonical genus $5$ curve given as the intersection of three quadric cones each of rank $3$.
		
		\item
		Let $\ell \subseteq \mbp^m$ be a secant line of $\mathfrak{C}$ and let $[H] = \psi(\ell)$. Then $H$ is tangent to $X$ at each point of $X \cap H$.
		
		\item
		Write $\mbp^m = \mbp(V_1 \oplus \ldots \oplus V_r)$ such that $\mbp V_\lindex$ is the $\lindex$-th block space for $\mathcal{A}$. If $\mathcal{A}^{(\lindex)}$ is a block of $\mathcal{A}$ of size $2$ and $X_\lindex$ the corresponding quadric in $\mbp^n$, then $H$ is tangent to $X_\lindex$ if and only if $\ell = \ell(p, \sigma_\lindex(p))$, where $\sigma_\lindex \in \DAut(\mathcal{A})$ acts on $V_\lindex$ by the identity and on $V_{\lindex'}$ by $(-1)$ for $\lindex' \neq \lindex$.  
	\end{enumerate}
\end{proposition}

\begin{proof}
	The hypotheses imply that $r = n-1 = m-2$. However, each $d_\lindex \geq 2$ and $d_1 + \ldots + d_r = m+1$. The three cases identified in part (a) are the only ones possible. 

	By Proposition~\ref{prop: psi sends sing locus to planes}, if $S \subseteq (X \bs \sing X)$ is some subscheme, then 
	\[
		\bigcap_{x\in S} \psi(\ker \mathcal{A}(x, \cdot, \cdot))
	\]
	is the locus in $\dual \mbp^n$ of hyperplanes in $\mbp^n$ tangent to $X$ along $S$. 
	
	Let $x \in (X \cap H)(\bar k)$, and let $Q_x := \dual \psi(x)$. Since $[H] = \psi(\ell)$, we have that $Q_x$ is a quadric in $\mbp^m$ which contains $\ell$ as an isotropic subspace and $\dim \sing Q_x = r-1$. Because the dimension of a maximal isotropic subspace of $Q_x$ is $m-2$, we see that $\ell$ must intersect $\ker \mathcal{A}(x, \cdot, \cdot)$. Thus, $\psi(\ell) \in \psi(\ker \mathcal{A}(x, \cdot, \cdot))$. In particular, 
	\[
		\psi(\ell) \in \bigcap_{x\in (X \cap H)(\bar k)} \psi(\ker \mathcal{A}(x, \cdot, \cdot))
	\]
	which proves part (b).
	
	Finally, let $\ell$ be a secant of $\mathfrak{C}$ such that $H$ is tangent to $X_\lindex$. Let $x \in (X \cap H)(\bar k)$. Because $X$ is smooth, we have that $\mathfrak{C}$ is smooth and hence $\ker \mathcal{A}(x, \cdot, \cdot) \cap \mathfrak{C} = \emptyset$. On the other hand, $\ell$ meets $\ker \mathcal{A}(x, \cdot, \cdot)$ at a single point $y$. We have that $\psi$ induces a morphism $\psi\: \ker \mathcal{A}(x, \cdot, \cdot) \rightarrow \dual \mbp^n$ that is finite of degree $2^{r-1}$ onto its image. We have $\psi(y) = \psi(\varphi_\lindex(x))$, and since $\psi$ is branched at $\varphi_\lindex(x)$, we see $y = \varphi_\lindex(x)$. If $p \in \ell$ is one of the points on $\mathfrak{C}$, we see that $y$ is the projection of $p$ to $\mbp V_\lindex$, else $X_\lindex$ is reducible. Thus, $p$ and $\sigma_\lindex(p)$ span $\ell$. The converse for part (c) is straightforward.
\end{proof}

\begin{remark} \label{rem: converse for sigma fixed lines}
	We can say something more general in the case of Proposition~\ref{prop: tangency for curves}(c). Namely, if $\ell := \ell(p, \sigma_l(p))$ is a secant of the Cayley variety, then the hyperplane $H$ defined by $[H] := \psi(\ell)$ is tangent to $X_l$. This is true for any block-diagonal tensor $\mathcal{A} \in k^{n+1} \otimes (\bigoplus_{\lindex=1}^r \Sym_2 k^{d_\lindex})$ satisfying our standard conventions.
\end{remark}

\begin{corollary}
\label{cor: lines are n-tangents for our main examples}
	Let $X$ denote the intersection of symmetroids defined by a tensor $\mathcal{A}$ from one of the following cases:
	\newcommand{\enumalignspace}[1]{\makebox[5cm][l]{#1}}
	\begin{enumerate}[(a)]
		\item \enumalignspace{$\mathcal{A} \in k^3 \otimes \Sym_2 k^4$} (and $X$ is a smooth canonical genus $3$ curve),
		\item \enumalignspace{$\mathcal{A} \in k^4 \otimes (\Sym_2 k^2 \oplus \Sym_2 k^3)$} (and $X$ is a smooth canonical genus $4$ curve),
		\item \enumalignspace{$\mathcal{A} \in k^5 \otimes (\Sym_2 k^2)^{\oplus 3}$} (and $X$ is a smooth canonical genus $5$ curve),
		\item \enumalignspace{$\mathcal{A} \in k^4 \otimes \Sym_2 k^5$} (and $X$ is a generic quintic symmetroid).
	\end{enumerate}
	Let $\ell$ be a secant line of the Cayley variety of $\mathcal{A}$ not intersecting the block spaces. Then for $[H] := \psi(\ell)$, we have in cases (a-c) that $X\cap H = 2D$, where $D$ is an effective representative of a theta characteristic of $X$. In case $(d)$, we have that $H$ is a tritangent hyperplane to $X$.
\end{corollary}

\begin{comment}
\comm{
The proof would need to update in the following ways if we want to eliminate genericity:

\bigskip
In cases (a-c), we need to use the smoothness hypothesis to use Prop 5.1.5. \\
Then, we need to use the lemma above to say this perfectly accounts for the contact points. \\
Normally, Bezout tells us the rest, but it may be more subtle than this in general. \\
\bigskip
We can likely use fat points in Prop 5.1.6 if that helps.
}

\begin{proof}
	It suffices to show that $\sing(X \cap H)$ is reduced, of dimension $0$, and of degree $n$ for a generic $X$. In turn, it suffices to exhibit a particular example where this is the case as the condition is open. This is easily done on a computer.
\end{proof}
\end{comment}

\begin{proof}
	For a generic $X$, it suffices to show that $\sing(X \cap H)$ is reduced, of dimension $0$, and of degree $n$. In turn, it suffices to exhibit a particular example where this is the case. This is easily done.
	
	In cases (a-c), we now show that the result from the generic case implies the result for any smooth $X$. We fix our attention to case (a), the others being similar. Up to linear transformations and base extension, we may assume that $\ell$ is the secant between the two points $[1:0:0:0], [0:1:0:0]$ of the Cayley variety. Let $\mathcal{A}$ be a tensor defining a smooth curve $X$. We may choose $\mathcal{B}$ to be a generic tensor of the same format as $\mathcal{A}$ and such that the distinguished points $[1:0:0:0], [0:1:0:0]$ also lie in the Cayley variety of $\mathcal{B}$. We consider the pencil of tensors $\mathcal{A}_t := \mathcal{A} + t \cdot \mathcal{B}$ over $\Spec k[t]_{(t)}$, which defines a family of curves $\{X_t\}$ and hyperplanes $\{H_t\}$. We obtain a proper flat family of intersections $f\: \{X_t \cap H_t\} \rightarrow \Spec k[t]_{(t)}$, whose coordinate ring $R$ is a finite flat (thus free) $k[t]_{(t)}$-algebra. By genericity of $\mathcal{B}$, the intersection of the generic fibre of $X_t$ and the generic fibre of $H_t$ is a pair of double points. Thus, by \cite[Section~1B]{Thiel2018} we have that the multiplicities of the geometric points of the special fibre of $f$ are all even. (Informally, if two points of even multiplicity degenerate together, the limit point also has even multiplicity.)
\end{proof}

In these cases, we are particularly interested in the number of hyperplanes arising from secants of the Cayley variety. For a block-diagonal tensor $\mathcal{A}$ with $r$ blocks, we have that $\psi(g(\ell)) = \psi(\ell)$ for every $g \in \Aut(\mathcal{A})$. The converse is also true under a technical assumption on the blocks $\mathcal{A}$, which happens to be satisfied for all of the cases discussed in Corollary~\ref{cor: lines are n-tangents for our main examples}. Lemma~\ref{lem: psi is injective on secants mod diagonal symmetries} allows us to say that the images of secant lines of the Cayley variety under $\psi$ are distinct up to $\Aut(\mathcal{A})$.

\begin{lemma} \label{lem: psi is injective on secants mod diagonal symmetries}
	Let $\mathcal{A} \in k^{n+1} \otimes (\bigoplus_{\lindex=1}^r \Sym_2 k^{d_\lindex})$ be a block-diagonal tensor, let $X$ be the associated intersection of symmetroids, and let $\mathfrak{C}$ be the Cayley variety. Assume that $X$ is a complete intersection, $\mathfrak{C}$ is a complete $0$-dimensional intersection, and that $\sum_{\lindex \in I} d_\lindex \neq \frac{m+1}{2}$ for any $I \subset \{1, \ldots, r\}$. Let $\ell, \ell'$ be two secants of $\mathfrak{C}$ such that $\psi(\ell) = \psi(\ell')$ and let $[H] := \psi(\ell)$. If $X \cap H$ has no essential singularities or coincident accidental singularities, $\dim \sing(X \cap H) = 0$, and $\sing(X \cap H)$ contains at least two distinct geometric points, then there is some $\sigma \in \DAut(\mathcal{A})$ such that $\ell' = \sigma(\ell)$.
\end{lemma}
 
\begin{proof}
	Let $\mathcal{B}$ be the restriction of $\mathcal{A}$ to $H$. We see from Theorem~\ref{thm: singularities of symmetroid intersections with multiplicity} that 
	\[
		 Z := \{(x,y) \in (X\cap H) \times \mbp^m : \mathcal{B}(x, y, \cdot) = 0, \mathcal{B}(\cdot, y, y) = 0\}
	\]
	is a finite cover of $\sing(X \cap H)$ of degree $2^{r-1}$ via the projection to $X$ and that $\DAut(\mathcal{B})/\gen{\pm I}$ act transitively on the fibres of $Z \rightarrow X$. 
	
	Let $L,L' \subset k^{m+1}$ represent $\ell, \ell'$, and let $\mathcal{M}\: k^{n} \otimes L \otimes k^{m+1} \rightarrow k$, $\mathcal{M}'\: k^{n} \otimes L' \otimes k^{m+1} \rightarrow k$ denote the restriction of the multilinear map $\mathcal{B}$ to $L, L'$ (respectively).
	As we have seen before, because $\ell$ is contained in the Cayley variety for $X \cap H$, we have that the degeneracy loci
	\[
		\{(x,y) \in (X\cap H) \times \ell : \mathcal{M}(x, y, \cdot) = 0\}, \quad \{(x,y) \in (X \cap H) \times \ell' : \mathcal{M}'(x, y, \cdot) = 0\}
	\]
	are both subschemes of $Z$. 
	Choose some distinct points $p_1, p_2 \in \sing(X \cap H)$. Then we may find $q_1, q_2 \in \ell$, and $q_1', q_2' \in \ell'$ such that
	$
		\mathcal{M}(p_j, q_j, \cdot) = \mathcal{M}'(p_j, q_j', \cdot) = 0
	$
	for $j = 1,2$.
	Up to applying an automorphism of $\mathcal{B}$, we may assume that $q_1 = q_1'$. We also have that $q_2' = \sigma(q_2)$ for some $\sigma \in \DAut(\mathcal{B})$. 
	
	If $r = 1$, then $\DAut(\mathcal{B}) = \gen{\pm I}$ and so $\ell = \ell'$. Otherwise, suppose for the sake of contradiction that $\sigma \not \in \gen{\pm I}$. We endow $k^{n+1}, k^{m+1}$ with coordinates and write
	\[
		\mathcal{B} :=
		\begin{bmatrix}
			\mathcal{B}^{(1)} & 0 \\
			0 & \mathcal{B}^{(2)}
		\end{bmatrix},
		\qquad 
		\sigma =
		\begin{bmatrix}
			I & 0 \\
			0 & -I
		\end{bmatrix},	
	\]
	where either $\mathcal{B}^{(1)}$ or $\mathcal{B}^{(2)}$ may be further decomposable. We may write $q_2 = (q_2^{(1)}, q_2^{(2)})$ and $\sigma(q_2) = (q_2^{(1)}, -q_2^{(2)})$. By definition of $\mathcal{M}, \mathcal{M}'$ we see
	\begin{align*}
		\mathcal{M}(\cdot, q_1, q_2) &= \mathcal{B}^{(1)}(\cdot, q_1^{(1)}, q_2^{(1)}) + \mathcal{B}^{(2)}(\cdot, q_1^{(2)}, q_2^{(2)}) = 0 \\
		\mathcal{M}'(\cdot, q_1, q_2) &= \mathcal{B}^{(1)}(\cdot, q_1^{(1)}, q_2^{(1)}) - \mathcal{B}^{(2)}(\cdot, q_1^{(2)}, q_2^{(2)}) = 0. 
	\end{align*}
	Thus, both terms are $0$, and by symmetry of $\mathcal{B}$ we have $\mathcal{B}^{(1)}(\cdot, q_2^{(1)}, q_1^{(1)}) = \mathcal{B}^{(2)}(\cdot, q_2^{(2)}, q_1^{(2)}) = 0$.
	
	If $k^{m+1} = V_1 \oplus V_2$ is the eigenspace decomposition for $\sigma$, we have by the condition of the block sizes that (without loss of generality) $2\dim V_1 < m+1$. Additionally, by considering the projection $\pi(L)$ of $L$ to $V_1$, we have that $\mathcal{B}^{(1)}$ defines a linear map $\dual \psi_{\mathcal{B}^{(1)}}\: k^{n} \rightarrow (\pi(L) \otimes V_1)^\vee$. The image lies in a space of dimension $2\dim V_1 - 2$. Since $\mathfrak{C}$ is of dimension $0$, $n \geq m-1$, so the kernel must have dimension at least $1$. If $p \in \ker \dual \psi_{\mathcal{B}^{(1)}}$, then the associated point in $X \cap H$ is an essential singularity. This contradicts the hypothesis, and so we are done.
\end{proof}

For a general generic complete intersection Cayley variety of dimension $0$, we can expect a secant of the Cayley variety to define a hyperplane with exactly $n$ singularities.

\begin{lemma}
\label{lem: degree of segre with dependence}
Let $\mbp^1 \times \mbp^{n-1} \rightarrow \mbp^{2n-1}$ be the Segre embedding. Choose $d_1,\dots, d_r \in \mathbb{Z}_{>1}$ such that $d_1 + \dots + d_r = n$. Write an element $M \in \mbp^{2n-1}$ as the projectivization of a $2 \times n$ matrix
\[
	\begin{bmatrix}
	a_1 & a_2 & \dots & a_n  \\
	b_1 & b_2 & \dots & b_n  \\
	\end{bmatrix}.
\]
Let $L$ be the linear subspace of $\mbp^{2n-1}$ defined by the equations
\[
	b_1 = a_2, \; b_{d_1+1}=a_{d_1+2}, \; b_{d_2+1}=a_{d_2+2}, \; \dots, \; a_{d_{r-1}+2}=b_{d_{r-1}+1}.
\]
Then the intersection of $L$ and the image of $\mathbb{P}^1 \times \mathbb{P}^{n-1}$ is a smooth variety of dimension $n-r$.
\end{lemma}

\begin{proof}
	It suffices to assume that $r = 1$; thus we will show that  the intersection of the hyperplane $a_2 = b_1$ with the Segre variety is a smooth variety of dimension $n-1$. As the Segre variety is irreducible and is not contained in the hyperplane $Z(a_2-b_1)$; the intersection clearly has dimension $n-1$. If $n \leq 3$, a computation shows that the intersection of $L$ and the image of $\mathbb{P}^1 \times \mathbb{P}^{n-1}$ is a smooth variety of dimension $n-1$. By induction on $n$ it suffices to demonstrate the claim on the open set $U \coloneqq D(\prod_{i=3}^n a_ib_i)$. Let $\sigma$ denote the Segre embedding.
		
	First, we we prove that the intersection
	\[
		L \cap \Im(\sigma) \cap D(\prod_{i=1}^na_ib_i)
	\]	
	is smooth. On $D(\prod_{i=1}^n a_ib_i)$, the Segre variety is a complete intersection of the equations $\{a_ib_{i+1}-a_{i+1}b_i\}_{1\leq i < n}$. Let $M$ be the $n \times 2n$ matrix defined as follows. For $1 \leq i < n$ and $1 \leq j \leq n$, let the entry in the $i$-th row and $j$-th column be $\frac{\partial (a_ib_{i+1}-a_{i+1}b_i)}{\partial a_j}$. For $1 \leq i < n$ and $n < j \leq 2n$, let the entry in the $i$-th row and $j$-th column be $\frac{\partial (a_ib_{i+1}-a_{i+1}b_i)}{\partial b_{j-n}}$. For $1 \leq j \leq n$, let the entry in the $n$-th row and $j$-th column be $\frac{\partial (a_2 - b_1)}{\partial a_{j}}$, and for $n < j \leq 2n$, let the entry in the $n$-th row and $j$-th column be $\frac{\partial (a_2 - b_1)}{\partial b_{j-n}}$. Then notice that the determinant of $M$ is invertible on $D(\prod_{i=1}^na_ib_i)$; thus the intersection 
	\[
		L \cap \Im(\sigma) \cap D(\prod_{i=1}^na_ib_i)
	\]	
	is smooth.
		
	Now consider the locus $Z(a_1a_2b_1b_2) \cap U$. Note that for $x \in \{a_1, a_2, b_1, b_2\}$, the variety $L \cap Z(x) \cap U \cap \Im(\sigma)$ is a smooth complete intersection in $U$ isomorphic to the Segre image of $\mbp^1 \times \mbp^{n-3}$, given by the equations
	\[
		(a_3b_{4}-a_{4}b_3, \dots, a_{n-1}b_{n}-a_{n}b_{n-1}, a_1, b_1, a_2, b_2).		
	\]
	As the Segre image of $\mbp^1 \times \mbp^{n-3}$ is smooth, the lemma is proved.
\end{proof}

%% SECTION
\section{Genus 4: sextic space curves and tritangents} \label{sec: genus 4 curves} 

In this section we consider tensors $\mathcal{A} \in k^4 \otimes (\Sym_2 k^2 \oplus \Sym_2 k^3)$. The tensor $\mathcal{A} \in k^4 \otimes \left(\Sym_2 k^2 \oplus \Sym_2 k^3 \right)$ defines a quadric cone, denoted by $X_2$, and a symmetric determinantal cubic surface, denoted by $X_3$. The intersection of symmetroids generically defines a canonical genus $4$ curve in $\mbp^3$. We study the relationship between the orbits of the action of $\GL_4(k) \times \GL_2(k) \times \GL_3(k)$ on $k^4 \otimes (\Sym_2 k^2 \oplus \Sym_2 k^3)$ and the moduli of canonical genus $4$ curves with a vanishing even theta characteristc. Secondly, the Cayley variety associated to $\mathcal{A}$ is generically a smooth intersection of $4$ quadrics in $\mbp^4$ and is therefore $16$ points. We have $\DAut(\mathcal{A}) = \gen{-I, \sigma}$, where $\sigma([y_0:y_1:y_2:y_3 :y_4]) = [y_0 : y_1 : -y_2 : -y_3 : -y_4]$. The connection between the tritangents of a genus $4$ curve and the Cayley variety is particularly rich. We describe how these tensors relate to the work of Milne \cite{Milne1923} and its generalizations by Bruin-Sert\"oz~\cite{BruinSertoz2018}. We also show how the work of Recillas~\cite{Recillas1974} can be interpreted in terms of tensors. 

Throughout this section we use a specific convention regarding notation. We use the notation $X_2, X_3$ and subscripts $2,3$ accordingly so that the index reflects the degree of the corresponding hypersurface or block.

\subsection{Moduli}
	We say that the tensor $\mathcal{A}$ is \emph{nondegenerate} if the associated genus $4$ curve is smooth. The vector space $k^4 \otimes \left(\Sym_2 k^2 \oplus \Sym_2 k^3 \right)$ has a natural $\GL_4(k) \times \GL_2(k) \times \GL_3(k)$ action. The nondegenerate locus of $k^4 \otimes \left(\Sym_2 k^2 \oplus \Sym_2 k^3 \right)$ is a $\GL_4(k) \times \GL_2(k) \times \GL_3(k)$-invariant open subscheme, and thus we can speak of orbit classes being nondegenerate.
	If $\mathcal{A}$ is nondegenerate, then both $X_2$ and $X_3$ must have empty Cayley varieties by Lemmas~\ref{lem: cayley variety for quadrics} and \ref{lem: cayley variety for cubics}. Thus, by Theorem~\ref{thm: singularities of symmetroid intersections} we see that the Cayley variety associated to $\mathcal{A}$ is a smooth complete intersection of dimension $0$.

	We consider genus $4$ curves $X$ with a vanishing even theta characteristic $\theta_0$. The distinguished theta characteristic allows us to assign a pairity to the $2$-torsion points, specifically,
		\[
			\Jac(X)[2]_{\odd} := \{ [\theta - \theta_0] \in \Jac(X)[2] : \theta \text{ is an odd theta characteristic} \}.
		\]
	Our main result of this subsection is:
	
	\begin{theorem} \label{thm: genus 4: Moduli of genus 4 curves with data}
	Let $B$ be a scheme over $\Spec \mbz[\frac{1}{6}]$. There is a canonical bijection between:
	\[
		\left \{  \quad \parbox{7cm}{\centering
		isomorphism classes of tuples $(X, \epsilon, \theta_0)$,\\
		where $X$ is a smooth genus $4$ curve over $B$ with vanishing even theta
		characteristic $\theta_0$ with a rational divisor class defined over $B$,
		and $\epsilon$ is an even $2$-torsion class}
		 \quad  \right\} 	 
		 \longleftrightarrow 
		 \left \{ \begin{array}{c}
		 \text{nondegenerate orbit classes of } \\
		 \mathcal{O}_B^4 \otimes \left(\Sym_2 \mathcal{O}_B^2 \oplus \Sym_2 \mathcal{O}_B^3 \right) \\
		 \text{under the action of } \\
		 \GL_4(\mathcal{O}_B) \times \GL_2(\mathcal{O}_B) \times \GL_3(\mathcal{O}_B) 
		 \end{array} \right\}
	\]
	\end{theorem}
	
	In order to prove this theorem, we first study the orbits of $k^4 \otimes \Sym_2 k^2$ and $k^4 \otimes \Sym_2 k^3$ and generalize these results to study orbits in families without much extra effort. A smooth canonical genus $4$ curve over a base scheme $B$ is a smooth morphism $X \rightarrow B$ whose fibres over points are canonical curves of genus $4$. Tensors over $k$ are replaced by tensors $\mathcal{O}_B^n \otimes \Sym_2 \mathcal{O}_B^m$ whose entries are sections of $\mathcal{O}_B$. Specialization at the closed points of $B$ recovers a tensor over a field. 
	%
	% NOTE: This is the definition that Wei Ho uses, so it seems like the right one for our purposes. At the very least, we have all the sections we need to construct the resolutions.
	%
	A vanishing even theta characteristic of $f\: X \rightarrow B$ is a line bundle $\theta_0$ on $X$ whose square is the relative canonical bundle and such that $\R^0f_*(\theta_0)$ is a free $\mathcal{O}_B$-module whose rank is even and greater than $0$.
		
	\begin{proposition} \label{prop: genus 4: quadrics and orbits}
		Let $B$ be a scheme over $\Spec \mbz[\frac{1}{6}]$ and let $X$ be a genus $4$ curve over $B$ with a vanishing even theta characteristic $\theta_0$ that admits a rational representative. Then
		there is a canonical construction of a determinantal quadric over $B$ containing $X$. The determinantal representation of this quadric uniquely determines an orbit class of $\mathcal{O}_B^4 \otimes \Sym_2 \mathcal{O}_B^2$ under $\GL_4 \times \GL_2$.
	\end{proposition}
		
	\begin{proof}
		Let $f\: X \rightarrow B$ denote the structure morphism. In this case, we can construct the determinantal representation directly. We consider the multiplication map $\theta_0 \times \theta_0 \rightarrow \kappa_X$. Let $s,t \in \R^0f_*(\theta_0)(B)$ be a basis for the relative global sections. Note that in the fibres, $s,t$ globally generate $\theta_B$. We also choose a basis $x_0, \ldots, x_3$ for $\R^0f_*(\kappa_X)(B)$. From the multiplication map, we have that $a := s^2, b := st, c := t^2$ all define sections of $\R^0f_*(\kappa_X)(B)$. We now consider the exact sequence of sheaves on $\mbp^3_B$ given by
		\[
			\xym{
			0 \ar[r] & \mathcal{O}_{\mbp^3}(-1)^{\oplus 2} \ar[rr]^{\tiny \begin{bmatrix} a & b \\ b & c \end{bmatrix}} & & \mathcal{O}_{\mbp^3}^{\oplus 2} \ar[r]^-{g} & \mathcal{E} \ar[r] & 0
			}.
		\]
		Restricting to $X$, we obtain the complex
		\[
			\xym{
			(\kappa_X^\vee)^{\oplus 2} \ar[rr]^{\tiny \begin{bmatrix} s^2 & st \\ st & t^2 \end{bmatrix}} & & \mathcal{O}_{X}^{\oplus 2} \ar[r]^-{g} & \mathcal{E}|_X \ar[r] & 0
			}.
		\]
		We see that $g|_X = \begin{bmatrix}t & -s \end{bmatrix}$, which is to say, $\mathcal{E} |_X \iso \theta_0$. The only choices made in selecting the quadric $Z(ac-b^2)$ were the choices of bases for the relative global sections of $\theta_0$ and $\kappa_X$, so the result is proven.
	\end{proof}
	
	If $B = \Spec k$, the tensor $\mathcal{A}_2$ is determined by the choice of bases of $\H^0(X, \theta_0)$ and $\H^0(X, \kappa_X)$, but not vice-versa.
	The entries of $\mathcal{A}_2$ determine a $3$-dimensional space of linear forms. The choice of the fourth form is free, so long as it is not within this $3$-dimensional space.
	
	\begin{remark}
		A parametrization of a quadric cone $X_2$ is an isomorphism $p\: \mathbb{P}(1:1:2) \rightarrow X_2$. A parametrization is determined by an isomorphism $\beta\: \H^0(X_2, \theta^{\otimes 2}) \rightarrow \H^0(\mbp^3, \mathcal{O}_{\mbp^3}(1))$, and conversely any parametrization determines an isomorphism $\beta$ up to scaling. However, a choice of tensor $\mathcal{A}_2$ such that $X_2 = Z(\det \mathcal{A}_2(\mathbf{x}, \cdot, \cdot))$ does not determine a parametrization, since the entries of $\mathcal{A}_2$ do not determine a basis for $\H^0(\mbp^3, \mathcal{O}_{\mbp^3}(1))$. 
	\end{remark}

	%YYY
	
	Over a field, the classical result of Wirtinger and Coble \cite[Theorem~1.5]{Catanese1981}, allows us to construct a cubic symmetroid containing $X$ from the even $2$-torsion class $\epsilon$. 

	\begin{theorem}[Wirtinger, Coble] \label{thm: genus 4: cubics and orbits}
		There is a canonical bijection between symmetric determinantal Cayley cubics over $k$ containing a given canonical genus $4$ curve $X$ and pairs $(X, \epsilon)$ where $\epsilon^{\otimes 2} \iso \kappa_X$, such that $\epsilon$ is an even $2$-torsion class.
	\end{theorem}

	With some additions, Catanese's argument from \cite{Catanese1981} allows us to prove Theorem~\ref{thm: genus 4: Moduli of genus 4 curves with data} in families.

	\begin{theorem} \label{thm: genus 4: cubics and orbits in families}
		Let $B$ be a scheme over $\mathbb{Z}[\frac{1}{6}]$ and let $X \rightarrow B$ be a smooth genus $4$ curve with no bielliptic fibre. There is a canonical bijection between symmetric determinantal Cayley cubics $X_3$ over $B$ containing $X$ and pairs $(X, \epsilon)$ where $\epsilon^{\otimes 2} \iso \kappa_X$, such that $\epsilon$ is an even $2$-torsion class.
	\end{theorem}

	\begin{proof}
		We let $f$ denote the various structure morphisms for schemes over $B$. We explain how to construct a determinantal representation associated to an even $2$-torsion class $\epsilon$ in the style of Proposition~\ref{prop: genus 4: quadrics and orbits}.
		The bundle $\kappa_X \otimes \epsilon$ defines a morphism $X \rightarrow |\kappa_X \otimes \epsilon|^\vee \iso \mbp^2$. As the arithmetic genus of $X$ is equal to $4$ and $\epsilon$ is even, we see that the image of $X$ is of degree $6$ and has a nontrivial singular subscheme $\Delta$.
		Let $S$ be the blow-up of $\mbp^2$ along $\Delta$. On $S$, we obtain divisors $\Delta, X$ over the corresponding divisors on $\mbp^2$ as well as the pullback of the hyperplane class $H$. Note that the relative canonical bundle on $S$ is $\mathcal{O}_S(\Delta - 3H)$. As $X \subset \mbp^2$ admits singularities along each of the points in $\Delta$, by the adjunction formula we have that $\kappa_X \iso \res_X (\mathcal{O}_S(X) \otimes \kappa_S) \iso \res_X \mathcal{O}_S(3H - \Delta)$. The relative global sections of $\mathcal{O}_S(3H - \Delta)$ define a morphism $\phi\: S \rightarrow \mbp^3 \iso |\kappa_X|^\vee$.
		
		% Note: O_X(4H-2*Delta) is trivial.
		Since $\mathcal{O}_S(X) \iso \mathcal{O}_S(6H-2\Delta)$, we have by the exact sequence
		\[
			\xym{
				0 \ar[r] & \mathcal{O}_S(4H-2\Delta-X) \ar[r] & \mathcal{O}_S(4H-2\Delta) \ar[r] & \mathcal{O}_X \ar[r] & 0
			}
		\]
		that $\R^0f_*\mathcal{O}_S(4H-2\Delta)(B) = 1$. The unique relative global section of $\mathcal{O}_S(4H-2\Delta)$ defines an isomorphism $\mathcal{O}_{S}(4H-2\Delta) \iso \mathcal{O}_S$, so $\mathcal{O}_S(3H - \Delta)^{\otimes 2} \iso \mathcal{O}_S(2H)$. 
		Let $\{s, t, u\}$ be a basis for $\R^0f_*\mathcal{O}_{S}(H)(B)$. The matrix
		\[
			\begin{bmatrix}
				s^2 & st & su \\
				st & t^2 & tu \\
				su & tu & u^2
			\end{bmatrix}
		\]
		can be identified with a matrix $\mathcal{B}(\mathbf{x})$ with entries in $\R^0f_*\mathcal{O}_{\mbp^3}(2)(B)$. The corresponding exact sequence of sheaves on $\mbp^3$ is
		\[
			\xym{
			0 \ar[r] & \mathcal{O}_{\mbp^3}(-2)^{\oplus 3} \ar[r]^-{\mathcal{B}} & \mathcal{O}_{\mbp^3}^{\oplus 3} \ar[r]^-{g} & \mathcal{E} \ar[r] & 0
			}.
		\]
		The matrix $\mathcal{B}(\mathbf{x})$ defines a rational map $\mbp^3 \dashrightarrow \mbp^5$. Let $\nu\: \mbp^2 \rightarrow \mbp^5$ be the Veronese map, and observe that the diagram
		\[
			\xym{
				\mbp^2 \ar[r]^\nu \ar@{-->}[d]_\phi & \mbp^5 \\
				 \mbp^3 \ar@{-->}[ur]_{\mathbf{x} \mapsto \mathcal{B}(\mathbf{x})}
			}
		\]
		commutes due to the identifications $\kappa_X^{\otimes 2} \iso (\kappa_X \otimes \epsilon)^{\otimes 2}$ and $\kappa_X \iso \res_X \mathcal{O}_{S}(3H-\Delta)$. In particular, $\mathcal{B}(\mathbf{x})$ has rank $1$ along an open subscheme of $X_3 := \overline{\phi(S)}$. Thus, we see that $\adj \mathcal{B}(\mathbf{x})$ is identically zero along $X_3$. As $X_3$ is a hypersurface, it is defined by a polynomial $F$, and each entry of $\adj \mathcal{B}(\mathbf{x})$ is divisible by $F$. Over the generic point of the base scheme, the degree of $F$ is equal to three by Theorem~\ref{thm: genus 4: cubics and orbits}, so we conclude that $X_3$ is a hypersurface of degree $3$ and that $\adj \mathcal{B}(\mathbf{x}) = F \cdot \mathcal{A}(\mathbf{x})$ for some matrix of linear forms. Since $X_3$ is irreducible of degree $3$, we have that $\det \mathcal{A}(\mathbf{x})$ is a scalar multiple of $F$, hence, a symmetric determinantal representation for a cubic surface containing the canonical image of $X$.
	\end{proof}

	\begin{remark}		
		For the sake of concreteness, we can describe this construction in a specific example. In the generic case $\Delta$ is the locus of $6$ intersection points of some quadruple of lines in $\mbp^2$. If we take as an example the lines to be $y_0, y_1, y_2, y_3 := -(y_0 + y_1 + y_2)$, then
		\[
		\phi(y_0, y_1, y_2) = (y_1y_2y_3 : y_0y_2y_3 : y_0y_1y_3 : y_0y_1y_2)
		\]
		and the image has defining equation
		\[
			\det
			\begin{bmatrix}
				x_0 + x_3 & x_3 & x_3 \\
				x_3 & x_1 + x_3 & x_3 \\
				x_3 & x_3 & x_2 + x_3
			\end{bmatrix}.
		\]
		Note that $\mathcal{O}_S(X) \iso \mathcal{O}_S(6H - 2\Delta)$ and $\mathcal{O}_S(4H - 2\Delta) = \gen{y_0y_1y_2y_3}$. On $X$ we have the equivalence
		\[
			y_0^2 \sim y_0^2 \cdot (y_0y_1y_2y_3) \sim (x_1 + x_3)(x_2 + x_3) - x_3^2,
		\]
		and similarly for the other minors.
	\end{remark}

	\begin{proof}[Proof of Theorem~\ref{thm: genus 4: Moduli of genus 4 curves with data}]
		We let $f$ denote the various structure morphisms for schemes over $B$ as before.
		We first rigidify the situation by choosing bases for $\R^0f_*\kappa_X(B)$, $\R^0f_*(\kappa_X \otimes \epsilon)(B)$, and $\R^0f_*\theta_0(B)$.
		Given an isomorphism class of a tuple $(X, \epsilon, \theta_0)$, we apply Proposition~\ref{prop: genus 4: quadrics and orbits} and Theorem~\ref{thm: genus 4: cubics and orbits in families} to construct a determinantal quadric $X_2$ and determinantal cubic $X_3$ over $B$. The labelled bases provide associated determinantal representations $\mathcal{A}_2$, $\mathcal{A}_3$ (respectively). 
		
		Conversely, a tensor $\mathcal{A}_3 \in \mathcal{O}_B^4 \otimes \Sym_2 \mathcal{O}_B^3$ determines a determinantal cubic $X_3$ and a tensor $\mathcal{A}_2 \in \mathcal{O}_B^4 \otimes \Sym^2 \mathcal{O}_B^2$ determines a determinantal quadric $X_2$. If $X := X_2 \cap X_3$ is a nonsingular genus $4$ curve, we see that the aCM sheaf giving rise to $\mathcal{A}_2$ defines a vanishing even theta characteristic $\theta_0$ on $X$, together with a basis for $\R^0f_*\theta_0(B)$. Similarly, the aCM sheaf on $X_3$ defined by $\mathcal{A}_3$ restricts to a line bundle $\mathcal{E}$ on $X$ such that $\mathcal{E} \otimes \kappa_X^\vee$ is an even $2$-torsion class $\epsilon$. The tensor $\mathcal{A}_3$ is defined with reference to coordinates for $\mbp^2$ and $\mbp^3$, so also determines bases for both $\R^0f_*\mathcal{E}(B)$ and $\R^0f_*\mathcal{O}_{\mbp^3}(1)(B)$, and therefore a bases for $\R^0f_*(\kappa_X \otimes \epsilon)(B)$ and $\R^0f_*\kappa_X(B)$.
		
		In particular, we have demonstrated a $\GL_4 \times \GL_2 \times \GL_3$-equivariant bijection between boxes $\mathcal{A} \in \mathcal{O}_B^4 \otimes (\Sym_2 \mathcal{O}_B^2 \oplus \Sym_2 \mathcal{O}_B^3)$ and tuples of data $(X, \theta_0, \epsilon, B_1, B_2, B_3)$, where $B_1, B_2, B_3$ are bases for $\R^0f_*\theta_0(B)$, $\R^0(\kappa_X \otimes \epsilon)(B)$, $\R^0f_*\kappa_X(B)$ respectively. The two constructions given above are inverses of each other.
	\end{proof}
		
\subsection{Theta characteristics} \label{sec: sub: genus 4 theta characteristics}

The action of $\DAut(\mathcal{A})$ on $\mathfrak{C}$ naturally splits the Cayley variety into $8$ orbits of $2$ points each. There are $\binom{16}{2} = 120$ secants of the Cayley variety; $8$ of these secants are fixed by the action of $\sigma$ and the remaining $112$ are divided into $56$ orbits of size $2$. Thus, Lemma~\ref{lem: psi is injective on secants mod diagonal symmetries} shows that the secants of the Cayley variety define $8 + 56$ distinct tritangent planes of $X$.

The effective representative of the vanishing even theta characteristic $\theta_0$ define an infinite family of planes tangent to $X$ at three points corresponding. By Proposition~\ref{prop: tangency for curves}, any hyperplane $H$ defined by a secant of the form $\ell:= \ell(p,\sigma(p))$ for some $p \in \mathfrak{C}(\bar k)$ is tangent to the quadric cone containing $X$, so defines an effective representative of $\theta_0$. Remark~\ref{rem: converse for sigma fixed lines} shows that $H$ is also tangent to some point of $X_3$ (however, this point does not lie on $X$). There are precisely eight lines of the form $\ell(p, \sigma(p))$, which correspond to the eight planes tangent to both $X_2$ and $X_3$. The remaining secants define odd theta characteristics of $X$ by Proposition~\ref{prop: tangency for curves}.

\subsection{Weil pairing}
The $56$ odd theta characteristics obtained from this construction also have a special property with respect to the Weil pairing. The Weil pairing on $\Jac(X)[2]$ is a skew-symmetric bilinear form
\[
	e_2 \colon \Jac(X)[2] \times \Jac(X)[2] \rightarrow \mathbb{F}_2.
\]
For any $\epsilon \in \Jac(X)[2]_{\mathrm{even}}$, there are precisely $56$ points in $D \in \Jac(X)[2]_{\odd}$ such that $e_2(\epsilon, D) = 0$. A theta characteristic defines a quadratic form $\mathcal{Q}$ on $\Jac(X)[2]$ such that for any $D_1, D_2 \in \Jac(X)[2]$ we have
\[
	e_2(D_1, D_2) = \mathcal{Q}(D_1+D_2) - \mathcal{Q}(D_1) - \mathcal{Q}(D_2).
\]
We denote the quadratic form on $\Jac(X)[2]$ defined by $\theta_0$ by $\mathcal{Q}_{0}$.

\begin{lemma} \label{lem: projected singular points are colinear}
	Let $\ell$ be a secant line of $\mathfrak{C}$, let $[H] = \psi(\ell)$, let $\{x_1,x_2,x_3\} = X \cap H$, and let $\pi \colon \mbp^4 \dashrightarrow \mbp^2$ be the projection map onto the last $3$ coordinates. Then $\{\pi(\ker \mathcal{A}(x_i, \cdot, \cdot)) : i = 1, \ldots, 3\}$ is a set of collinear points in $\mbp^2$. 
\end{lemma}
\begin{proof}
	Because $\ker \mathcal{A}(x_i, \cdot, \cdot) = \gen{\varphi_{X_2}(x_i), \varphi_{X_3}(x_i)}$, we see $\pi(\ker \mathcal{A}(x_i, \cdot, \cdot)) = \varphi_{X_3}(x_i)$. Each intersection $\ker \mathcal{A}(x_i, \cdot, \cdot) \cap \ell$ is nonempty, so the projections of the kernels lie on $\pi(\ell)$ and thus are collinear.
\end{proof}

Restricted to $X$, we have that $\varphi_{X_3} \colon X \rightarrow \mbp^2$ is a morphism and that $\varphi_{X_3}^* \mathcal{O}_{\mbp^2}(1) \cong \kappa_X \otimes \epsilon$. Lemma~\ref{lem: projected singular points are colinear} shows $h^0(X, \kappa_X \otimes \epsilon \otimes \theta^{\vee}) = h^0(X, \theta \otimes \epsilon) = 1$ for any odd theta characteristic $\theta$ arising from our construction.

\begin{proposition} \label{prop: genus four Weil pairing}
The $56$ odd theta characteristics obtained by $\mathcal{A}$ are precisely the odd characteristics $\theta$ such that $e_2(\epsilon, \theta \otimes \theta_0^{\vee}) = 0$.
\end{proposition}
\begin{proof}
For any $2$-torsion element $D$, we have the formula from Mumford
\[
	\mathcal{Q}_{0}(D) \equiv h^0(\theta_0) + h^0(D \otimes \theta_0) \pmod 2.
\]
Thus, since $\epsilon \otimes \theta \otimes \theta_0^{\vee}$ is a non-trivial $2$-torsion class we have that
\begin{alignat*}{3}
	\mathcal{Q}_{0}(\epsilon \otimes \theta \otimes \theta_0^{\vee}) & \equiv h^0(\epsilon \otimes \theta \otimes \theta_0^{\vee}) + h^0(\epsilon \otimes \theta)  \equiv h^0(\epsilon \otimes \theta) & {} \equiv 1 \pmod 2, \\
	\mathcal{Q}_{0}(\theta \otimes \theta_0^{\vee}) & \equiv h^0(\theta_0) + h^0(\theta) & \equiv 1  \pmod 2, \\
	\mathcal{Q}_{0}(\epsilon) & \equiv h^0(\theta_0) + h^0(\epsilon \otimes \theta_0) & \equiv 0 \pmod 2,
\end{alignat*}
by choice of $\epsilon$. Thus, $e_2(\epsilon, \theta \otimes \theta_0^{\vee}) = 0$.
\end{proof}

\noindent
We obtain Theorem~\ref{thm: genusfoursummary} by combining Theorem~\ref{thm: genus 4: Moduli of genus 4 curves with data}, Proposition~\ref{prop: genus four Weil pairing}, and the discussion of Section~\ref{sec: sub: genus 4 theta characteristics}.

\genusfoursummary*

\subsection{Milne's Construction} \label{sec: sub: BruinSertoz}

In this subsection we discuss the relation between our construction and a construction of Milne \cite{Milne1923}, as extended by Bruin and Sert\"oz \cite{BruinSertoz2018}. We describe the details of this construction.

For a general canonical genus $4$ curve $X$, there is a bijection between $2$-torsion classes $\Jac(X)[2]$ and Cayley cubics containing $X$. Choose a Cayley cubic $X_3$ containing $X$, and let $\epsilon$ be the corresponding $2$-torsion class of $\Jac(X)$ obtained from the double cover $\wtilde{X}_3 \rightarrow X_3$ unramified outside the singularities of $X_3$. Let $X_2$ be the unique quadric surface containing $X$. Then $\dual X_2 \cap \dual X_3$ is a singular model of a smooth genus $3$ curve $Y$. Milne showed that an odd theta characteristic of $Y$ naturally gives rise to a pair of tritangent planes of $X$. Recillas showed that if $\wtilde X \rightarrow X$ is the unramified double cover of $X$ obtained from the $2$-torsion class $\epsilon$, then $Y$ is a genus $3$ curve such that $\Jac(Y) \iso \Prym(\wtilde{X}/X)$ \cite{Recillas1974}. 

Bruin and Sert\"oz generalized Milne's construction by removing genericity assumptions. Suppose now that $X$ is a smooth genus $4$ curve contained in a quadric cone $X_2$ and let $X_3$ be a Cayley cubic containing $X$ as before. Let $\theta_0$ denote the vanishing even theta characteristic of $X$ given by the quadric cone. The image of $X_2$ under the Gauss map is a connected curve of degree $2$ in $\dual \mbp^3$. 
It is easily seen that if $H \subset \dual \mbp^3$ is the hyperplane dual to the vertex of $X_2$, that $H$ is the unique hyperplane containing $\dual X_2$. By virtue of being a Cayley cubic, we obtain a symmetric determinantal representation for $X_3$ as well as maps $\varphi_3\: X_3 \dashrightarrow \mbp^2$ and $\psi_3 \: \mbp^2 \dashrightarrow \mbp^3$. 

The variety $\psi_3^{-1}(H \cap \dual X_3)$ is a plane conic with the $8$ marked points $\psi_3^{-1}(\dual X_2 \cap \dual X_3)$. The genus $3$ curve $Y$ is the hyperelliptic curve given by taking a double cover of the plane conic branched at the $8$ marked points. The ``bitangents'' are the $28$ lines between these $8$ points.  %
%
% NOTE: So...checking Dolgachev, Bruin and Sertoz use the term "enveloping cone" a bit differently. I'm just going to use a different (more standard) term.
%
Given a smooth conic $C \subset \dual \mbp^3$, there is a unique quadric cone $\Lambda \subset \mbp^3$ such that the dual variety to $\Lambda$ is $C$. Of course, $\Lambda$ is tangent to every hyperplane parameterized by $C$, so is the \emph{envelope} of this family of hyperplanes -- as an abbreviation we say that $\Lambda$ is the envelope of $C$. 
Given a ``bitangent'' line $\ell$, its image $\psi_3(\ell)$ is a conic in $\dual \mbp^3$. If $\Lambda$ is the envelope of the image of a ``bitangent'', then $\Lambda$ meets $X_2$ along two hyperplane sections; moreover, the two associated hyperplanes are tritangent planes to $X$. The construction described above is the construction in \cite{BruinSertoz2018} that associates a pair of tritangent planes of $X$ to a ``bitangent'' of $\psi_3^{-1}(H \cap \dual X_3)$. We summarize this construction below:

\begin{theorem}[Bruin-Sert\"oz]
\label{thm: bruin sertoz}
	Let $X$ be a genus $4$ curve with a vanishing theta null $\theta_0$ and let $\epsilon \in \Jac(X)$ be a $2$-torsion point not of the form $[\theta-\theta_0]$ for some odd theta characteristic $\theta$ (in their terminology, $\epsilon$ is an \emph{even} $2$-torsion point). Let $X_3$ be the cubic symmetroid associated to $\epsilon$ and let $\varphi\: X_3 \dashrightarrow \mbp^2$ be the kernel map. 
	
	Then the genus $3$ curve associated to $(C, \epsilon)$ is the double cover of a conic in $\mbp^2$ ramified at $8$ marked points. Furthermore, each of the $28$ lines through these $8$ points in $\mbp^2$ corresponds to a pair of odd theta characteristics of $X$. If $D_1$ and $D_2$ are the two effective representatives of the odd theta characteristics associated to a line $L$, then $\varphi(D_1), \varphi(D_2) \subset L$.
\end{theorem}

One may wonder whether the $56$ tritangent planes obtained from the Cayley variety associated to $X$ are related to the $56$ tritangents from the construction in \cite{BruinSertoz2018}. As expected, the answer is yes; in fact, the two constructions are essentially the same.
The projection $\pi\: \mbp^4 \dashrightarrow \mbp^2$ features prominently  in the argument to show that these constructions are the same.

\begin{lemma} \label{lem: genus 4: image of Cayley in P2}
	We have that $\pi(\mathfrak{C}) = \psi_3^{-1}(\dual X_3 \cap \dual X_2)$.
	The induced morphism $\pi\: \mathfrak{C} \rightarrow \psi_3^{-1}(\dual X_3 \cap \dual X_2)$ of $0$-dimensional schemes is of degree $2$. 
\end{lemma}

\begin{proof}
	Let $p \in \mathfrak{C}(\bar k)$ and let $\ell := \ell(p, \sigma(p))$.
	If $p = [a \colon b \colon c \colon d \colon e]$, then $\sigma(p) = [-a \colon -b \colon c \colon d \colon e]$ and thus $\pi(\ell) = [c \colon d \colon e]$. Recall that $\psi(\ell)$ is a point and that $\psi(\ell) = \psi_3(\pi(\ell))$. As $[a \colon b \colon 0 \colon 0 \colon 0] \in \ell$ it follows that $\psi(\ell) \in \dual X_2$. As $[0 \colon 0 \colon c \colon d \colon e] \in \ell$, we have $\psi(\ell) \in \dual X_3$. 
\end{proof}

Notice, in particular, that if $p,q \in \mathfrak{C}$ such that $p \neq \sigma(q)$, then $\pi(\ell(p,q))$ is a line between two of the eight marked points on $\psi_3^{-1}(X_3 \cap X_2)$, i.e., it is a ``bitangent'' of $Y$. If $p', q'$ are two points in $\psi_3^{-1}(\dual X_3 \cap \dual X_2)$, then there are two points $p, \sigma(p) \in \mathfrak{C}$ lying over $p'$, and two points $q, \sigma(q) \in \mathfrak{C}$ lying over $q'$. The four secants
	\[
	 \quad \ell(p, q), \ \ \ell(p, \sigma(q)), \ \ \ell(\sigma(p), q), \ \ \ell(\sigma(p), \sigma(q))
	\]
define two odd theta characteristics of $X$. In this way we can relate pairs of odd theta characteristics with ``bitangents'' of the conic. The next proposition shows that this association is the same one from \cite{BruinSertoz2018}.

\begin{proposition}
\label{prop: compatibility for milne}
	Choose two distinct points $p, q \in \mathfrak{C}$ such that $p \neq \sigma(q)$. Let $H_1$ and $H_2$ denote the $2$ tritangent planes given by $\psi(\ell(p,q))$ and $\psi(\ell(p,\sigma(q)))$. Let $H_1'$ and $H_2'$ be the two planes containing the two conics $X_2 \cap \Lambda$, where
	$\Lambda$ is the envelope of $\psi_3(\ell(\pi(p),\pi(q)))$. Then $\{H_1, H_2\} = \{H_1', H_2'\}$.
\end{proposition}

\begin{proof}
	Let $D_1, D_2$ be the effective representatives for the odd theta characteristics on $X$ defined by $H_1', H_2'$. By \cite[Corollary~7.2]{BruinSertoz2018}, the divisor $D_1' + D_2'$ is the preimage under $\varphi_3\: X \dashrightarrow \mbp^2$ of a ``bitangent line", that is, a line between two of the eight marked points $\psi_3^{-1}(\dual X_2 \cap \dual X_3)$. 
	
	For each $x \in X \cap H_1$, we have that $\ker \mathcal{A}(x, \cdot, \cdot) =  \gen{\varphi_2(x), \varphi_3(x)}$ and that $\ker \mathcal{A}(x, \cdot, \cdot) \cap \ell(p,q) \neq \emptyset$. 
	We see that the six points $\{\varphi_3(x) : x \in D_1 + D_2\}$ are collinear by Lemma~\ref{lem: genus 4: image of Cayley in P2} and lie on the ``bitangent'' $\pi(\ell(p,q)) = \pi(\ell(p, \sigma(q)))$. In other words, $D_1 + D_2 = \varphi_3^{-1}(\pi(\ell(p,q)))$ as a divisor of $X$, so $D_1 + D_2 = D_1' + D_2'$. In particular, the pairs of associated hyperplanes are equal.
\end{proof}

The tensor $\mathcal{A}$ defines a natural quadruple cover
		\[
			\left \{ \begin{array}{c}
				\text{The $112$ secant lines } \ell \text{ of } \mathfrak{C} \text{ defining} \\
				\text{to odd theta characteristics of } X 			\end{array}
			\right \} \rightarrow \left \{ \begin{array}{c}
			\text{The $28$ lines given by } \pi(\ell) \\
			\text{for } \ell \text{ a secant line of } \mathfrak{C}
			\end{array}\right \}
		\]
which factors through double cover
\[
			\left \{ \begin{array}{c}
				\text{The $112$ secant lines } \ell \text{ of } \mathfrak{C} \text{ defining} \\
				\text{to odd theta characteristics of } X 			\end{array}
			\right \} \rightarrow \left \{ \begin{array}{c}
				\text{The $56$ odd theta characteristics represented} \\
				\text{by the image of the secant lines of } \mathfrak{C} \text{ under } \psi
			\end{array} \right\}.
		\]
In this way, we associate a ``bitangent'' in $\mbp^2$ to a pair of odd theta characteristics defined by the image of the secant lines of $\mathfrak{C}$ under $\psi$. We now prove:

\tritangentbitangenttheorem*
\begin{proof}
Apply Proposition~\ref{prop: compatibility for milne}. Theorem~\ref{thm: tritangent bitangent theorem} follows immediately.
\end{proof}

%%% SUBSECTION
\subsection{Tensorification of Recillas' construction} \label{sec: sub: Recillas}

For genus $4$ curves with a vanishing even theta characteristic, we show how to realize Recillas' construction in terms of tensors. If $Y$ is a hyperelliptic genus $3$ curve defined by $y^2 = -F(x, z)$, we can represent an open subscheme of $\Sym^4 Y$ by a pair of polynomials $(a(x, z), b(x, z) - y)$ such that $a,b$ are homogeneous forms of degree $4$ and $b^2 + F \equiv 0 \pmod a$. This is the Mumford representation of such a divisor \cite{Cantor1987}. Not every divisor in $\Sym^4 Y$ can be represented in this way; if $\pi\: Y \rightarrow \mathbb{P}^1$ is the canonical map, the collection of divisors which do not have a Mumford representation are of the form $D' + \pi^*(p)$ for some $D' \in \Div^2 Y_{k^\al}$ and $p \in \mathbb{P}^1(k^\al)$. We will say that a divisor in $\Sym^4 Y$ is \emph{semireduced} if it admits a Mumford representation with $a,b$ binary quartics. We will say that the pair of divisors $\{D_1, D_2\}$ of $Y$ is \emph{residual} if $D_1$ is the image of $D_2$ under the hyperelliptic involution.

\begin{proposition}
	There is a bijection between $\GL_2 \times \GL_2$-orbits of $\Sym^4 k^2 \otimes \Sym_2 k^2$ and hyperelliptic genus $3$ curves $Y$ with a residual pair of semireduced rational divisors $D_1, D_2 \in \Sym^4(Y)(k)$.
\end{proposition}

\begin{proof}
	If $Y$ is defined by $y^2 = -F(x,z)$ and $(a, b-y)$ is the Mumford representation of a divisor $D$ in $\Sym^4(Y)(k)$, then $F = ac - b^2$ for some binary quartic polynomial $c$. We obtain the determinantal representation $\begin{bmatrix} a & \pm b \\ \pm b & c \end{bmatrix}$. Note that the choice of sign is irrelevant, as the two determinantal representations are equivalent under $\{\id\} \times \GL_2$.

	We see that the morphisms $|D_1|, |D_2| \: Y \rightarrow \mathbb{P}^1$ of degree $4$ are explicitly represented by $(x, y, z) \mapsto (b\pm y : c)$. The natural $\PGL_2$ action on the morphism $(b-y : c)\: Y \rightarrow \mbp^1$ is identified with the natural action of $\{\id\} \times \SL_2$ on the tensor. The action of $\begin{bmatrix} 1 & 0 \\ 0 & -1 \end{bmatrix}$ on the tensor exchanges the two morphisms.
\end{proof}

\begin{remark}
	Scaling the defining equation of a hyperellitpic curve by a constant does not change the orbit under $\GL_2(\bar k)$, as
	\[
		\lambda \cdot \det \begin{bmatrix} 	a & b \\
			b & c
			\end{bmatrix}
		= \det \left(
		\begin{bmatrix}
			1 & 0 \\
			0 & \sqrt{\lambda}
		\end{bmatrix}
		\begin{bmatrix}
		a & b \\
		b & c
		\end{bmatrix}
		\begin{bmatrix}
			1 & 0 \\
			0 & \sqrt{\lambda}
		\end{bmatrix}^T
		\right).
	\]
	
\end{remark}

With $\mbp^2 = \Proj(k[y_2, y_3, y_4])$, note that the Veronese map $\nu_2\: (s : t) \mapsto (s^2 : st : t^2)$ gives a surjection $\nu_2^*\: k[y_2, y_3, y_4]_{(2)} \rightarrow k[s,t]_{(4)}$. So, given a tensor $\mathcal{A} \in \Sym^4 k^2 \otimes \Sym_2 k^2$ representing a hyperelliptic genus $3$ curve $Y$ and a residual pair of semireduced divisors $D_1, D_2$ of degree $4$, we can construct a tensor $\mathcal{B} \in \Sym^2 k^3 \otimes \Sym_2 k^2$ by choosing inverse images for the entries of $\mathcal{A}$ under $\nu_2^*$, which are unique up to multiples of $y_3^2 - y_2y_4$. The intersection $Z(\det \mathcal{B}(\mathbf{y}, \cdot), y_3^2 - y_2y_4)$ consists of the images under $\nu_2$ of the eight branch points of $Y \rightarrow \mbp^1$. In other words, if $Y \rightarrow C$ is a double cover of a rational plane conic, then a residual pair $D_1, D_2$ of semireduced divisors of degree $4$ is equivalent to a presentation of the branch locus of the form $Z(\det \mathcal{B}(\mathbf{y}, \cdot)) \cap C$ for some $\mathcal{B} \in \Sym^2 k^3 \otimes \Sym_2 k^2$.

\begin{construction} \label{cons: tensor Recillas}
	Let $\mbp^2 := \Proj(k[y_2, y_3, y_4])$. Given a conic $C \subset \mbp^2$ and a tensor $\mathcal{B} \in \Sym^2 k^3 \otimes \Sym_2 k^2$ representing a $2 \times 2$ matrix of quadratic forms, we can construct a tensor $\mathcal{A} \in k^4 \otimes (\Sym_2 k^2 \oplus \Sym_2 k^3)$ as follows: Writing $\mathcal{B}(\mathbf{y}, \cdot) = \begin{bmatrix} a & b \\ b & c \end{bmatrix}$, we have that
	\begin{align*}
		a(\mathbf{y}) = \sum_{2 \leq i,j \leq 4} a_{ij} y_iy_j, \quad
		b(\mathbf{y}) = \sum_{2 \leq i,j \leq 4} b_{ij} y_iy_j, \quad
		c(\mathbf{y}) = \sum_{2 \leq i,j \leq 4} c_{ij} y_iy_j.
	\end{align*}
	We may write the defining equation for $C$ as $\sum_{2 \leq i,j \leq 4} d_{ij} y_iy_j$. Define the slices of $\mathcal{A}$ to be
	\begin{alignat*}{3}
		&A_0 := 
		\begin{bmatrix}
			-1 & 0 & 0 & 0 & 0\\
			0 & 0 & 0 & 0 & 0 \\
			0 & 0 & a_{22} & a_{23} & a_{24} \\
			 0 & 0 & a_{32} & a_{33} & a_{34} \\
			 0 & 0 & a_{42} & a_{43} & a_{44} \\
		\end{bmatrix},
		\quad 
		&&A_1 := 
		\begin{bmatrix}
			0 & -1 & 0 & 0 & 0 \\
			-1 & 0 & 0 & 0 & 0 \\
			0 & 0 & b_{22} & b_{23} & b_{24} \\
			 0 & 0 & b_{32} & b_{33} & b_{34} \\
			 0 & 0 & b_{42} & b_{43} & b_{44} \\
		\end{bmatrix},
		\\
		&A_2 := 
		\begin{bmatrix}
			0 & 0 & 0 & 0 & 0 \\
			0 & -1 & 0 & 0 & 0 \\
			0 & 0 & c_{22} & c_{23} & c_{24} \\
			 0 & 0 & c_{32} & c_{33} & c_{34} \\
			 0 & 0 & c_{42} & c_{43} & c_{44} \\
		\end{bmatrix},
		\quad
		&&A_3 :=
		\begin{bmatrix}
			0 & 0 & 0 & 0 & 0 \\
			0 & 0 & 0 & 0 & 0 \\
			0 & 0 & d_{22} & d_{23} & d_{24} \\
			 0 & 0 & d_{32} & d_{33} & d_{34} \\
			 0 & 0 & d_{42} & d_{43} & d_{44} \\
		\end{bmatrix}.			
	\end{alignat*}
	The tensor $\mathcal{A}$ defines a genus $4$ curve via an intersection of symmetroids. 
\end{construction}

Notice Construction~\ref{cons: tensor Recillas} has the following property: if $\alpha = (\alpha_2 : \alpha_3 : \alpha_4) \in C(\bar k)$ satisfies $\det \mathcal{B}(\alpha, \cdot) = 0$, then, fixing a choice of square root, at least one of
	\[
	\left(\frac{a(\alpha)}{\sqrt{a(\alpha)}} : \frac{b(\alpha)}{{\sqrt{a(\alpha)}}} : \alpha_2 : \alpha_3 : \alpha_4 \right), 
	\quad \left(\frac{b(\alpha)}{\sqrt{c(\alpha)}} : \frac{c(\alpha)}{\sqrt{c(\alpha)}} : \alpha_2 : \alpha_3 : \alpha_4 \right)
	\]
is well-defined, since the intersection $C \cap Z(\det \mathcal{B}(\mathbf{y}, \cdot))$ must consist of eight reduced points. Since $b(\alpha)^2 - a(\alpha)c(\alpha) = 0$, both representatives are equivalent; the choices of square root differ by a diagonal automorphism of $\mathcal{A}$. By construction, the two points obtained by choosing a sign are points in the Cayley variety associated to $\mathcal{A}$. This construction is a considerably simplified version of the construction from \cite[Section~6.4]{BruinSertoz2018}.
	
\begin{theorem}
	Let $C$ be a plane conic and let $Y \rightarrow C$ be a geometrically hyperelliptic genus $3$ curve with a residual pair of divisors $\{D_1, D_2\}$. Let $\mathcal{B}$ be a tensor associated to this data, and let $\mathcal{A}$ be the tensor obtained from Construction~\ref{cons: tensor Recillas}. Let $X$ be the genus $4$ curve defined by $\mathcal{A}$. If $X$ is smooth, then $X$ is the genus $4$ curve obtained from $(Y, \{D_1, D_2\})$ by Recillas' construction.
\end{theorem}

\begin{proof}
	Let $\psi_3\: \mbp^2 \dashrightarrow \dual \mbp^3$ be the map defined by the second block of $\mathcal{A}$. Notice that the conic $C$ is the pullback under $\psi$ of the hyperplane dual to $(0:0:0:1) \in \mbp^3$. Letting $X_2$ be the determinantal quadric defined by $\mathcal{A}$, by Lemma~\ref{lem: genus 4: image of Cayley in P2} we have that the branch points of $Y$ are exactly the pullback under $\psi$ of $\dual X_2$. Thus, the result is given by Theorem~\cite[Theorem~5.1 and Remark~5.2]{BruinSertoz2018}.
\end{proof}

%% Section
\section{Genus 5 curves: quadritangents} \label{sec: genus 5 curves}
In this section, we discuss smooth genus $5$ curves whose canonical model is an intersection of $3$ quadrics of rank $3$ in $\mbp^4$. %The space $k^5 \otimes (\Sym_2 k^2)^{\oplus 3}$ naturally embeds in $k^5 \otimes \Sym_2 k^6$ via the diagonal. 
Let $\mathcal{A} \in k^5 \otimes (\Sym_2 k^2)^{\oplus 3}$. The determinantal hypersurface associated to each block of $\mathcal{A}$ is a quadric of rank $3$ in $\mbp^4$; we denote these by $X_1$, $X_2$, and $X_3$. We say that $\mathcal{A}$ is \emph{nondegenerate} if the intersection $X := X_1 \cap X_2 \cap X_3$ is a smooth curve of genus $5$. 
Each quadric $X_i$ gives rise to a vanishing even theta characteristic $\theta_i$ on $X$. The Cayley variety of $\mathcal{A}$ is the intersection of $5$ quadrics in $\mbp^5$. When $\mathcal{A}$ is nondegenerate, Theorem~\ref{thm: singularities of symmetroid intersections} shows that $\mathfrak{C}$ is a smooth complete intersection of dimension $0$. The $32$ geometric points of $\mathfrak{C}$ give rise to $\binom{32}{2} = 496$ secant lines of $\mathfrak{C}$.

\subsection{Moduli} \label{sec: sub: moduli of genus 5 curves}

The vector space $k^5 \otimes (\Sym_2 k^2 \oplus \Sym_2 k^2 \oplus \Sym_2 k^2)$ has a natural $\GL_5(k) \times \GL_2(k)^{\oplus 3}$ action. The nondegenerate locus of $k^5 \otimes (\Sym_2 k^2 \oplus \Sym_2 k^2 \oplus \Sym_2 k^2)$ is a $\GL_5(k) \times \GL_2(k)^{\oplus 3}$-invariant open subscheme, and thus we can speak of orbit classes being nondegenerate. As with genus $4$ curves, the orbit classes of tensors classify isomorphism classes of genus $5$ curves with extra structure.

\begin{theorem} \label{thm: genus 5 moduli}
	Let $B$ be a scheme over $\Spec \mbz[\frac{1}{2}]$. There is a canonical bijection between:
	\[
		\left \{ \begin{array}{c}
		\text{isomorphism classes of ordered quadruples} \\
		(X, \theta_1, \theta_2, \theta_3) \text{ where } X \rightarrow B \text{ is a smooth}\\
		\text{genus $5$ curve over $B$, and } \theta_1, \theta_2, \theta_3 \text{ are} \\
		\text{distinct vanishing even theta characteristics}  
		
	\end{array}	 \right \}
	\longleftrightarrow
	\left \{ \begin{array}{c}
		\text{nondegenerate orbit classes of} \\
		\mathcal{O}_B^5 \otimes (\Sym_2 \mathcal{O}_B^2 \oplus \Sym_2 \mathcal{O}_B^2 \oplus \Sym_2 \mathcal{O}_B^2) \\
		\text{under the action of } \GL_5(\mathcal{O}_B) \times \GL_2(\mathcal{O}_B)^{\oplus 3}
		\end{array}
	\right \}
	\]	
\end{theorem}

\begin{proof}
	%Given an element $\mathcal{A} \in \mathcal{O}_B^5 \otimes (\Sym_2 \mathcal{O}_B^2 \oplus \Sym_2 \mathcal{O}_B^2 \oplus \Sym_2 \mathcal{O}_B^2)$, we obtain a triple of quadrics $X_1$, $X_2$, $X_3$ over $B$ of rank $3$; set $X = X_1 \cap X_2 \cap X_3$. 
	%
	%For each $X_i$, there is a one-parameter family of maximal isotropic subspaces parameterized by $\dual X_i \iso \mbp^1$. Each maximal isotropic subspace defines an effective representative of the even theta characteristic $\theta_i$.
	%
	The method of proof is similar to Section~\ref{sec: genus 4 curves}. 
	As usual we let $f\: X \rightarrow B$ denote the structure morphism. 
	%Each vanishing even theta characteristic $\theta_{\lindex}$ determines a quadric $X_{\lindex}$ as well as an aCM sheaf $\theta_{\lindex}$ on this quadric restricting to $\theta_{\lindex}$ on $X$. 
	%
	With a choice of basis for $\R^0f_*(\theta_{\lindex})(B)$ and $\R^0f_*(\kappa_X)(B)$, the multiplication map $\theta_{\lindex} \times \theta_{\lindex} \rightarrow \kappa_X$ gives rise to an exact sequence of sheaves on $\mbp^4$
		\[
			\xym{
			0 \ar[r] & \mathcal{O}_{\mbp^4}(-1)^{\oplus 2} \ar[rr]^-{\tiny \mathcal{A}^{(\lindex)} := \begin{bmatrix} a & b \\ b & c \end{bmatrix}} & & \mathcal{O}_{\mbp^4}^{\oplus 2} \ar[r]^-{g} & \mathcal{E} \ar[r] & 0
			}
		\]
	which restricted to $X$ gives the complex
		\[
			\xym{
			(\kappa_X^\vee)^{\oplus 2} \ar[rr]^{\tiny \begin{bmatrix} s^2 & st \\ st & t^2 \end{bmatrix}} & & \mathcal{O}_{X}^{\oplus 2} \ar[r]^-{g} & \theta_l \ar[r] & 0
			}.
		\]	
	The tensor $\mathcal{A}^{(\lindex)} \in \mathcal{O}_B^5 \otimes \Sym_2 \mathcal{O}_B^2$ is read off from this resolution. Since $X$ is a smooth genus $5$ curve, we see that the tensor $\mathcal{A}$ obtained by diagonally joining $\mathcal{A}^{(1)}, \mathcal{A}^{(2)}, \mathcal{A}^{(3)}$ is nondegenerate, and furthermore the entries of $\mathcal{A}(\mathbf{x}, \cdot, \cdot)$ span $\R^0f_*\mathcal{O}_{|\kappa_X|^\vee}(1)(B)$.
	Conversely, given a nondegenerate block-diagonal tensor $\mathcal{A}(\mathbf{x}, \cdot, \cdot)$, the intersection of the $Z(\det \mathcal{A}^{(\lindex)} : 1 \leq \lindex \leq 3)$ is a canonical genus $5$ curve $X$. Furthermore, the aCM sheaves associated to each $\mathcal{A}^{(\lindex)}$ restrict to vanishing even theta characteristics $\theta_\lindex$ on $X$. The implicit choice of basis for the tensor determines bases for each $\R^0f_*(\theta_\lindex)(B)$.
\end{proof}

%%% SUBSECTION %%%
\subsection{Theta characteristics} \label{sec: sub: theta characteristics of genus 5 curves}

We can identify $\mbp^5$ with $\mbp(\H^0(X, \theta_1) \oplus \H^0(X, \theta_2) \oplus \H^0(X, \theta_3))$. The group $\DAut(\mathcal{A})$ is of order $8$ and is generated by $\sigma_1, \sigma_2, \sigma_3$, where each $\sigma_l$ acts on $\H^0(X, \theta_1) \oplus \H^0(X, \theta_2) \oplus \H^0(X, \theta_3)$ by $-1$ in the $\lindex$-th summand and identity in the other factors.
The action on the $32$ geometric points of $\mathfrak{C}$ has $8$ orbits, each of size $4$.
The action of $\DAut(\mathcal{A})$ on the $496$ secants of the Cayley variety has $136$ orbits, giving $136$ distinct images under the $\psi$ map. Proposition~\ref{prop: tangency for curves} allows us to tabulate these images:
\begin{enumerate}
	\item There are $24 \times 2$ secants of the form $\ell(p, \sigma_\lindex(p))$ for some $p\in \mathfrak{C}(\bar k)$ and $l \in \{1,2,3\}$.
	The image under $\psi$ of a secant of the form $\ell(p, \sigma_\lindex(p))$ gives rise to an effective representative of $\theta_\lindex$. For each $\theta_\lindex$, there are $8$ distinct orbits of secants giving rise to an effective representative of $\theta_\lindex$. Given any point $p \in \mathfrak{C}(\bar k)$, the orbit $\{p, \sigma_1(p), \sigma_2(p), \sigma_3(p)\}$ defines $6$ secants. Because $\sigma_1\sigma_2\sigma_3 = -I$, the pair $\{\ell(p, \sigma_1(p)), \ell(\sigma_2(p), \sigma_3(p))\}$ is an orbit of $\DAut(\mathcal{A})$, which gives rise to a single effective representative of $\theta_1$. Similarly, the other two orbits give rise to effective representatives of $\theta_2$ and $\theta_3$.
	
	\item The remaining $112 \times 4$ secants define $112$ distinct quadritangent planes to $X$. Given two orbits $\{\sigma(p) : \sigma \in \DAut(\mathcal{A})\}, \{\tau(q) : \tau \in \DAut(\mathcal{A})\}$ of points in $\mathfrak{C}$, there are $16$ secants of the form $\ell(\sigma(p), \tau(q))$, $\sigma, \tau \in \DAut(\mathcal{A})$ which are partitioned into $4$ orbits under $\DAut(\mathcal{A})$. The secants
	\[
		\ell(p,q), \ \ \ell(p, \sigma_1(q)), \ \ \ell(p, \sigma_2(q)), \ \ \ell(p, \sigma_3(q))
	\]
	are representatives of each of the four orbits.
	The odd theta characteristics defined by each orbit of secants are grouped into $28$ sets of $4$ by this combinatorial association.
\end{enumerate}

%% Subsection
\subsection{Weil pairing}

In the case of genus $5$ curves, the three distinguished vanishing even theta characteristics $\theta_1, \theta_2, \theta_3$ define a distinguished $2$-torsion subgroup $\gen{\theta_2 - \theta_1, \theta_3 - \theta_1} \subset J_X[2]$. Just like in the genus $4$ case, the $112$ odd theta characteristics obtained from this construction also have a special property with respect to the Weil pairing. %The following theorem about the Weil pairing on these genus $5$ curves.

\begin{proposition} \label{prop: weil pairing genus 5}
Let $\mathcal{A} \in k^5 \otimes (\Sym_2 k^2)^{\oplus 3}$ be a nondegenerate tensor, let $\theta_1, \theta_2, \theta_3$ denote the three associated vanishing even theta characteristics. Denote the Weil pairing on $\Jac(X)[2]$ by $e_2$. Then the $112$ distinct odd theta characteristics constructed from the secants of the Cayley variety of $\mathcal{A}$ are precisely the odd theta characteristics $\theta$ such that $e_2(\theta \otimes \theta_\lindex^{\vee}, T) = 0$ for all $T \in \{\theta_i \otimes \theta_j^\vee : 1 \leq i,j \leq 3\}$ and for some (equivalently, all) $1 \leq \lindex \leq 3$.
\end{proposition}

\begin{proof}
	Let $W := \{\theta_i \otimes \theta_j^\vee : 1 \leq i, j \leq 3\}$ denote the distinguished $2$-torsion subgroup.
	We give the argument to show that $e_2(\theta - \theta_1, \theta_2 - \theta_1) = 0$. Since $W$ is an isotropic subspace of $e_2$, it follows whenceforth that $e_2(\theta - \theta_1, T) = 0$ for all $T \in W$.
	First, we establish that $h^0(\theta + \theta_2 - \theta_1) = 1$. We consider the projections from $V := \H^0(\theta_1) \oplus \H^0(\theta_2) \oplus \H^0(\theta_3)$ onto the various factors and denote these by $\pi_i, \pi_{ij}$ for $i,j = 1, 2, 3$. If $\ell$ is a secant line of $\mathfrak{C}$ giving rise to an odd theta characteristic, then $\ell$ surjects onto each $|\theta_{\lindex}| \iso \mbp^1$ for $\lindex=1,2,3$, and additionally the projection restricted to $\ell$ is a morphism of degree $1$. In particular, it follows that the image of $\ell$ under $\pi_{12}\: \mbp V \dashrightarrow |\theta_1| \times |\theta_2|$ is a curve of bidegree $(1,1)$.
	
	Let $[H] = \psi(\ell)$ and let $\{x_1,x_2,x_3,x_4\} = X \cap H$. We see that each quadric $Q_{x_i}$ is rank~$3$, so $\sing(Q_{x_i})$ is a plane in $\mbp V$. Furthermore, each block of $\mathcal{A}$ has size $2$, so each $\pi_{\lindex}(\sing(Q_{x_i})) \subset |\theta_{\lindex}|$ is a point, so $\pi_{12}(\sing(Q_{x_i})) \subset |\theta_1| \times |\theta_2|$ is also a point. Moreover, as $\sing(Q_{x_i})$ meets $\ell$, the point $\pi_{12}(\sing(Q_{x_i}))$ lies on the $(1,1)$ curve from before.
	
	Next, we have that $\theta + \theta_2 - \theta_1 \sim \theta_1 + \theta_2 - \theta$. We observe that the image of $X$ under 
		\[
			(\varphi_1, \varphi_2)\: X \rightarrow |\theta_1| \times |\theta_2|
		\]
	must contain the $4$ points $\pi_{12}(\sing(Q_{x_i}))$, so $\pi_{12}(\ell)$ defines an effective section of $\theta_1 + \theta_2 - \theta$.

	Let $\mathcal{Q}_1$ denote the quadratic form associated to the even theta characteristic $\theta_1$.
	We now use Mumford's formula for the Weil pairing. Observe that if $T'$ is a non-trivial $2$-torsion class we have $h^0(X, T') = 0$, so 
	\begin{align*}
		e_2(\theta - \theta_1, \theta_2 - \theta_1) &\equiv \mathcal{Q}_1(\theta + \theta_2 - 2\theta_1) - \mathcal{Q}_1(\theta-\theta_1) - \mathcal{Q}_1(\theta_2 - \theta_1) \\
		&\equiv \mathcal{Q}_1(\theta + \theta_2 - 2\theta_1) - \left(h^0(\theta-\theta_1) + h^0(\theta) \right) - \left(h^0(\theta_2 - \theta_1) + h^0(\theta_2)\right) \\
		&\equiv \mathcal{Q}_1(\theta + \theta_2 - 2\theta_1) + 1 \\
		&\equiv h^0(\theta + \theta_2 - 2\theta_1) + h^0(\theta + \theta_2 - \theta_1) + 1 \pmod{2}.
	\end{align*}
	Next, since we have that $\mathcal{Q}_1(\theta - \theta_1) = 1$ and $\mathcal{Q}_1(\theta_2 - \theta_1) = 0$, the sum $\theta + \theta_2 - 2\theta_1$ is not the trivial class. Finally, we see that $h^0(\theta + \theta_2 - \theta_1) = 1$, and this concludes the proof.
\end{proof}

\begin{remark}
    Let $X$ be a general genus $5$ curve and let $\mathcal{Q}_1$ be the quadratic form on $\Jac(X)[2]$ defined by an even theta characteristic $\theta_1$. There is one subgroup $H \subset \Jac(X)[2]$ of order $4$ up to conjugacy such that $\mathcal{Q}_1(v) = 0$ for all $v \in H$. 
    For such an $H$, there are exactly $112$ odd theta characteristics $\theta$ such that $e_2(\theta - \theta_1, v) = 0$ for all $v \in H$; furthermore, each translate $\theta_1 + v$ is an even theta characteristic.
\end{remark}

\noindent
Theorem~\ref{thm: genusfivesummary} is obtained by combining Theorem~\ref{thm: genus 5 moduli}, Proposition~\ref{prop: weil pairing genus 5}, and the discussion of Section~\ref{sec: sub: theta characteristics of genus 5 curves}.

\genusfivesummary*

%%% SECTION
\section{Quintic symmetroid surfaces: tritangents} \label{sec: quintic symmetroid}

A general element $\mathcal{A} \in k^4 \otimes \Sym_2 k^5$ defines a determinantal surface $X$ in $\mathbb{P}^3$ classically called a quintic symmetroid. The Cayley variety associated to $\mathcal{A}$ is a complete intersection of $4$ quadrics in $\mbp^4$, so generically defines a locus of $16$ points and thus $\binom{16}{2} = 120$ secants. The singular locus of $X$ consists of $20$ nodal singularities and $\varphi(X)$ is a smooth degree $10$ surface in $\mbp^4$ birational to $X$; it is the normalization of $X$. 

It was shown in \cite{Roth1930} that a generic quintic symmetroid surface has at least $120$ tritangent planes. Roth constructs these $120$ tritangent planes as we do, via the secants of the $16$ points of the Cayley variety \cite[Section~4.2]{Roth1930}. In light of Proposition~\ref{prop: deg and dim of tangents of symmetric determinantal hypersurfaces}, it is immediate why these secants correspond to tritangent planes, simplifying Roth's proof. This particular set of $120$ tritangents is distinguished by the determinantal representation. We say $H$ is a \emph{distinguished} tritangent of $\mathcal{A}$ if the contact points $x_1, x_2, x_3$ are smooth points of $X$ and the three points $\varphi(x_1)$, $\varphi(x_2)$, $\varphi(x_3)$ are collinear.

\begin{corollary}
Let $\mathcal{A} \in k^4 \otimes \Sym_2 k^5$ be a generic tensor. Then
there are $120$ distinguished tritangent planes of $\mathcal{A}$; each of these is of the form $[H] = \psi(\ell)$, where $\ell$ is one of the $120$ secant lines of $\mathfrak{C}$.
\end{corollary}

\begin{proof}
	First consider $\ell$ a secant line of $\mathfrak{C}$. Let $x_1$, $x_2$, $x_3$ be the contact points of $H$ with $X$. This follows from the fact that the incidence variety
	\[
		\{(x, q) \in X \times \ell : \mathcal{A}(x, q, \cdot) = 0\}
	\]
	consists of $3$ points; c.f. the proof of Proposition~\ref{prop: deg and dim of tangents of symmetric determinantal hypersurfaces}. Thus, $\varphi(x_1)$, $\varphi(x_2)$, and $\varphi(x_3)$ all lie on $\ell$ and so $H$ is a distinguished tritangent. The tritangents defined by the secants of $\mathfrak{C}$ are distinct by Lemma~\ref{lem: psi is injective on secants mod diagonal symmetries}.
	
	Conversely, suppose $H$ is a distinguished tritangent plane with contact points $x_1$, $x_2$, $x_3$. Let $\ell$ be the line in $\mbp^4$ containing $\varphi(x_1), \varphi(x_2), \varphi(x_3)$. Since $\theta_X = \psi \circ \varphi$, we have that $\psi^{-1}([H])$ contains $\varphi(x_1), \varphi(x_2), \varphi(x_3)$. However, any quadric containing these three points contains $\ell$. Since $\psi^{-1}([H])$ is defined by a complete intersection of three quadrics, $\ell \subset \psi^{-1}([H])$. The only lines contracted by $\psi$ are secant lines of $\mathfrak{C}$.
\end{proof}

%%%%%%%%%%%%%%%%%%%%%%
% BIBLIOGRAPHY
%%%%%%%%%%%%%%%%%%%%%%

%% Kill the MR tag in the bibliography.	
\renewcommand{\MR}[1]{}


\begin{bibdiv}
	\begin{biblist}
		
		\bib{Beauville2000}{article}{
			author = {Beauville, Arnaud},
			title={Determinantal hypersurfaces},
			date = {2000},
			journal={Michigan Mathematical Journal},
			volume = {48},
			issue = {1},
			pages= {39-64},
			review= {\MR{1786479}},
		}

		\bib{BhargavaHCL1}{article}{
			author={Bhargava, M.},
			title={Higher composition laws I: A new view on Gauss composition, and quadratic generalizations},
			journal={Annals of Mathematics},
			pages={217--250},
			volume={159},
			date={2004},
			issue={1},
		}
		\bib{BhargavaHCL2}{article}{
			author={Bhargava, M.},
			title={Higher composition laws II: On cubic analogues of Gauss composition},
			journal={Annals of Mathematics},
			pages={865--886},
			volume={159},
			date={2004},
			issue={2},
		}
		\bib{BhargavaHCL3}{article}{
			author={Bhargava, M.},
			title={Higher composition laws III: The parametrization of quartic rings},
			journal={Annals of Mathematics},
			pages={1329--1360},
			volume={159},
			date={2004},
			issue={3},
		}
		\bib{BhargavaHCL4}{article}{
			author={Bhargava, M.},
			title={Higher composition laws IV: The parametrization of quintic rings},
			journal={Annals of Mathematics},
			pages={53--94},
			volume={167},
			date={2008},
			issue={1},
		}

		\bib{BhargavaGrossWang2017}{article}{
		   author={Bhargava, Manjul},
		   author={Gross, Benedict H.},
		   author={Wang, Xiaoheng},
		   title={A positive proportion of locally soluble hyperelliptic curves over
		   $\mathbb{Q}$ have no point over any odd degree extension},
		   note={With an appendix by Tim Dokchitser and Vladimir Dokchitser},
		   journal={J. Amer. Math. Soc.},
		   volume={30},
		   date={2017},
		   number={2},
		   pages={451--493},
		   issn={0894-0347},
		   review={\MR{3600041}},
		   doi={10.1090/jams/863},
		}

		\bib{BhargavaHo2016}{article}{
		   author={Bhargava, Manjul},
		   author={Ho, Wei},
		   title={Coregular spaces and genus one curves},
		   journal={Camb. J. Math.},
		   volume={4},
		   date={2016},
		   number={1},
		   pages={1--119},
		   issn={2168-0930},
		   review={\MR{3472915}},
		   doi={10.4310/CJM.2016.v4.n1.a1},
		}

		\bib{BhargavaHoKumar2016}{article}{
		   author={Bhargava, Manjul},
		   author={Ho, Wei},
		   author={Kumar, Abhinav},
		   title={Orbit parametrizations for K3 surfaces},
		   journal={Forum Math. Sigma},
		   volume={4},
		   date={2016},
		   pages={Paper No. e18, 86},
		   review={\MR{3519436}},
		   doi={10.1017/fms.2016.12},
		}
		
		\bib{BruinSertoz2018}{article}{
			author={Bruin, Nils},
			author={Sert\"{o}z, Emre Can},
			title={Prym varieties of genus four curves},
			journal={Trans. Amer. Math. Soc.},
			volume={373},
			date={2020},
			number={1},
			pages={149--183},
			issn={0002-9947},
			review={\MR{4042871}},
			doi={10.1090/tran/7902},
		}

		\bib{Cantor1987}{article}{
		   author={Cantor, David G.},
		   title={Computing in the Jacobian of a hyperelliptic curve},
		   journal={Math. Comp.},
		   volume={48},
		   date={1987},
		   number={177},
		   pages={95--101},
		   issn={0025-5718},
		   review={\MR{866101}},
		   doi={10.2307/2007876},
		}

		\bib{Catanese1981}{article}{
		   author={Catanese, F.},
		   title={On the rationality of certain moduli spaces related to curves of
		   genus $4$},
		   conference={
		      title={Algebraic geometry},
		      address={Ann Arbor, Mich.},
		      date={1981},
		   },
		   book={
		      series={Lecture Notes in Math.},
		      volume={1008},
		      publisher={Springer, Berlin},
		   },
		   date={1983},
		   pages={30--50},
		   review={\MR{723706}},
		   doi={10.1007/BFb0065697},
		}

		\bib{Coble1919}{article}{
		   author={Coble, Arthur B.},
		   title={The Ten Nodes of the Rational Sextic and of the Cayley Symmetroid},
		   journal={Amer. J. Math.},
		   volume={41},
		   date={1919},
		   number={4},
		   pages={243--265},
		   issn={0002-9327},
		   review={\MR{1506391}},
		   doi={10.2307/2370285},
		}
		
		\bib{Dixon}{article}{
			author={Dixon, A. C.},
			title={Note on the reduction of a ternary quartic to a symmetrical determinant},
			journal={Proc. Cambridge Phil. Soc},
			volume={11},
			data={1902},
			pages={350-351}
		}
			
		\bib{Dol2012}{book}{
			author={Dolgachev, Igor V.},
			title={Classical algebraic geometry},
			note={A modern view},
			publisher={Cambridge University Press, Cambridge},
			date={2012},
			pages={xii+639},
			isbn={978-1-107-01765-8},
			review={\MR{2964027}},
			doi={10.1017/CBO9781139084437},
		}	

		\bib{DolgachevOrtland1988}{article}{
		   author={Dolgachev, Igor},
		   author={Ortland, David},
		   title={Point sets in projective spaces and theta functions},
		   language={English, with French summary},
		   journal={Ast\'{e}risque},
		   number={165},
		   date={1988},
		   pages={210 pp. (1989)},
		   issn={0303-1179},
		   review={\MR{1007155}},
		}

		\bib{ElsenhansJahnel2019}{article}{
		   author={Elsenhans, Andreas-Stephan},
		   author={Jahnel, J\"{o}rg},
		   title={On plane quartics with a Galois invariant Cayley octad},
		   journal={Eur. J. Math.},
		   volume={5},
		   date={2019},
		   number={4},
		   pages={1156--1172},
		   issn={2199-675X},
		   review={\MR{4015450}},
		   doi={10.1007/s40879-018-0292-3},
		}

		\bib{GriffithsHarris1994}{book}{
		   author={Griffiths, Phillip},
		   author={Harris, Joseph},
		   title={Principles of algebraic geometry},
		   series={Wiley Classics Library},
		   note={Reprint of the 1978 original},
		   publisher={John Wiley \& Sons, Inc., New York},
		   date={1994},
		   pages={xiv+813},
		   isbn={0-471-05059-8},
		   review={\MR{1288523}},
		   doi={10.1002/9781118032527},
		}
		
		\bib{HesseI}{article}{
		   author={Hesse, Otto},
		   title={\"{U}ber die Doppeltangenten der Curven vierter Ordnung},
		   language={German},
		   journal={J. Reine Angew. Math.},
		   volume={49},
		   date={1855},
		   pages={279--332},
		   issn={0075-4102},
		   review={\MR{1578918}},
		   doi={10.1515/crll.1855.49.279},
		}
		
		\bib{HesseII}{article}{
		   author={Hesse, Otto},
		   title={\"{U}ber Determinanten und ihre Anwendung in der Geometrie,
		   insbesondere auf Curven vierter Ordnung},
		   language={German},
		   journal={J. Reine Angew. Math.},
		   volume={49},
		   date={1855},
		   pages={243--264},
		   issn={0075-4102},
		   review={\MR{1578915}},
		   doi={10.1515/crll.1855.49.243},
		}
		
		\bib{Ho2009}{book}{
		   author={Ho, Wei},
		   title={Orbit parametrizations of curves},
		   note={Thesis (Ph.D.)--Princeton University},
		   publisher={ProQuest LLC, Ann Arbor, MI},
		   date={2009},
		   pages={157},
		   isbn={978-1109-41055-6},
		   review={\MR{2713823}},
		}
		
		\bib{Kerner2012}{article}{
		   author={Kerner, Dmitry},
		   author={Vinnikov, Victor},
		   title={Determinantal representations of singular hypersurfaces in
		   $\mathbb{P}^n$},
		   journal={Adv. Math.},
		   volume={231},
		   date={2012},
		   number={3-4},
		   pages={1619--1654},
		   issn={0001-8708},
		   review={\MR{2964618}},
		   doi={10.1016/j.aim.2012.06.014},
		}

		\bib{Milne1923}{article}{
			author={Milne, W. P.},
			title={Sextactic Cones and Tritangent Planes of the Same System of a
				Quadri-Cubic Curve},
			journal={Proc. London Math. Soc. (2)},
			volume={21},
			date={1923},
			pages={373--380},
			issn={0024-6115},
			review={\MR{1575365}},
			doi={10.1112/plms/s2-21.1.373},
		}

		\bib{Plaumann2011}{article}{
			author={Plaumann, Daniel},
			author={Sturmfels, Bernd},
			author={Vinzant, Cynthia},
			title={Quartic curves and their bitangents},
			journal={J. Symbolic Comput.},
			volume={46},
			date={2011},
			number={6},
			pages={712--733},
			issn={0747-7171},
			review={\MR{2781949}},
			doi={10.1016/j.jsc.2011.01.007},
		}

		\bib{Recillas1974}{article}{
		   author={Recillas, Sevin},
		   title={Jacobians of curves with $g^{1}_{4}$'s are the Prym's of
		   trigonal curves},
		   journal={Bol. Soc. Mat. Mexicana (2)},
		   volume={19},
		   date={1974},
		   number={1},
		   pages={9--13},
		   review={\MR{480505}},
		}	

		\bib{Reid1972quadrics}{thesis}{
			author = {Reid, Miles}, 
			title = {The complete intersection of two or more quadrics}, 
			note = {Ph.D. dissertation},
			publisher = {Cambridge University},
			year = {1972}
		}
						
		\bib{Roth1930}{article}{
			author={Roth, L.},
			title={The tritangent planes of the quintic symmetroid},
			journal={Proceedings of the London Mathematical Society},
			volume={s2-30},
			date={1930},
			pages={425--432},
			issue={1},
			doi={https://doi.org/10.1112/plms/s2-30.1.425},		
		}
		
		\bib{SatoKimura1977}{article}{
			author={Sato, M.},
			author= {Kimura, T.},
			title={A classification of irreducible prehomogeneous vector spaces and their relative invariants},
			journal={Nagoya Mathematical Journal},
			volume={65},
			pages={1--155},
			date={1977},
		}

		\bib{Thiel2018}{article}{
		   author={Thiel, Ulrich},
		   title={Blocks in flat families of finite-dimensional algebras},
		   journal={Pacific J. Math.},
		   volume={295},
		   date={2018},
		   number={1},
		   pages={191--240},
		   issn={0030-8730},
		   review={\MR{3778331}},
		   doi={10.2140/pjm.2018.295.191},
		}
			
		\bib{Tyurin1975}{article}{
			author = {Tyurin, A. N.},
			title = {On intersections of quadrics},
			year = {1975},
			journal = {Russian Mathematical Surveys},
			doi = {http://dx.doi.org/10.1070/RM1975v030n06ABEH001530
			},
			volume = {30},
			issue = {6}
		}

		\bib{Wood2014}{article}{
		   author={Wood, Melanie Matchett},
		   title={Parametrization of ideal classes in rings associated to binary
		   forms},
		   journal={J. Reine Angew. Math.},
		   volume={689},
		   date={2014},
		   pages={169--199},
		   issn={0075-4102},
		   review={\MR{3187931}},
		   doi={10.1515/crelle-2012-0058},
		}

		\bib{WrightYukie1992}{article}{
			author={Wright, D. J.},
			author={Yukie, A.},
			title={Prehomogeneous vector spaces and field extensions},
			journal={Invent Math},
			volume={110},
			pages={283--314},
			date={1992},
		}

	\end{biblist}
\end{bibdiv}
\end{document}